\documentclass[11pt]{amsart}   

\usepackage{verbatim}
\usepackage{amssymb,amsthm,amsmath}
\usepackage{color}
\definecolor{darkblue}{rgb}{0, 0, .4}
\definecolor{grey}{rgb}{.7, .7, .7}
\usepackage[breaklinks]{hyperref}
\hypersetup{
    colorlinks=true,
    linkcolor=darkblue,
    anchorcolor=darkblue,
    citecolor=darkblue,
    pagecolor=darkblue,
    urlcolor=darkblue,
    pdftitle={},
    pdfauthor={}
}

\newtheorem{theorem}{Theorem}[section]
\newtheorem{lemma}[theorem]{Lemma}

\theoremstyle{definition}
\newtheorem{definition}[theorem]{Definition}
\newtheorem{example}[theorem]{Example}

\theoremstyle{remark}
\newtheorem{remark}[theorem]{Remark}

\numberwithin{equation}{section}

\theoremstyle{theorem}

\newtheorem{proposition}[theorem]{Proposition}
\newtheorem{conjecture}[theorem]{Conjecture}

\newtheorem{deflemma}[theorem]{Definition-Lemma}

\newcommand{\N}[0]{\mathbb{N}}
\newcommand{\Z}[0]{\mathbb{Z}}

\newcommand{\s}[0]{\sigma}

\newcommand{\bftau}[0]{{\boldsymbol{\tau}}}

\newcommand{\Gr}[0]{\operatorname*{Gr}}
\newcommand{\Span}[0]{\operatorname*{Span}}

\newcommand{\hs}{}  
\newcommand{\hd}{\diamond}  
\newcommand{\hz}{\circ}  
\newcommand{\hf}{\bullet}  
\newcommand{\hb}{{\color{grey} \bullet}}  
\newcommand{\ho}{{\ast}}  

\newcommand{\gn}{\bullet}  

\newcommand{\hp}{+}
\newcommand{\hm}{-}

\newcommand{\heap}{ \xymatrix @=-9pt @! } 
\def\SDSize{6}  
\def\SDSizeTANRIGHT{4}  
\def\SDSizeTANLEFT{2}  
\def\SDMidpt{3}  
\def\SDColor{blue}
\def\SDEColor{black}

\newcommand{\StringL}[1]{{\color{\SDColor}\xy (\SDSize, \SDSize)*{}; (0, 0)*{}; (0, \SDSize)*{}; **\crv{(\SDSizeTANLEFT,\SDMidpt)};  (\SDMidpt, \SDMidpt)*{{\color{\SDEColor}#1}}; \endxy}}
\newcommand{\StringR}[1]{{\color{\SDColor}\xy (0,0)*{}; (\SDSize, 0)*{}; (\SDSize, \SDSize)*{}; **\crv{(\SDSizeTANRIGHT,\SDMidpt)}; (\SDMidpt, \SDMidpt)*{{\color{\SDEColor}#1}}; \endxy }}
\newcommand{\StringLX}[1]{{\color{\SDColor}\xy (0, \SDSize)*{}; (\SDSize, 0)*{}; **\crv{(\SDMidpt,\SDMidpt)};  (\SDMidpt, \SDMidpt)*{{\color{\SDEColor}#1}}; \endxy }}
\newcommand{\StringRX}[1]{{\color{\SDColor}\xy (\SDSize, \SDSize)*{}; (0, 0)*{}; **\crv{(\SDMidpt,\SDMidpt)};  (\SDMidpt, \SDMidpt)*{{\color{\SDEColor}#1}}; \endxy }}
\newcommand{\StringLR}[1]{{\color{\SDColor}\xy (0, 0)*{}; (0, \SDSize)*{}; **\crv{(\SDSizeTANLEFT,\SDMidpt)};  (\SDSize, 0)*{}; (\SDSize, \SDSize)*{}; **\crv{( \SDSizeTANRIGHT,\SDMidpt)}; (\SDMidpt, \SDMidpt)*{{\color{\SDEColor}#1}}; \endxy }}
\newcommand{\StringLRX}[1]{{\color{\SDColor}\xy (0, \SDSize)*{}; (\SDSize, 0)*{}; **\crv{(\SDMidpt,\SDMidpt)};  (\SDSize, \SDSize)*{}; (0, 0)*{}; **\crv{(\SDMidpt,\SDMidpt)}; (\SDMidpt, \SDMidpt)*{{\color{\SDEColor}#1}}; \endxy }}

\usepackage[all, dvips, knot]{xy}
\usepackage[all, knot, dvips, color]{xypic}
\xyoption{arc}

\usepackage{graphicx}

\newcommand{\w}{\mathsf{w}}

\renewcommand{\v}{\mathsf{v}} 

\newcommand{\after}{\mathsf{rpred}}
\newcommand{\before}{\mathsf{lpred}}
\newcommand{\last}{\mathsf{last}}

\newcommand{\ldim}{\mathsf{ldim}}
\newcommand{\rdim}{\mathsf{rdim}}

\newcommand{\id}{\mathrm{id}}

\begin{document}

\title{Mask formulas for cograssmannian Kazhdan--Lusztig polynomials}

\begin{abstract}
We give two contructions of sets of masks on cograssmannian permutations that
can be used in Deodhar's formula for Kazhdan--Lusztig basis elements of the
Iwahori--Hecke algebra.  The constructions are respectively based on a formula
of Lascoux--Sch\"utzenberger and its geometric interpretation by Zelevinsky.
The first construction relies on a basis of the Hecke algebra constructed from
principal lower order ideals in Bruhat order and a translation of this basis
into sets of masks.  The second construction relies on an interpretation of
masks as cells of the Bott--Samelson resolution.  These constructions give
distinct answers to a question of Deodhar.
\end{abstract}

\author{Brant Jones}
\address{Department of Mathematics and Statistics, MSC 1911, James Madison University, Harrisonburg, VA 22807}
\email{\href{mailto:brant@math.jmu.edu}{\texttt{brant@math.jmu.edu}}}
\urladdr{\url{http://www.math.jmu.edu/\~brant/}}
\thanks{The first author received support from NSF grant DMS-0636297.}

\author{Alexander Woo}
\address{Department of Mathematics, Statistics and Computer Science, St. Olaf
  College, 1520 St. Olaf Avenue, Northfield, MN 55057}
\email{\href{mailto:woo@stolaf.edu}{\texttt{woo@stolaf.edu}}}

\keywords{}

\date{\today}

\maketitle


\bigskip
\section{Introduction}\label{s:background}

Kazhdan--Lusztig polynomials were introduced in \cite{k-l} as the entries of
the transition matrix for expanding the Kazhdan--Lusztig canonical basis of
the Hecke algebra in terms of the standard basis.  Shortly afterwards these
polynomials were shown to have important interpretations as Poincar\'e
polynomials for the intersection cohomology of Schubert varieties \cite{K-L2}
and as a $q$-analog for the multiplicities of Verma modules \cite{BeilBern,
BryKash}.  The first of these interpretations shows, in an entirely
non-constructive fashion, that Kazhdan--Lusztig polynomials for Weyl groups
have positive integer coefficients.  While general closed formulas
\cite{billera-brenti, brenti} are known, they are fairly complicated and
non-positive.  A manifestly positive combinatorial rule for Kazhdan--Lusztig
polynomials is still sought, as is a proof of their positivity beyond the case
of Weyl groups.

Deodhar \cite{d} proposed a framework to express Kazhdan--Lusztig polynomials
positively in terms of combinatorial objects known as {\it masks} that are
defined on reduced expressions.  He gave two properties on sets of masks,
called {\it boundedness} and {\it admissibility}, and showed that bounded
admissible sets of masks can be appropriately counted to give Kazhdan--Lusztig
polynomials.  An independent explicit construction of bounded admissible sets
of masks satisfying this characterization would then be a combinatorial proof
of nonnegativity for Kazhdan--Lusztig polynomials.  However, he was only able
to provide a recursive method, dependent upon {\it a priori} knowledge of
positivity, for constructing such a set.

At the end of \cite{d}, Deodhar mentions that it should be possible to
reconcile his framework with formulas that had already appeared in the
literature.  One such formula is that of Lascoux and
Sch\"utzenberger \cite{LS11} for cograssmannian elements of the symmetric group
$S_n$.  In this paper we give two constructions of sets of masks which realize
the Lascoux--Sch\"utzenberger formula.

The Lascoux--Sch\"utzenberger formula states that the Kazhdan--Lusztig
polynomial $P_{u,w}(q)$ is given by $q$-counting edge-labellings of a rooted
tree, where the tree depends only on $w$ and the precise set of valid
edge-labellings depends on $u$.  Our first construction is combinatorial and
proceeds from this formula.  This construction depends both on a basis of the
Hecke algebra, which we denote $\{B^\prime_w\}$, constructed from principal
lower order ideals in Bruhat order and on an interpretation of this basis in
terms of sets of masks.  The combinatorics of this construction takes place on
heaps~\cite{viennot}, first used in conjunction with Deodhar's framework by
Billey and Warrington~\cite{b-w}.  We build a mask for each edge-labelling by
cutting the tree as well as the heap associated to our permutation into
segments and directly constructing a portion of the mask on each segment.  The
details of this construction involve not only heaps and trees but also some
partition combinatorics.

Our second construction is based on a geometric interpretation of the
Lascoux--Sch\"utzenberger formula due to Zelevinsky \cite{zelevinsky}.  He
constructs some resolutions of singularities for Schubert varieties and shows
geometrically that Kazhdan--Lusztig polynomials can be calculated from a cell
decomposition (in the usual topological sense) of this resolution.  These
cells have a combinatorial indexing set with a nontrivial bijection to the
Lascoux--Sch\"utzenberger trees.  Our construction proceeds by constructing a
mask from the combinatorial data indexing each cell.  Behind our construction
is a bijective association between masks on a reduced decomposition and cells
on the Bott--Samelson resolution, a more general resolution of singularities
for Schubert varieties.  There is a map $\pi$ from the Bott--Samelson
resolution to the Zelevinsky resolution.  We call a set of masks {\bf
  geometric} (with respect to the Zelevinsky resolution) if, for each cell of
the Zelevinsky resolution, the mask set contains one cell in its pre-image
(under $\pi$) of the same dimension.  We give an algorithm for producing a
geometric mask set.

Our constructions show that the problem of finding a set of masks that
implements the Lascoux--Sch\"utzenberger formula is rather delicate, with
choices which, while highly constrained, are numerous.  Both of our
constructions admit families of variations which produce different bounded
admissible sets of masks.  In addition, we give an example showing that our two
constructions fail to conincide even after allowing any choice in their
respective families of variations.  These differences show that the set of
masks that represent the terms of a particular cograssmannian Kazhdan--Lusztig
basis element is far from canonical.  Furthermore, our constructions are highly
algorithmic in nature; only in a few limited cases does it seem possible to
derive an explicit characterization of all the masks in our bounded admissible
set, or even a characterization of the masks which evaluate to a particular
element, for example the identity.  However, as Kazhdan--Lusztig polynomials
behave in mysterious and unexpected ways (as shown in \cite{MW02} for example),
such complexity and subtlety is not surprising.

Nevertheless, we believe there may be some important consequences and extensions
to our work.  As far as we know, our results provide the first example of a large
class of elements for which Deodhar's model does not produce a unique mask set.
One can interpret this as evidence that Deodhar's model is somehow incomplete in
the sense that additional axioms beyond boundedness and admissibility should be
included so as to produce mask sets that are canonical.

As part of our first construction, we introduce the $B'$ basis of the Hecke
algebra that is constructed from principal lower order ideals in Bruhat order.
This basis has the property that both nonnegativity and monotonicity of the
corresponding Kazhdan--Lusztig polynomials follow immediately in any case where
the Kazhdan--Lusztig basis element $C_{w}'$ can be expressed as a positive
combination of the $B_{x}'$.  Although this positivity is not true in general,
it does hold for the Kazhdan--Lusztig basis elements associated to
cograssmannian permutations.  In light of this, it would be interesting to
characterize the elements $w$ such that $C_w'$ expands nonnegatively into the
$B'$ basis.

The Lascoux--Sch\"utzenberger formula has been generalized to covexillary
permutations by Lascoux \cite{Lascoux95}.  However, the Zelevinsky resolution
has not been generalized to this case, though recent work of Li and
Yong~\cite{LiYong10} suggest some geometric explanation for Lascoux's formula.
It would be natural to try to extend either of our constructions to all
covexillary permutations, though the failure of the Zelevinsky resolution to
properly generalize to the covexillary case suggests that our first
construction or some other non-geometric set of masks may be better for this
purpose.  Work of Cortez~\cite{Cortez} suggests that the covexillary case may
be an important base case to the problem of finding an explicit nonnegative
formula for the Kazhdan--Lusztig polynomials in type $A$.  In addition, the
Lascoux--Sch\"utzenberger formula has been generalized to (co)minuscule
elements in other types \cite{boe}, so it would be natural to generalize our
work to this setting.  The Zelevinsky resolution has also been generalized in
that case \cite{perrin, SanVan}.  Finally, the Lascoux--Sch\"utzenberger
formula has also been recently studied by Shegechi and
Zinn-Justin~\cite{SZJ10}, and it would be interesting to explicitly compare
our constructions to the formulas found there.

In addition, it is an open question whether the map $\pi$ from the
Bott--Samelson to the Zelevinsky resolution always restricts to a bijection
over the cells we choose, since it might be the case that a cell in the
Zelevinsky resolution has the same dimension as the cell we choose in its
pre-image for more subtle reasons.

We now describe the contents of this paper in further detail.
Sections~\ref{s:bases} and~\ref{s:fwp} present background material and introduce
the tools we need for our first construction.  In Section~\ref{s:bases} we
introduce the $B'$ basis of the Hecke algebra that is constructed from principal
lower order ideals in Bruhat order.  In Section~\ref{s:fwp}, we show how to
decompose the set of masks associated to a particular reduced expression into
subsets that are compatible with the $B'$ basis.

Our first construction is given in Section~\ref{s:cog}.  We begin by recalling
the combinatorial formula of Lascoux and Sch\"utzenberger~\cite{LS11} for
Kazhdan--Lusztig polynomials associated to cograssmannian permutations as a
sum of powers of $q$ over certain edge-labelled trees.  The first main result
of this section is Theorem~\ref{t:cog_bp}, which states that each coefficient
of our ideal basis $B_x'$ in the expansion of the Kazhdan--Lusztig basis
element $C_w'$ corresponds precisely to one of Lascoux--Sch\"utzenberger's
edge labelled trees.  The second main result is Theorem~\ref{t:main}, which
gives our first construction of a set of masks that realizes the formula of
Lascoux--Sch\"utzenberger in Deodhar's framework.

Before giving our second construction, we explain the connection between
Bott--Samelson resolutions~\cite{BS58} and Deodhar's theorem in
Section~\ref{s:b-s}.  In particular, we describe the natural bijection between
cells of the Bott--Samelson resolution and masks on $\w$.  Although such a
connection was pointed out in Deodhar's paper~\cite{d} by the anonymous referee
and is likely well-known to experts, the details of this connection have never
before, as far as we can tell, appeared in print.

In Section~\ref{s:zel}, we give our second construction.  First we describe
the resolutions of Zelevinsky~\cite{zelevinsky} and a map from Bott--Samelson
resolutions to Zelevinsky resolutions.  The main result in this section is
Theorem~\ref{t:main2}, which gives our second construction realizing the
Lascoux--Sch\"utzenberger formula in Deodhar's framework.  At the end of this
section, we give an example showing our two constructions do not coincide.
More precisely, we show that our the construction of Section~\ref{s:cog} can
produce mask sets which are not geometric.

\bigskip
\section{Bases for the Hecke algebra}\label{s:bases}

\subsection{Coxeter groups}
Let $W$ be a Coxeter group with a generating set $S$ of involutions and
relations of the form $(s_{i}s_{j})^{m(i,j)} = 1$.  The {\bf Coxeter graph}
for $W$ is the graph on the generating set $S$ with edges connecting $s_{i}$
and $s_{j}$ labelled $m(i,j)$ for all pairs $i,j$ with $m(i,j)>2$.  Note that
if $m(i,j)=3$ it is customary to leave the corresponding edge unlabelled.
Also, if $m(i,j)=2$ then the relations imply that $s_i$ and $s_j$ commute.

We view the symmetric group $S_n$ as a Coxeter group of type $A$ with generators
$S = \{ s_1, \dots, s_{n-1} \}$ and relations of the form $(s_i s_{i \pm 1})^3 =
1$ together with $(s_{i}s_{j})^{2} = 1$ for $| i - j | \geq 2$ and $s_i^2 = 1$.  The Coxeter
graph of type $A_{n-1}$ is a path with $n-1$ vertices.

We may also refer to elements in the symmetric group by the {\bf 1-line
notation} $w=[w_{1}w_{2} \dots w_{n}]$ where $w$ is the bijection mapping $i$
to $w_{i}$.  Then the generators $s_{i}$ are the adjacent transpositions
interchanging the entries $i$ and $i+1$ in the 1-line notation.

An {\bf expression} is any product of generators from $S$ and the {\bf length}
$\ell(w)$ is the minimum length of any expression for the element $w$.  Such a
minimum length expression is called {\bf reduced}.  Each element $w \in W$ can
have several different reduced expressions representing it; for example, the
reduced expressions for $[3412]$ are $\{ s_2 s_3 s_1 s_2, s_2 s_1 s_3 s_2 \}$.
Given $w \in W$, we represent reduced expressions for $w$ in sans serif font,
say $\w=\w_{1} \w_{2}\cdots \w_{p}$ where each $\w_{i} \in S$.  We call any
expression of the form $s_i s_{i \pm 1} s_i$ a {\bf short-braid} after A.
Zelevinski (see \cite{fan1}).  We say that $x \leq w$ in {\bf Bruhat order} if a
reduced expression for $x$ appears as a subword (that is not necessarily
consecutive) of some reduced expression for $w$.  If $s_i$ appears as the last
(first, respectively) factor in some reduced expression for $w$, then we say
that $s_i$ is a {\bf right (left, respectively) descent} for $w$; otherwise,
$s_i$ is a {\bf right (left, respectively) ascent} for $w$.

The following lemma gives a useful property of Bruhat order.

\begin{lemma}{\bf (Lifting Lemma) \cite[Proposition 2.2.7]{b-b}}\label{l:lifting}
Suppose $x < w$, $s_i$ is a right descent for $w$, and $s_i$ is a right ascent for $x$.
Then, $x s_i \leq w$ and $w s_i \geq x$.
\end{lemma}

It is a theorem of Matsumoto \cite{matsumoto} and Tits \cite{t} that every
reduced expression for an element $w$ of a Coxeter group can be obtained from
any other by applying a sequence of braid moves of the form 
\[
{\underbrace{s_i s_j s_i s_j \cdots }_{m(i,j)} } \mapsto
{\underbrace{s_j s_i s_j s_i \cdots}_{m(i,j)}}
\]
where $s_i$ and $s_j$ are generators in $S$ that appear in the reduced
expression for $w$, and each factor in the move has $m(i,j)$ letters.

As in \cite{s1}, we define an equivalence relation on the set of reduced
expressions for a permutation by saying that two reduced expressions are in
the same {\bf commutativity class} if one can be obtained from the other by a
sequence of {\bf commuting moves} of the form $s_i s_j \mapsto s_j s_i$ where
$|i-j| \geq 2$.  If the reduced expressions for a permutation $w$ form a
single commutativity class, then we say $w$ is {\bf fully commutative}.  A
permutation $w$ is {\bf short-braid avoiding} if none of its reduced words
contain a short-braid; in $S_n$ fully commutative permutations are short-braid
avoiding.

\subsection{Heaps}

If $\w = \w_1 \cdots \w_k$ is a reduced expression, then following \cite{s1} we
define a partial ordering on the indices $\{1, \cdots, k\}$ by the transitive
closure of the relation $i \prec j$ if $i > j$ and $\w_i$ does not commute with
$\w_j$.  We label each element $i$ of the poset by the corresponding generator
$\w_i$.  It follows from the definition that if $\w$ and $\w'$ are two reduced
expressions for a permutation $w$ that are in the same commutativity class, then
the labelled posets of $\w$ and $\w'$ are isomorphic.  This isomorphism class of
labelled posets is called the {\bf heap} of $\w$, where $\w$ is a reduced
expression representative for a commutativity class of $w$.  In particular, if
$w$ is fully commutative then it has a single commutativity class, and so there
is a unique heap of $w$.

As in \cite{b-w}, we will represent a heap as a set of lattice points embedded
in $\N^2$.  To do this, we assign coordinates $(x,y) \in \N^2$ to each entry of
the labelled Hasse diagram for the heap of $\w$ in such a way that:
\begin{enumerate}
\item[(1)] An entry represented by $(x,y)$ is labelled $s_i$ in the heap if and
only if $x = i$, and
\item[(2)] If an entry represented by $(x,y)$ is greater than an entry
represented by $(x',y')$ in the heap, then $y > y'$.
\end{enumerate}
Since the Coxeter graph of type $A$ is a path, it follows from the definition
that $(x,y)$ covers $(x',y')$ in the heap if and only if $x = x' \pm 1$, $y >
y'$, and there are no entries $(x'', y'')$ such that $x'' \in \{x, x'\}$ and $y'
< y'' < y$.  Hence, we can completely reconstruct the edges of the Hasse
diagram and the corresponding heap poset from a lattice point representation.
This representation will enable us to make arguments ``by picture'' that would
otherwise be difficult to formulate.  Although there are many coordinate
assignments for any particular heap, the $x$ coordinates of each entry are
fixed for all of them, and the coordinate assignments of any two entries only
differ in the amount of vertical space between them.

Note that in contrast to conventions used in other work, we are drawing the heap
so that the left side of the reduced expression occurs at the top of the picture
and the right side occurs at the bottom.

\begin{example}\label{ex:heap}
One lattice point representation of the heap of $w = s_2 s_3 s_1 s_2 s_4$ is
shown below, together with the labelled Hasse diagram for the unique heap poset
of $w$.

\smallskip
\begin{center}
\begin{tabular}{ll}
\xymatrix @=-4pt @! {
\hs & \hs & \hs & \hs & \hs \\
& \hf & \hs & \hs & \hs & \hs \\
\hf & \hs & \hf & \hs & \hs & \hs \\
& \hf & \hs & \hf & \hs & \hs \\
s_1 & s_2 & s_3 & s_4 \\
} &
\xymatrix @=0pt @! {
& \hf_{s_2} \ar@{-}[dl] \ar@{-}[dr] \\
\hf_{s_1} \ar@{-}[dr] & & \hf_{s_3} \ar@{-}[dl] \ar@{-}[dr] \\
 & \hf_{s_2} & & \hf_{s_4} \\
} \\
\end{tabular}
\end{center}
\end{example}

To describe the local structure of heaps of fully commutative permutations,
suppose that $x$ and $y$ are a pair of entries in the heap of $\w$ that
correspond to the same generator $s_i$, so that they lie in the same column $i$
of the heap.  Assume that $x$ and $y$ are a {\bf minimal pair} in the sense that
there is no other entry between them in column $i$.  If $\w$ is short-braid
avoiding, there must actually be two heap entries that lie strictly between $x$
and $y$ in the heap and do not commute with them.  In type $A$, these entries
must lie in distinct columns.  This property is called {\bf lateral convexity}
and is known to characterize those permutations that are fully commutative
\cite{b-w}.

In type $A$, the heap construction can be combined with another combinatorial
model for permutations in which the entries from the 1-line notation are
represented by strings.  The points at which two strings cross can be viewed as
adjacent transpositions of the 1-line notation.  Hence, we may overlay strings
on top of a heap diagram to recover the 1-line notation for the permutation by
drawing the strings from top to bottom so that they cross at each entry in the
heap where they meet and bounce at each lattice point not in the heap.
Conversely, each permutation string diagram corresponds with a heap by taking
all of the points where the strings cross as the entries of the heap.

For example, we can overlay strings on the two heaps of $[3214]$.  Note that the
labels in the picture below refer to the strings, not the generators.
\begin{center}
\begin{tabular}{ll}
    \heap{
    & 1 \ \ 2 &  & 3 \ \ 4 & \\
    \StringR{\hs} & \hs & \StringLRX{\hf} & \hs & \StringL{\hs} \\
    \hs & \StringLRX{\hf} & \hs & \StringLR{\hs} & \hs \\
    \StringR{\hs} & \hs & \StringLRX{\hf} & \hs & \StringL{\hs} \\
    & 3 \ \ 2 &  & 1 \ \ 4 & \\
    } &
    \heap{
    & 1 \ \ 2 &  & 3 \ \ 4 &  \\
    \hs & \StringLRX{\hf} & \hs & \StringLR{\hs} \\
    \StringR{\hs} & \hs & \StringLRX{\hf} & \hs & \StringL{\hs} \\
    \hs & \StringLRX{\hf} & \hs & \StringLR{\hs} \\
    & 3 \ \ 2 &  & 1 \ \ 4 &  \\
    } \\
\end{tabular}
\end{center}

For a more leisurely introduction to heaps and string diagrams, as well as
generalizations to Coxeter types $B$ and $D$, see \cite{billey-jones}.  Cartier and
Foata \cite{cartier-foata} were among the first to study heaps of dimers, which
were generalized to other settings by Viennot \cite{viennot}.  Stembridge has
studied enumerative aspects of heaps \cite{s1,s2} in the context of fully
commutative elements.  Green has also considered heaps of pieces with
applications to Coxeter groups in \cite{green1,green2,green3}.

\subsection{Hecke algebras}

Given any Coxeter group $W$, we can form the {\bf Hecke algebra}
$\mathcal{H}$ over the ring $\Z[q^{1/2}, q^{-1/2}]$ with basis $\{ T_w
: w \in W \}$ and relations:
\begin{align}\label{e:hecke.def}
T_s T_w = & T_{s w} \text{ for } \ell(s w) > \ell(w)  \\
T_s T_w = & (q - 1) T_w + q T_{s w} \text{ for } \ell(s w) < \ell(w) \label{e:sinv}
\end{align}
where $T_1$ corresponds to the identity element.  In particular, this implies
that \[ T_w = T_{\w_1} T_{\w_2} \dots T_{\w_p} \] whenever $\w_1 \w_2 \dots
\w_p$ is a reduced expression for $w$.  Also, it follows from (\ref{e:sinv})
that the basis elements $T_w$ are invertible.  Observe that when $q = 1$, the
Hecke algebra $\mathcal{H}$ becomes the group algebra of $W$.  

Let $\{C_w' : w \in W \}$ denote the basis of $\mathcal{H}$ defined by Kazhdan
and Lusztig \cite{k-l}.  This basis is invariant under the ring involution on
the Hecke algebra defined by $\overline{q^{1/2}} = q^{-1/2}$, $\overline{T_w}
= (T_{w^{-1}})^{-1}$; we denote this involution with a bar over the element.
The {\bf Kazhdan--Lusztig polynomials} $P_{x,w}(q)$ describe how to change
between the $T$ and $C'$ bases of $\mathcal{H}$:
\begin{equation*}
C_w' = q^{-\frac{1}{2} \ell(w)} \sum_{x \leq w} P_{x,w}(q) T_x. 
\end{equation*}
The $C_{w}'$ are defined uniquely to be the Hecke algebra elements
that are invariant under the bar involution and have expansion
coefficients as above, where $P_{x,w}$ is a polynomial in $q$ required to
satisfy
\[ \text{degree } P_{x,w}(q) \leq \frac{\ell(w)-\ell(x)-1}{2} \]
for all $x<w$ in Bruhat order and $P_{w,w}(q) = 1$ for all $w$.  We use the
notation $C_{w}'$ to be consistent with the literature because there is
already a related basis denoted $C_{w}$.

The following open conjecture is one of the motivations for our work.  This conjecture is
known to be true in a number of special cases including when the Coxeter group
is finite or affine~\cite{K-L2}.

\begin{conjecture}{\bf (Nonnegativity Conjecture) \cite{k-l}}\label{c:nn}
The coefficients of $P_{x,w}(q)$ are nonnegative in the Hecke algebra associated
to any Coxeter group.
\end{conjecture}

In addition, there is a related conjecture that implies nonnegativity and has
been proven for finite and affine Coxeter groups~\cite{irving,BM}.

\begin{conjecture}{\bf (Monotonicity Conjecture)}\label{c:mo}
If $x \leq y \leq w$ in a Coxeter group $W$, then $P_{x,w}(q) - P_{y,w}(q) \in
\N[q]$.
\end{conjecture}

We now consider another basis for the Hecke algebra $\mathcal{H}$.

\begin{proposition}
Let $w \in W$.  Define
\[ B_w' = q^{-\frac{1}{2} \ell(w)} \sum_{x \leq w} T_x. \] 
Then, $\{B_w' : w \in W \}$ is a linear basis of $\mathcal{H}$.
\end{proposition}
\begin{proof}
By M\"obius inversion together with the fact that Bruhat order is Eulerian,
proved independently by Verma \cite{verma} and Deodhar \cite{deodhar_mobius}, we
can recover the $T_w$ basis elements uniquely as 
\[ T_w = \sum_{x \leq w} (-1)^{\ell(w)-\ell(x)} q^{\frac{1}{2} \ell(x)} B_x'. \]
Hence, $\{B_w' : w \in W\}$ forms a basis of $\mathcal{H}$.
\end{proof}

Observe that when $C_{w}'$ can be expressed as a positive polynomial
combination of $B_{x}'$, we obtain both nonnegativity and monotonicity in the
sense of Conjectures~\ref{c:nn} and \ref{c:mo}.  Although $C_w'$ cannot be
expressed as a positive polynomial combination of $B_x'$ in general, $C_w'$
can be so expressed when $w$ is cograssmannian do have this property, as we
will see in Theorem~\ref{t:cog_bp}.

We say $w \in W$ is {\bf rationally smooth} if $P_{x,w}(q) = 1$ for all $x \leq
w$.  This terminology arises because, when $W$ is a finite Weyl group, it
follows from \cite{K-L2} that $w$ indexes a rationally smooth Schubert variety
precisely when $w$ is a rationally smooth element.  Observe that when $w$ is
rationally smooth, the basis element $B_w'$ is exactly equal to the
Kazhdan--Lusztig basis element $C_w'$.

\bigskip
\section{Deodhar's model and masks with prescribed defects}\label{s:fwp}

The main object of this work is to give formulas for Kazhdan--Lusztig
polynomials of cograssmannian permutations in terms of the combinatorial
model introduced by Deodhar \cite{d} and further developed by
Billey--Warrington \cite{b-w}.  We now proceed to describe this model.  Fix a
reduced expression $\w = \w_1 \w_2 \cdots \w_p$.  Define a {\bf mask} $\s$
associated to the reduced expression $\w$ to be any binary vector $(\s_1,
\ldots, \s_p)$ of length $p = \ell(w)$.  Every mask corresponds to a
subexpression of $\w$ defined by $\w^\s = \w_{1}^{\s_1} \cdots \w_{p}^{\s_p}$
where
\[
\w_{j}^{\s_j}  =
\begin{cases}
\w_{j}  &  \text{ if  }\s_j=1\\
\text{1}  &  \text{ if  }\s_j=0.
\end{cases}
\]
Each $\w^\s$ is a product of generators so it determines an element of $W$.  For
$1\leq j\leq p$, we also consider initial sequences of a mask denoted $\s[j] =
(\s_1, \ldots, \s_j)$, and the corresponding initial subexpression $\w^{\s[j]}
= \w_{1}^{\s_1} \cdots \w_{j}^{\s_j}$.  In particular, we have $\w^{\s[p]} =
\w^\s$.  We also use this notation to denote initial sequences of expressions,
so $\w[j] = \w_1 \cdots \w_j$.

We say that a position $j$ (for $2 \leq j \leq p$) of the fixed reduced
expression $\w$ is a {\bf defect} with respect to the mask $\s$ if
\begin{equation*}
\w^{\s[j-1]} \w_{j} < \w^{\s[j-1]}.
\end{equation*}
Note that the defect status of position $j$ does not depend on the value of
$\s_j$.  We say that a defect position is a {\bf zero-defect} if it has
mask-value 0, and call it a {\bf one-defect} if it has mask-value 1.  We call a
position that is not a defect a {\bf plain-zero} if it has mask-value 0, and we
call it a {\bf plain-one} if it has mask-value 1.

Let $d_{\w}(\s)$ denote the number of defects of $\w$ for the mask $\s$.
We will use the notation $d(\s) = d_{\w}(\s)$ if the reduced word $\w$ is fixed.
Deodhar's framework gives a combinatorial interpretation
for the Kazhdan--Lusztig polynomial $P_{x,w}(q)$ as the generating function for
masks $\s$ on a reduced expression $\w$ with respect to the defect statistic
$d(\s)$.  We begin by considering subsets of the set 
\[ \mathcal{S} = \{ 0, 1 \}^{\ell(w)}. \]
of all possible masks on $\w$.
For $\mathcal{E} \subset \mathcal{S}$, we define a prototype for $P_{x,w}(q)$:
\[ P_x(\mathcal{E}) = \sum_{ \substack{ \s \in \mathcal{E} \\ \w^{\s} = x } } q^{d(\s)} \]
and a corresponding prototype for the Kazhdan--Lusztig basis element $C_{w}'$:
\[ h(\mathcal{E}) = q^{-{1 \over 2} \ell(w)} \sum_{ \s \in \mathcal{E} } q^{d(\s)}
T_{\w^{\s}}. \]

\begin{definition}{\bf \cite{d}}\label{d:admissible}
Fix a reduced word $\w = \w_1 \w_2 \dots \w_k$.  We say that $\mathcal{E}
\subset \mathcal{S}$ is {\bf admissible} on $\w$ if:
\begin{enumerate}
\item $\mathcal{E}$ contains $\s = (1, 1, \dots, 1)$.
\item $\mathcal{E} = \tilde{\mathcal{E}}$ where $\tilde{\s} = (\s_1, \s_2, \dots, \s_{k-1}, 1 - \s_k)$.
\item $h(\mathcal{E}) = \overline{h(\mathcal{E})}$ is invariant under the bar involution on the Hecke algebra.
\end{enumerate}

We say that $\mathcal{E}$ is {\bf bounded} on $\w$ if $P_x(\mathcal{E})$ has
degree $\leq {1 \over 2} (\ell(w) - \ell(x) - 1)$ for all $x < w$ in Bruhat order.
\end{definition}

\begin{theorem}{\bf \cite{d}}\label{t:deodhar}
Let $x, w$ be elements in any Coxeter group $W$, and fix a reduced expression
$\w$ for $w$.  If $\mathcal{E} \subset \mathcal{S}$ is bounded and admissible
on $\w$, then 
\[ P_{x,w}(q) = P_x(\mathcal{E}) = \sum_{ \substack{ \s \in \mathcal{E} \\ \w^{\s} = x } } q^{d(\s)} \]
and hence
\[ C_{w}' = h(\mathcal{E}) = q^{-{1 \over 2} \ell(w)} \sum_{ \s \in \mathcal{E} } q^{d(\s)} T_{\w^{\s}}. \]
\end{theorem}

If a mask has no defect positions at all, then we say it is a {\bf constant
mask} on the reduced expression $\w$ for the element $\w^{\s}$.  This
terminology arises from the fact that these masks correspond precisely to the
unique constant term in the Kazhdan--Lusztig polynomial $P_{x,w}(q)$ in the
combinatorial model above.  Other authors \cite{marsh-rietsch,rietsch-williams}
have used the term ``positive distinguished subexpression'' to define an
equivalent notion.

\begin{definition}
Let $\w = \w_1 \w_2 \cdots \w_p$ be a fixed reduced expression for an element $w \in W$.
Suppose $P \subset \{2, 3, \ldots p\}$.  Define
\[ \mathcal{F}_{\w}^P = \{ \text{ masks $\s$ on $\w$ with defects precisely at the positions in
$P$ } \}. \]
\end{definition}

The following result generalizes \cite[Proposition 2.3(iii)]{d} which has been
used in the work \cite{marsh-rietsch} related to totally nonnegative flag
varieties, as well as \cite{armstrong} in the context of sorting algorithms on
Coxeter groups.  See \cite{jones-match} for a derivation of the M\"obius
function of Bruhat order based on a specialization of this result.

\begin{lemma}\label{l:fwp}
Let $\w$ be a reduced expression for an element $w \in W$.  Then, each $x \in
W$ occurs at most once in $\{ \w^{\s} : \s \in \mathcal{F}_{\w}^P \}$.  In fact, the set
of elements
\[ F_{\w}^P := \{ \w^{\s} : \s \in \mathcal{F}_{\w}^P \} \]
is a lower order ideal in the Bruhat order of $W$.  In other words, if $v \in
F_{\w}^P$ and $u \leq v$, then $u \in F_{\w}^P$.
\end{lemma}
\begin{proof}
The constraint that $\s$ have defects precisely at the positions in $P$ forces
there to be at most one mask $\s$ on $\w$ for $x$.  We describe an algorithm to
construct such a mask.

Let $r_{p+1}(x) = x$ and $i = p$.  We inductively assign
\[ \s_i := \begin{cases}
    0 &  \text{ if ($i \notin P$ and $\w_i$ is a right ascent for $r_{i+1}(x)$) or } \\
      &  \text{ ($i \in P$ and $\w_i$ is a right descent for $r_{i+1}(x)$)} \\
    1 &  \text{ otherwise } \\
\end{cases} \ \ 
\]
\[
\text{ and } \ \ \ \ r_i(x) := \begin{cases}
    r_{i+1}(x) &  \text{ if ($i \notin P$ and $\w_i$ is a right ascent for $r_{i+1}(x)$) } \\
               &  \text{ or ($i \in P$ and $\w_i$ is a right descent for $r_{i+1}(x)$)} \\
    r_{i+1}(x) \cdot \w_i &  \text{ otherwise } \\
\end{cases} 
\]
for each $i$ from $p$ down to $1$.  An inductive argument on the length of $\w$
shows that the assignments given above are the only ones that can produce a
mask for $x$ in $\mathcal{F}_{\w}^P$.  Hence, there is at most one mask for $x$
in $\mathcal{F}_{\w}^P$.  Observe that the algorithm succeeds if and only if
$r_1(x)$ is the identity.

Suppose the algorithm succeeds in constructing a mask in $\mathcal{F}_{\w}^P$
for $y \leq w$.  If $x < y$ and we run the algorithm for both elements
simultaneously, we initially have $r_{p+1}(x) = x \leq y = r_{p+1}(y)$.
Observe that for each $i \leq p$, if we have $r_{i+1}(x) \leq r_{i+1}(y)$, then
the algorithm considers right multiplying these elements by the same $s_i$ and
whether $i \in P$ or not is the same for both elements $r_{i+1}(x)$ and
$r_{i+1}(y)$.  Therefore, by the Lifting Lemma~\ref{l:lifting} we have
$r_{i}(x) \leq r_{i}(y)$.  Since $r_1(y) = 1$, this implies by induction that
$r_1(x) = 1$ so the algorithm succeeds for all $x<y$.  Hence, $F_{\w}^P$ is a
lower ideal in Bruhat order.
\end{proof}

For example, if $P = \emptyset$ then $\mathcal{F}_{\w}^P$ corresponds to the set
of masks on $\w$ with no defects at all.  In this case, $F_{\w}^{\emptyset}$ is
the Bruhat interval $[1, w]$.  Since the constant term of every Kazhdan--Lusztig
polynomial $P_{x,w}(q)$ is 1 when $x \leq w$, we see that the masks in
$\mathcal{F}_{\w}^{\emptyset}$ correspond precisely to the constant terms of
$P_{x,w}(q)$.

\begin{example}
If $W = A_4$, $\w = s_2 s_1 s_3 s_2 s_3$, and $P = \{5\}$, then $F_{\w}^P$ has maximal elements $s_2 s_3 s_2$ and $s_2 s_1 s_3$.
The mask for $s_2 s_1$ is
\[
\begin{tabular}{cccccc}
    $s_2$ & $s_1$ & $s_3$ & $s_2$ & $s_3$ \\
    1 & 1 & 1 & 0 & 1 \\
\end{tabular}
\]
as a result of 
\[ r_6(x) = s_2 s_1, r_5(x) = s_2 s_1 s_3 = r_4(x), r_3(x) = s_2 s_1, r_2(x) = s_2, r_1(x) = 1. \]
\end{example}

\bigskip
\section{The cograssmannian construction}\label{s:cog}

\subsection{The Lascoux--Sch\"utzenberger formula}
We now restrict to the case where $w$ is a permutation with at most one right
ascent.  We call such permutations {\bf cograssmannian}.  Following
\cite[Section 6]{Brenti98}, we describe a formula for $C_{w}'$ when $w$ is
cograssmannian that is originally due to Lascoux and Sch\"utzenberger
\cite{LS11}.  Lascoux \cite{Lascoux95} has generalized this formula
to the case where $w$ is covexillary; a permutation $w$ is {\bf covexillary} if
there do not exist indices $i_1<i_2<i_3<i_4$ for which
$w(i_3)<w(i_4)<w(i_1)<w(i_2)$.

Fix a cograssmannian permutation $w$, and let $s_z$ be the unique right ascent
of $w$.  Then, let $J = S \setminus \{ s_z \}$ and $W_J$ be the corresponding
parabolic subgroup of $W$.  By the parabolic decomposition (see
\cite[Proposition 2.4]{b-b}, for example), there is a unique reduced
decomposition $w = v w_0^J$ where $v$ is a minimal length coset representative
in $W / W_J$ and $w_0^J$ is the unique element of maximal length inside the
parabolic subgroup $W_J$.  Here, $v$ has a unique right descent, so we say that
it is {\bf grassmannian}.  Also, this implies that $v$ is fully-commutative.

Since $v$ has a unique commutativity class, we can assume that the heap of $w$
has a prescribed form in which $w_0^J$ is the product of 
\[ (s_1 s_2 \cdots s_{z-1}) (s_1 s_2 \cdots s_{z-2}) \cdots \text{ and } (s_{n-1} s_{n-2} \cdots s_{z+1}) (s_{n-1}
s_{n-2} \cdots s_{z+2}) \cdots. \]
We now fix a reduced expression $\w$ belonging to this commutativity class.
Figure~\ref{f:co_heap} illustrates a typical cograssmannian heap.  Recall that
in contrast to conventions used in other work, we draw the heap so that the
left side of the reduced expression occurs at the top of the picture and the
right side occurs at the bottom.  This allows the partition associated to $v$
to appear inside the heap in the ``Russian'' style.  
Here,
\[ v = s_1 s_5 s_7 s_2 s_4 s_6 s_3 s_5 s_4 \] 
\[ w_0^J = s_1 s_2 s_3 s_1 s_2 s_1 s_7 s_6 s_5 s_7 s_6 s_7 \]
and $w = v w_0^J$.

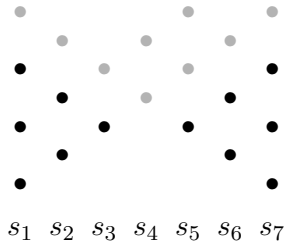
\begin{figure}[h]
\begin{center}
\begin{tabular}{c}
\xymatrix @-1.5pc@M=0pt {
{\hs} & & {\hs} & & {\hs} & & {\hs} & & {\hs} & & {\hs} \\
& {\hs} & & {\hs} & & {\hs} & & {\hs} & & {\hs} & & {\hs} \\
{\hs} & & {\hs} & & {\hs} & & {\hs} & & {\hs} & & {\hs} \\
& {\hs} & & {\hs} & & {\hs} & & {\hs} & & {\hs} & & {\hs} \\
{\hs} & & {\hs} & & {\hs} & & {\hs} & & {\hs} & & {\hs} \\
& {\hs} & & {\hs} & & {\hs} & & {\hs} & & {\hs} & & {\hs} \\
{\hb} & & {\hs} & & {\hb} & & {\hb} & & {\hs} & & {\hs} \\
& {\hb} & & {\hb} & & {\hb} & & {\hs} & & {\hs} & & {\hs} \\
{\hf} & & {\hb} & & {\hb} & & {\hf} & & {\hs} & & {\hs} \\
& {\hf} & & {\hb} & & {\hf} & & {\hs} & & {\hs} & & {\hs} \\
{\hf} & & {\hf} & & {\hf} & & {\hf} & & {\hs} & & {\hs} \\
& {\hf} & & {\hs} & & {\hf} & & {\hs} & & {\hs} & & {\hs} \\
{\hf} & & {\hs} & & {\hs} & & {\hf} & & {\hs} & & {\hs} \\
 & & & \\
\tiny s_1 & \tiny s_2 & \tiny s_3 & \tiny s_4 & \tiny s_5 & \tiny s_6 & \tiny s_7 \\
}
\end{tabular}
\end{center}
\caption{Heap of a cograssmannian element}\label{f:co_heap}
\end{figure}

We consider the {\bf ridgeline} of the heap of $w$ to be the lattice path
formed from the maximal entry of the heap of $v$ in each column.  These entries
naturally form a lattice path in which the entry in column $i+1$ is either just
above or just below the entry in column $i$.  Record this lattice path as a
string of parentheses where ``('' corresponds to a down-move, and ``)'' corresponds
to an up-move, reading left to right along the ridgeline.
In the example shown in Figure~\ref{f:co_heap}, the parentheses would be
\[ ( \ ( \ ) \ ) \ ( \ ). \]

Next, we form a rooted tree $T(w)$ by matching these parentheses.  Each vertex
of the tree corresponds to a matching pair ``( \dots )'' of parentheses, and
one vertex is a descendent of another if and only if its pair of parentheses
is enclosed by the other pair.  Each consecutive matching pair is called a
{\bf valley}, and the valleys are the leaves of the tree.  In addition, we add
a single additional root node that is attached to all maximal elements of the
tree.  The result is denoted $T(w)$.

Next, we define certain nonnegative integers called {\bf capacities} that are
assigned to the leaves of the tree $T(w)$.  Each leaf corresponds to some
valley say in column $j$ of the ridgeline, and the capacity of the valley is
defined to be the number of entries in the heap of $v$ in column
$j$.  In other words, the capacity is the number of levels of the heap
between the valley and the entries of $w_0^J$.

In the running example, $T(w)$ is
\[ \xymatrix @-1pc {
\gn \ar@{-}[d] \ar@{-}[dr] & \\
\gn \ar@{-}[d] & \gn_1 \\
\gn_1 & } \]
where we have indicated the capacities of the leaf nodes.

Finally, we define $A_{w}$ to be the set of edge-labelings of $T(w)$ with
entries of $\mathbb{Z}_{\geq 0}$ such that:
\begin{enumerate}
    \item[(1)]  Labels weakly increase along all paths from the root to any
      leaf, and
    \item[(2)]  No edge that is adjacent to a leaf node has a label that
        strictly exceeds the capacity of the leaf.
\end{enumerate}

Let $t \in A_{w}$ be an edge-labelled tree and denote the sum of the edge
labels by $|t|$.  We associate a permutation $x \leq w$ to $t$.  Begin with
the heap of $\w = v w_0^J$ as described above.  Consider each valley column
$j$.  If the corresponding leaf edge is labelled by $m \in \mathbb{Z}_{\geq
  0}$, then set the top $m$ entries in column $j$ to have mask-value 0, and
also set all of the entries that lie above these entries in the heap
to have mask-value 0.  Once this has been done for each valley, we are left
with a constant mask $\gamma(t)$ on $\w$ that encodes a cograssmannian
element; we denote this element by $x(t)$, so $x(t) = \w^{\gamma(t)}$.
Note that it is possible for some leaves to implicitly zero out other leaves
above.  If $t$ and $t'$ have the same leaf edge labels, then $x(t) = x(t')$.

\begin{example}\label{e:el_trees}
The valid edge labelings of $T(w)$ are
\[
\xymatrix @-1pc { \gn \ar@{-}[d]^0 \ar@{-}[dr]^0 & \\ \gn \ar@{-}[d]^0 & \gn \\ \gn & }  \hspace{0.2in}
\xymatrix @-1pc { \gn \ar@{-}[d]^0 \ar@{-}[dr]^1 & \\ \gn \ar@{-}[d]^0 & \gn \\ \gn & }  \hspace{0.2in}
\xymatrix @-1pc { \gn \ar@{-}[d]^0 \ar@{-}[dr]^0 & \\ \gn \ar@{-}[d]^1 & \gn \\ \gn & }  \hspace{0.2in}
\xymatrix @-1pc { \gn \ar@{-}[d]^0 \ar@{-}[dr]^1 & \\ \gn \ar@{-}[d]^1 & \gn \\ \gn & }  \hspace{0.2in}
\xymatrix @-1pc { \gn \ar@{-}[d]^1 \ar@{-}[dr]^0 & \\ \gn \ar@{-}[d]^1 & \gn \\ \gn & }  \hspace{0.2in}
\xymatrix @-1pc { \gn \ar@{-}[d]^1 \ar@{-}[dr]^1 & \\ \gn \ar@{-}[d]^1 & \gn \\
\gn & }
\]

Let $t$ be the second of these edge-labelled trees.  Then $x(t)$ is obtained by
starting with the heap of $w$ and then ``zeroing out'' the entries above the
valley in column $s_6$ in the heap.  The constant mask $\gamma(t)$ that is
associated to $x(t)$ is shown in Figure~\ref{f:co_xt}.  Hence, $x(t)$ is
\[ x(t) = (s_1 s_2 s_4 s_3 s_5 s_4) (s_1 s_2 s_3)(s_1 s_2)(s_1)(s_7 s_6 s_5)(s_7 s_6)(s_7). \]

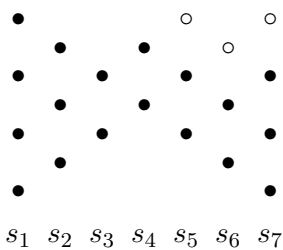
\begin{figure}[h]
\begin{center}
\begin{tabular}{c}
\xymatrix @-1.5pc@M=0pt {
{\hs} & & {\hs} & & {\hs} & & {\hs} & & {\hs} & & {\hs} \\
& {\hs} & & {\hs} & & {\hs} & & {\hs} & & {\hs} & & {\hs} \\
{\hs} & & {\hs} & & {\hs} & & {\hs} & & {\hs} & & {\hs} \\
& {\hs} & & {\hs} & & {\hs} & & {\hs} & & {\hs} & & {\hs} \\
{\hs} & & {\hs} & & {\hs} & & {\hs} & & {\hs} & & {\hs} \\
& {\hs} & & {\hs} & & {\hs} & & {\hs} & & {\hs} & & {\hs} \\
{\hf} & & {\hs} & & {\hz} & & {\hz} & & {\hs} & & {\hs} \\
& {\hf} & & {\hf} & & {\hz} & & {\hs} & & {\hs} & & {\hs} \\
{\hf} & & {\hf} & & {\hf} & & {\hf} & & {\hs} & & {\hs} \\
& {\hf} & & {\hf} & & {\hf} & & {\hs} & & {\hs} & & {\hs} \\
{\hf} & & {\hf} & & {\hf} & & {\hf} & & {\hs} & & {\hs} \\
& {\hf} & & {\hs} & & {\hf} & & {\hs} & & {\hs} & & {\hs} \\
{\hf} & & {\hs} & & {\hs} & & {\hf} & & {\hs} & & {\hs} \\
 & & & \\
\tiny s_1 & \tiny s_2 & \tiny s_3 & \tiny s_4 & \tiny s_5 & \tiny s_6 & \tiny s_7 \\
}
\end{tabular}
\end{center}
\caption{Constant mask $\gamma(t)$ for $x(t)$}\label{f:co_xt}
\end{figure}
\end{example}

We are now in a position to state our first main result that each coefficient of
$B_x'$ in $C_w'$ corresponds precisely to one of Lascoux--Sch\"utzenberger's
edge labelled trees.

\begin{theorem}\label{t:cog_bp}
Let $w$ be a cograssmannian permutation.  We have
\[ C_w' = q^{-\frac{1}{2} \ell(w)} \sum_{t \in A_{w}} q^{d_{\w}(t)} B_{x(t)}' \]
where $d_{\w}(t) = |t| + \frac{1}{2} \ell(x(t))$.
\end{theorem}
\begin{proof}
Let $s_z$ be the unique right ascent of $w$, and define $J = S \setminus \{ s_z
\}$ with the corresponding parabolic subgroup of $W$ denoted $W_J$.  Let $x
\leq w$.  Then there is a unique reduced parabolic decomposition of the form $x
= v u$ where $v \in W^J$, so $v$ is grassmannian with $s_z$ as the unique right
descent, and $u \in W_J$.  Let $\tilde{x}$ be $v w_0^J$ where $w_0^J$ is the
unique longest element of $W_J$.  Then $\tilde{x}$ is cograssmannian, and since
we can obtain $\tilde{x}$ from $x$ using right multiplication by elements of $J
= D_{R}(w)$, we have by \cite{k-l} that $P_{x, w}(q) = P_{\tilde{x},
w}(q)$.

Hence,
\[ C_{w}' = q^{-\frac{1}{2} \ell(w)} \sum_{x \leq w} P_{x,w}(q) T_{x} =
q^{-\frac{1}{2} \ell(w)} \sum_{x \leq w} P_{\tilde{x},w}(q) T_{x}. \]  Lascoux
and Sch\"utzenberger~\cite{LS11} showed that $P_{\tilde{x},w}(q) = \sum_{t}
q^{|t|}$ when $w$ and $\tilde{x}$ are cograssmannian.  (See also \cite[Theorem
  6.10]{Brenti98} for another description of this theorem.)  Here, the sum is
over the set $A_{\tilde{x},w}$ of all edge labelled trees from $A_w$ having
leaf capacities given by the number of mask-value 0 entries in a
given valley column of the constant mask for $\tilde{x}$ on $w = v w_0^J$.
Hence,
\[ C_{w}' = q^{-\frac{1}{2} \ell(w)} \sum_{x \leq w} \sum_{t \in A_{\tilde{x},w}}
q^{|t|} T_{x}. \]

Next, fix a tree $t \in A_w$, and let $I(t)$ be the set of elements $x$ that have
$t$ contributing to the coefficient of $T_x$ in the expansion of $C_{w}'$.  To
be explicit, $I(t)$ consists of all elements $x$ such that the grassmannian part
of the unique constant mask for $\tilde{x}$ has at least the number of
mask-value 0 entries in a given valley column as the corresponding leaf edge
label of $t$.  Then, we can rewrite the equation as
\[ C_{w}' = q^{-\frac{1}{2} \ell(w)} \sum_{t \in A_{w}} \sum_{x \in I(t)} q^{|t|} T_{x}. \]
We claim that $I(t)$ is a principal lower order ideal in Bruhat order on $W$,
with maximum element $x(t)$.

First, suppose $x \in I(t)$ and $y < x$.  Then we have $y^J \leq x^J$ by
\cite[Proposition 2.5.1]{b-b}, and this implies $\tilde{y} \leq \tilde{x}$.
Thus, the number of zeros in each leaf column of $\tilde{y}$ is greater than or
equal to the number of zeros in each leaf column of $\tilde{x}$.  Hence, the
leaf capacities for the trees in $A_{\tilde{y},w}$ are greater than or equal to
the leaf capacities for the trees in $A_{\tilde{x},w}$.  Therefore, $y \in
I(t)$, proving that $I(t)$ is a Bruhat lower order ideal.

Since $\tilde{x} \in I(t)$ whenever $x \in I(t)$, we observe that any Bruhat
maximal element of $I(t)$ must be cograssmannian.  Moreover, it follows from the
definition of $I(t)$ that to construct a Bruhat maximal element of $I(t)$, we
must precisely follow the procedure described to construct $x(t)$.  The
cograssmannian condition forces the mask value 0 entries while the Bruhat
maximal condition forces all the other entries to have mask value 1.  Hence,
$x(t)$ is the unique Bruhat maximal element in $I(t)$.  

Thus $\sum_{x \in I(t)} q^{|t|} T_{x} = q^{|t| + \frac{1}{2} \ell(x(t))}
B_{x(t)}'$, and the formula is proved.
\end{proof}

To conclude the running example, we have that $C_w'$ is given by 6 terms
\begin{align*}
C_w' = B_w' & + q^{-\frac{1}{2}} B_{s_6 s_7 s_5 w}' + q^{-\frac{3}{2}} B_{s_3
s_2 s_4 s_1 s_5 w}' + q^{-\frac{3}{2}} B_{s_3 s_2
s_4 s_6 s_1 s_5 s_7 w}' \\ & + q^{-\frac{1}{2}} B_{s_3 s_2 s_4 s_1 s_5 w}' +
q^{-\frac{1}{2}} B_{s_3 s_2
s_4 s_6 s_1 s_5 s_7 w}'
\end{align*}
corresponding precisely to the 6 edge labelled trees shown in
Example~\ref{e:el_trees}.  Hence, the expansion of $C_w'$ into $B'$ basis
elements has a combinatorial interpretation when $w$ is cograssmannian.  It
would be interesting to know whether there exist other classes of elements $w$
for which $C_w'$ always expands nonnegatively into the $B'$ basis.

\subsection{Masks for cograssmannian permutations}
Our next goal is to use the main result of Section~\ref{s:fwp} to give a set
of masks on $\w$ that encode $C_{w}'$.  We approach this by encoding each term
$q^{|t| + \frac{1}{2} \ell(x(t))} B_{x(t)}'$ from the formula of
Theorem~\ref{t:cog_bp} as $\sum_{\s \in \mathcal{F}_{\w}^{P(t)}} q^{|P(t)|}
T_{\w^{\s}}$ for some set of defect positions $P(t)$.  By Lemma~\ref{l:fwp},
it suffices to give a single mask $\s(t)$ with $|t|$ defects that encodes the
element $x(t)$, because the other elements of $B_{x(t)}'$ can all be encoded
by masks in $\mathcal{F}_{\w}^{P(t)}$ with defects in the same positions
$P(t)$ as $\s(t)$.  Note that there are generally several ways to construct an
appropriate $\s(t)$, and these different constructions may produce different
mask sets.

Preserving the notation of the previous section, let $T(w)$ be the unlabelled
tree constructed in the Lascoux--Sch\"utzenberger algorithm.  
Let $t \in A_w$ be one of the valid edge labelings of the tree $T(w)$.  Each
leaf of $t$ corresponds to a valley column in the heap of $\w = v w_0^J$.  Let
$\gamma(t)$ be the unique constant mask for $x(t)$ on $\w$.  Recall that this mask is
obtained by zeroing out entries in the heap of $\w$ starting from valley columns
as specified by the leaf labels of $t$.

{\bf Definition of valley statistics:}
Fix $\w$ and $t$.  Each valley in the ridgeline of the heap of $\w$ has several
parameters associated to it.  Let $v$ be a valley column of the heap, and let
$p(v)$ be the number of mask-value 0 entries in column $v$ in $\gamma(t)$.  We
call these mask-value 0 entries in the heap of $\gamma(t)$ {\bf valley
entries}.  We number these sequentially so that the lowest such entry in the
picture is the first valley entry, and the top valley entry is the $p(v)$-th
valley entry.  Define $q(v)$ to be the number of ``up steps'' in the ridgeline
lying between $v$ and the next peak in the ridgeline to the right of $v$.  We
say that the $i$th {\bf valley diagonal} consists of the entries extending to
the southeast in the heap from the $i$th valley entry.  If there exist
mask-value 1 entries below the $i$th entry of the $i$th valley diagonal in
$\gamma(t)$, then we say that the $i$th valley diagonal is {\bf not zeroed out}.
Otherwise, we say that the $i$th valley diagonal is {\bf zeroed out}.  These
zeroed out diagonals arise from mask-value 0 entries in a valley column further
to the right in the construction of $\gamma(t)$, so once a valley column is
zeroed out, all higher valley columns are also zeroed out.  Let $r(v)$ be the
number of valley diagonals that are not zeroed out.

{\bf Definition of segments and regions:}
We now define the {\bf segment} associated to $v$ to be the collection of
entries of the heap of $\w$ given as the following union of {\bf regions}.
Region $\mathfrak{1}$ is defined to be the top $p(v)$ entries in columns $v$
through $v+q(v)$.  Region $\mathfrak{2}$ is defined to be the entries of the
$r(v)$ valley diagonals that are not zeroed out and do not lie in region
$\mathfrak{1}$.  Region $\mathfrak{3}$ is defined to be those entries of
columns $v$ through $v+q(v)$ that lie on zeroed out valley diagonals and do not
lie in region $\mathfrak{1}$.

For example, Figure~\ref{f:heap_segments} illustrates how a particular heap
decomposes into segments.  The mask-values in Figure~\ref{f:heap_segments} come
from the constant mask $\gamma(t)$ that corresponds to a particular
edge-labelled tree $t$ (that we have not specified completely).  The labels $a, b,
\ldots, k$ correspond to the edge labels from the tree $t$ as shown in
Figure~\ref{f:el_trees}.  Figure~\ref{f:segment_regions} shows how one of the
segments in Figure~\ref{f:heap_segments} decomposes into regions.

\newcommand{\ab}{\ar@{-}'[dddddddddddddd]'[dddddddddddddddr]'[ur][]}
\newcommand{\ad}{\ar@{-}'[dddddddddddddddd]'[dddddddddddddddddr]'[ur][]}
\newcommand{\ac}{\ar@{-}'[dddddddddddddddddd]'[dddddddddddddddddddrdrdrdrdrdrdrdrdrdrdr]'[dddddddddddddddddrdrdrdrdrdrdrdrdrdrdrdr]'[dddddddddddddddddr]'[ur][]}
\newcommand{\af}{\ar@{-}'[dddddddddddddd]'[dddddddddddddddrdrdrdrdrdrdrdrdrdrdr]'[dddddddddddddrdrdrdrdrdrdrdrdrdrdrdrur]'[dddddddddddddrrurdrdddddddururur]'[ururururururur][]}
\newcommand{\ag}{\ar@{-}'[dddddddddddddddddddddddd]'[dddddddddddddddddddddddr]'[ur][]}

\begin{figure}[h]
\begin{center}
\begin{tabular}{c}
\xymatrix @-1.7pc@M=0pt {
& {\hs} & & {\hs} & & {\hs} & & {\hs} & & {\hs} & & {\hs} & & {\hs} & & {\hs} & & {\hs} & & {\hs} & & {\hs} & & {\hs} & & {\hs} & & {\hs} & & {\hs} & & {\hz} & & {\hs} & & {\hs} & & {\hs} & & {\hs} & & {\hs} & & {\hs} & & {\hs} & & {\hs} & & {\hs} \\
{\hs} & & {\hs} & & {\hs} & & {\hs} & & {\hs} & & {\hz} & & {\hs} & & {\hs} & & {\hs} & & {\hs} & & {\hs} & & {\hs} & & {\hs} & & {\hs} & & {\hs} & {\scriptscriptstyle j} & {\hz} & &  {\hz} & & {\hs} & & {\hs} & & {\hs} & & {\hs} & & {\hs} & & {\hs} & & {\hs} & & {\hs} \\
& {\hs} & & {\hs} & & {\hs} & & {\hs} & & {\hz} & & {\hz} & & {\hs} & &{\hs} & & {\hs} & & {\hs} & & {\hs} & & {\hs} & & {\hs} & & {\hs} & {\scriptscriptstyle i} & {\hz} & & {\hz} & & {\hz} & & {\hs} & & {\hs} & & {\hs} & & {\hs} & & {\hs} & & {\hs} & & {\hs} & & {\hs} \\
{\hs} & & {\hs} & & {\hs} & & {\hs} & & {\hz} & & {\hz} & & {\hz} & {\scriptscriptstyle a} & {\hz} & & {\hs} & & {\hs} & & {\hs} & & {\hs} & & {\hs} & & {\hs} & {\scriptscriptstyle h} & {\hz} & & {\hz} & & {\hz} & & {\hz} & & {\hs} & & {\hs} & & {\hs} & & {\hs} & & {\hs} & & {\hs} & & {\hs} \\
& {\hs} & & {\hs} & & {\hs} & & {\hz} & & {\hz} & & {\hz} & & {\hz}\ab & & {\hz} & & {\hs} & & {\hs} & & {\hs} & & {\hs} & & {\hs} & {\scriptscriptstyle g} & {\hz} & & {\hz} & & {\hz} & & {\hz} & & {\hz} & & {\hs} & & {\hs} & & {\hs} & & {\hs} & & {\hs} & & {\hs} & & {\hs} \\
{\hs} & & {\hs} & & {\hs} & & {\hz} & & {\hz} & & {\hz} & & {\hz} & & {\hz} & & {\hz} & & {\hs} & & {\hs} & & {\hs} & & {\hs} & {\scriptscriptstyle f} & {\hz} & & {\hz} & & {\hz} & & {\hz} & & {\hz} & & {\hz} & {\scriptscriptstyle k} & {\hz} & & {\hs} & & {\hs} & & {\hs} & & {\hs} & & {\hs} & & {\hs} \\
& {\hs} & & {\hs} & & {\hz} & & {\hz} & & {\hz} & & {\hz} & & {\hz} & & {\hz} & & {\hz} & {\scriptscriptstyle b} & {\hz} & {\scriptscriptstyle c} & {\hz} & & {\hs} & {\scriptscriptstyle e} & {\hz} & & {\hz} & & {\hz} & & {\hz} & & {\hz} & & {\hz} & & {\hz} \ag& & {\hz} & & {\hs} & & {\hs} & & {\hs} & & {\hs} & & {\hs} & & {\hs} \\
{\hs} & & {\hs} & & {\hz} & & {\hz} & & {\hz} & & {\hz} & & {\hz} & & {\hz} & & {\hz} & & {\hz}\ad & & {\hz}\ac& & {\hz} & {\scriptscriptstyle d} & {\hz} & & {\hz} & & {\hz} & & {\hz} & & {\hz} & & {\hz} & & {\hz} & & {\hz} & & {\hz} & & {\hs} & & {\hs} & & {\hs} & & {\hs} & & {\hs} \\
& {\hs} & & {\hz} & & {\hz} & & {\hz} & & {\hz} & & {\hz} & & {\hz} & & {\hz} & & {\hz} & & {\hz} & & {\hz} & & {\hz}\af& & {\hz} & & {\hz} & & {\hz} & & {\hz} & & {\hz} & & {\hz} & & {\hz} & & {\hz} & & {\hz} & & {\hs} & & {\hs} & & {\hs} & & {\hs} \\
{\hs} & & {\hf} & & {\hz} & & {\hz} & & {\hz} & & {\hz} & & {\hz} & & {\hz} & & {\hz} & & {\hz} & & {\hz} & & {\hz} & & {\hz} & & {\hz} & & {\hz} & & {\hz} & & {\hz} & & {\hz} & & {\hz} & & {\hz} & & {\hz} & & {\hz} & & {\hs} & & {\hs} & & {\hs} \\
& {\hf} & & {\hf} & & {\hz} & & {\hz} & & {\hz} & & {\hz} & & {\hz} & & {\hz} & & {\hz} & & {\hz} & & {\hz} & & {\hz} & & {\hz} & & {\hz} & & {\hz} & & {\hz} & & {\hz} & & {\hz} & & {\hz} & & {\hz} & & {\hz} & & {\hz} & & {\hs} & & {\hs} & & {\hs} \\
{\hs} & & {\hf} & & {\hf} & & {\hz} & & {\hz} & & {\hz} & & {\hz} & & {\hz} & & {\hz} & & {\hz} & & {\hz} & & {\hz} & & {\hz} & & {\hz} & & {\hz} & & {\hz} & & {\hz} & & {\hz} & & {\hz} & & {\hz} & & {\hz} & & {\hz} & & {\hz} & & {\hs} & & {\hs} \\
& {\hs} & & {\hf} & & {\hf} & & {\hz} & & {\hz} & & {\hz} & & {\hz} & & {\hz} & & {\hz} & & {\hz} & & {\hz} & & {\hz} & & {\hz} & & {\hz} & & {\hz} & & {\hz} & & {\hz} & & {\hz} & & {\hz} & & {\hz} & & {\hz} & & {\hz} & & {\hz} & & {\hs} & & {\hs} \\
{\hs} & & {\hs} & & {\hf} & & {\hf} & & {\hz} & & {\hz} & & {\hz} & & {\hz} & & {\hz} & & {\hz} & & {\hz} & & {\hz} & & {\hz} & & {\hz} & & {\hz} & & {\hz} & & {\hz} & & {\hz} & & {\hz} & & {\hz} & & {\hz} & & {\hz} & & {\hz} & & {\hz} & & {\hs} \\
& {\hs} & & {\hs} & & {\hf} & & {\hf} & & {\hz} & & {\hz} & & {\hz} & & {\hz} & & {\hz} & & {\hz} & & {\hz} & & {\hz} & & {\hz} & & {\hz} & & {\hz} & & {\hz} & & {\hz} & & {\hz} & & {\hz} & & {\hz} & & {\hz} & & {\hz} & & {\hz} & & {\hz} & & {\hs} \\
{\hs} & & {\hs} & & {\hs} & & {\hf} & & {\hf} & & {\hz} & & {\hz} & & {\hz} & & {\hz} & & {\hz} & & {\hz} & & {\hz} & & {\hz} & & {\hz} & & {\hz} & & {\hz} & & {\hz} & & {\hz} & & {\hz} & & {\hz} & & {\hz} & & {\hz} & & {\hz} & & {\hz} & & {\hz} \\
& {\hs} & & {\hs} & & {\hs} & & {\hf} & & {\hf} & & {\hz} & & {\hz} & & {\hz} & & {\hz} & & {\hz} & & {\hz} & & {\hz} & & {\hz} & & {\hz} & & {\hz} & & {\hz} & & {\hz} & & {\hz} & & {\hz} & & {\hz} & & {\hz} & & {\hz} & & {\hz} & & {\hz} & & {\hz} \\
{\hs} & & {\hs} & & {\hs} & & {\hs} & & {\hf} & & {\hf} & & {\hz} & & {\hz} & & {\hz} & & {\hz} & & {\hz} & & {\hz} & & {\hz} & & {\hz} & & {\hz} & & {\hz} & & {\hz} & & {\hz} & & {\hz} & & {\hz} & & {\hz} & & {\hz} & & {\hz} & & {\hz} & & {\hz} & & {\hz} \\
& {\hs} & & {\hs} & & {\hs} & & {\hs} & & {\hf} & & {\hf} & & {\hz} & & {\hz} & & {\hz} & & {\hz} & & {\hz} & & {\hz} & & {\hz} & & {\hz} & & {\hz} & & {\hz} & & {\hz} & & {\hz} & & {\hz} & & {\hz} & & {\hz} & & {\hz} & & {\hz} & & {\hz} & & {\hz} \\
{\hs} & & {\hs} & & {\hs} & & {\hs} & & {\hs} & & {\hf} & & {\hf} & & {\hz} & & {\hz} & & {\hz} & & {\hz} & & {\hz} & & {\hz} & & {\hz} & & {\hz} & & {\hz} & & {\hz} & & {\hz} & & {\hz} & & {\hz} & & {\hz} & & {\hz} & & {\hz} & & {\hz} & & {\hz} \\
& {\hs} & & {\hs} & & {\hs} & & {\hs} & & {\hs} & & {\hf} & & {\hf} & & {\hz} & & {\hz} & & {\hz} & & {\hz} & & {\hz} & & {\hz} & & {\hz} & & {\hz} & & {\hz} & & {\hz} & & {\hz} & & {\hz} & & {\hz} & & {\hz} & & {\hz} & & {\hz} & & {\hz} & & {\hs} \\
{\hs} & & {\hs} & & {\hs} & & {\hs} & & {\hs} & & {\hs} & & {\hf} & & {\hf} & & {\hz} & & {\hz} & & {\hz} & & {\hz} & & {\hz} & & {\hf} & & {\hz} & & {\hz} & & {\hz} & & {\hz} & & {\hz} & & {\hz} & & {\hz} & & {\hz} & & {\hz} & & {\hz} & & {\hs} \\
& {\hs} & & {\hs} & & {\hs} & & {\hs} & & {\hs} & & {\hs} & & {\hf} & & {\hf} & & {\hz} & & {\hz} & & {\hz} & & {\hz} & & {\hf} & & {\hf} & & {\hz} & & {\hz} & & {\hz} & & {\hz} & & {\hz} & & {\hz} & & {\hz} & & {\hz} & & {\hz} & & {\hs} & & {\hs} \\
{\hs} & & {\hs} & & {\hs} & & {\hs} & & {\hs} & & {\hs} & & {\hs} & & {\hf} & & {\hf} & & {\hz} & & {\hz} & & {\hz} & & {\hf} & & {\hf} & & {\hf} & & {\hz} & & {\hz} & & {\hz} & & {\hz} & & {\hz} & & {\hz} & & {\hz} & & {\hz} & & {\hs} & & {\hs} \\
& {\hs} & & {\hs} & & {\hs} & & {\hs} & & {\hs} & & {\hs} & & {\hs} & & {\hf} & & {\hf} & & {\hz} & & {\hz} & & {\hf} & & {\hf} & & {\hf} & & {\hf} & & {\hz} & & {\hz} & & {\hz} & & {\hz} & & {\hz} & & {\hz} & & {\hz} & & {\hs} & & {\hs} & & {\hs} \\
{\hs} & & {\hs} & & {\hs} & & {\hs} & & {\hs} & & {\hs} & & {\hs} & & {\hs} & & {\hf} & & {\hf} & & {\hz} & & {\hf} & & {\hf} & & {\hf} & & {\hf} & & {\hf} & & {\hz} & & {\hz} & & {\hz} & & {\hz} & & {\hz} & & {\hz} & & {\hs} & & {\hs} & & {\hs} \\
& {\hs} & & {\hs} & & {\hs} & & {\hs} & & {\hs} & & {\hs} & & {\hs} & & {\hs} & & {\hf} & & {\hf} & & {\hf} & & {\hf} & & {\hf} & & {\hf} & & {\hf} & & {\hf} & & {\hz} & & {\hz} & & {\hz} & & {\hz} & & {\hz} & & {\hs} & & {\hs} & & {\hs} & & {\hs} \\
{\hs} & & {\hs} & & {\hs} & & {\hs} & & {\hs} & & {\hs} & & {\hs} & & {\hs} & & {\hs} & & {\hf} & & {\hf} & & {\hf} & & {\hf} & & {\hf} & & {\hf} & & {\hf} & & {\hf} & & {\hz} & & {\hz} & & {\hz} & & {\hz} & & {\hs} & & {\hs} & & {\hs} & & {\hs} \\
& {\hs} & & {\hs} & & {\hs} & & {\hs} & & {\hs} & & {\hs} & & {\hs} & & {\hs} & & {\hs} & & {\hf} & & {\hf} & & {\hf} & & {\hf} & & {\hf} & & {\hf} & & {\hf} & & {\hf} & & {\hz} & & {\hz} & & {\hz} & & {\hs} & & {\hs} & & {\hs} & & {\hs} & & {\hs} \\
{\hs} & & {\hs} & & {\hs} & & {\hs} & & {\hs} & & {\hs} & & {\hs} & & {\hs} & & {\hs} & & {\hs} & & {\hf} & & {\hf} & & {\hf} & & {\hf} & & {\hf} & & {\hf} & & {\hf} & & {\hf} & & {\hz} & & {\hz} & & {\hs} & & {\hs} & & {\hs} & & {\hs} & & {\hs} \\
& {\hs} & & {\hs} & & {\hs} & & {\hs} & & {\hs} & & {\hs} & & {\hs} & & {\hs} & & {\hs} & & {\hs} & & {\hf} & & {\hf} & & {\hf} & & {\hf} & & {\hf} & & {\hf} & & {\hf} & & {\hf} & & {\hz} & & {\hs} & & {\hs} & & {\hs} & & {\hs} & & {\hs} & & {\hs} \\
{\hs} & & {\hs} & & {\hs} & & {\hs} & & {\hs} & & {\hs} & & {\hs} & & {\hs} & & {\hs} & & {\hs} & & {\hs} & & {\hf} & & {\hf} & & {\hf} & & {\hf} & & {\hf} & & {\hf} & & {\hf} & & {\hf} & & {\hs} & & {\hs} & & {\hs} & & {\hs} & & {\hs} & & {\hs} \\
& {\hs} & & {\hs} & & {\hs} & & {\hs} & & {\hs} & & {\hs} & & {\hs} & & {\hs} & & {\hs} & & {\hs} & & {\hs} & & {\hf} & & {\hf} & & {\hf} & & {\hf} & & {\hf} & & {\hf} & & {\hf} & & {\hs} & & {\hs} & & {\hs} & & {\hs} & & {\hs} & & {\hs} & & {\hs} \\
{\hs} & & {\hs} & & {\hs} & & {\hs} & & {\hs} & & {\hs} & & {\hs} & & {\hs} & & {\hs} & & {\hs} & & {\hs} & & {\hs} & & {\hf} & & {\hf} & & {\hf} & & {\hf} & & {\hf} & & {\hf} & & {\hs} & & {\hs} & & {\hs} & & {\hs} & & {\hs} & & {\hs} & & {\hs} \\
& {\hs} & & {\hs} & & {\hs} & & {\hs} & & {\hs} & & {\hs} & & {\hs} & & {\hs} & & {\hs} & & {\hs} & & {\hs} & & {\hs} & & {\hf} & & {\hf} & & {\hf} & & {\hf} & & {\hf} & & {\hs} & & {\hs} & & {\hs} & & {\hs} & & {\hs} & & {\hs} & & {\hs} & & {\hs} \\
{\hs} & & {\hs} & & {\hs} & & {\hs} & & {\hs} & & {\hs} & & {\hs} & & {\hs} & & {\hs} & & {\hs} & & {\hs} & & {\hs} & & {\hs} & & {\hf} & & {\hf} & & {\hf} & & {\hf} & & {\hs} & & {\hs} & & {\hs} & & {\hs} & & {\hs} & & {\hs} & & {\hs} & & {\hs} \\
& {\hs} & & {\hs} & & {\hs} & & {\hs} & & {\hs} & & {\hs} & & {\hs} & & {\hs} & & {\hs} & & {\hs} & & {\hs} & & {\hs} & & {\hs} & & {\hf} & & {\hf} & & {\hf} & & {\hs} & & {\hs} & & {\hs} & & {\hs} & & {\hs} & & {\hs} & & {\hs} & & {\hs} & & {\hs} \\
{\hs} & & {\hs} & & {\hs} & & {\hs} & & {\hs} & & {\hs} & & {\hs} & & {\hs} & & {\hs} & & {\hs} & & {\hs} & & {\hs} & & {\hs} & & {\hs} & & {\hf} & & {\hf} & & {\hs} & & {\hs} & & {\hs} & & {\hs} & & {\hs} & & {\hs} & & {\hs} & & {\hs} & & {\hs} \\
& {\hs} & & {\hs} & & {\hs} & & {\hs} & & {\hs} & & {\hs} & & {\hs} & & {\hs} & & {\hs} & & {\hs} & & {\hs} & & {\hs} & & {\hs} & & {\hs} & & {\hf} & & {\hs} & & {\hs} & & {\hs} & & {\hs} & & {\hs} & & {\hs} & & {\hs} & & {\hs} & & {\hs} & & {\hs} \\
}
\end{tabular}
\end{center}
\caption{The heap broken into segments}\label{f:heap_segments}
\end{figure}
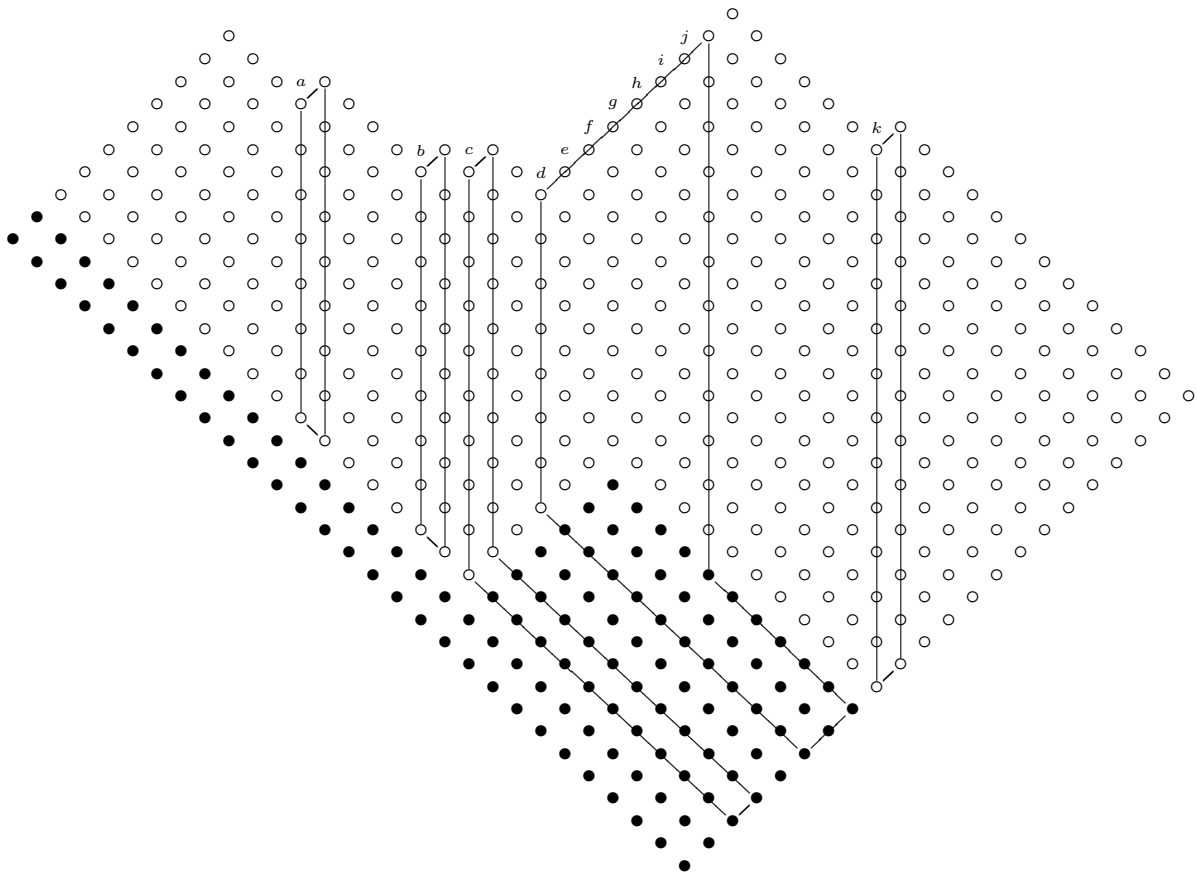

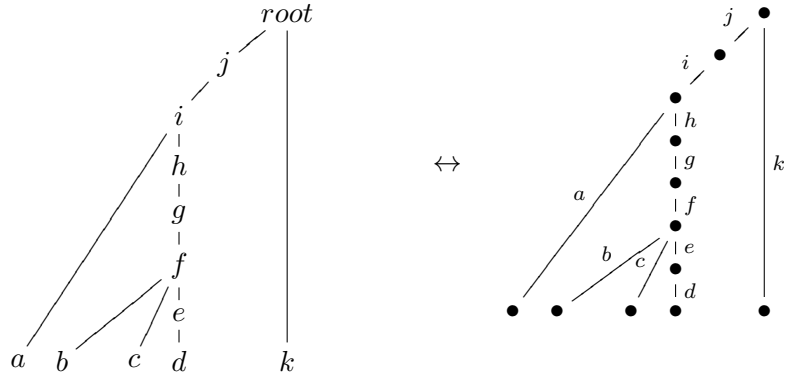
\begin{figure}[h]
\begin{center}
\begin{tabular}{cc}
\xymatrix @-1.6pc {
& & & & & & root \ar@{-}[dl] \ar@{-}[ddddddd] & \\
& & & & & j \ar@{-}[dl] & & \\
& & & & i \ar@{-}[d] \ar@{-}[dddddllll] & & \\
& & & & h \ar@{-}[d] & & & & & & \leftrightarrow \\
& & & & g \ar@{-}[d] & & \\
& & & & f \ar@{-}[d] \ar@{-}[ddl] \ar@{-}[ddlll] & & \\
& & & & e \ar@{-}[d] & & \\
a & b & & c & d & & k \\
} & \xymatrix @-1.6pc {
& & & & & & \gn \ar@{-}[dl]_{j} \ar@{-}[ddddddd]^{k} & \\
& & & & & \gn \ar@{-}[dl]_{i} & & \\
& & & & \gn \ar@{-}[d]^{h} \ar@{-}[dddddllll]_{a} & & \\
& & & & \gn \ar@{-}[d]^{g} & & \\
& & & & \gn \ar@{-}[d]^{f} & & \\
& & & & \gn \ar@{-}[d]^{e} \ar@{-}[ddl]_{c} \ar@{-}[ddlll]_{b} & & \\
& & & & \gn \ar@{-}[d]^{d} & & \\
\gn & \gn & & \gn & \gn & & \gn \\
}
\end{tabular}
\end{center}
\caption{Edge-labelled trees give ridgeline labels}\label{f:el_trees}
\end{figure}

\newcommand{\ah}{\ar@{-}'[dddddddddddddd]'[dddddddddddddddrdrdrdrdrdrdrdrdrdrdrdrdr]'[dddddddddddddrdrdrdrdrdrdrdrdrdrdrdrur]'[dddddddddddddrrurdrddddddddrurururul]'[ururururururur][]}
\newcommand{\ai}{\ar@{-}[ururururururur]}
\newcommand{\aj}{\ar@{-}[drdrdrdr]}
\newcommand{\ajs}{\ar@{-}[drdrdrdr]}

\begin{figure}[h]
\begin{center}
\begin{tabular}{c}
\xymatrix @-1.7pc@M=0pt {
 & {\hs} & & {\hs} & & {\hs} & & {\hs} & & {\hz} & & {\hs} & & {\hs} & & {\hs}\\
 {\hs} & & {\hs} & & {\hs} & & {\hs} & & {\hz} & &  {\hz} & & {\hs} & & {\hs} \\
 & {\hs} & & {\hs} & & {\hs} & & {\hz} & & {\hz} & & {\hs} & & {\hs} & & {\hs} \\
 {\hs} & & {\hs} & & {\hs} & & {\hz} & & {\hz} & & {\hz} & & {\hs} & & {\hs} &  \\
 & {\hs} & & {\hs} & & {\hz} & & {\hz} & & {\hz} & & {\hs} & & {\hs} & & {\hs}  \\
 {\hs} & & {\hs} & & {\hz} & & {\hz} & & {\hz} & & {\hz} & & {\hs} & & {\hs} &  \\
 & {\hs} & & {\hz} & & {\hz} & & {\hz} & & {\hz} & & {\hs} & & {\hs} & & {\hs}  \\
 {\hs} & & {\hz} & & {\hz} & & {\hz} & & {\hz} & & {\hz} & & {\hs} & & {\hs} & &  \\
 & {\hz}\ah& & {\hz} & & {\hz} & & {\hz} & & {\hz} & & {\hs} & & {\hs} & & {\hs} \\
 {\hs} & & {\hz} & & {\hz} & & {\hz} & & {\hz} & & {\hz} & & {\hs} & & {\hs} &  \\
 & {\hz} & & {\hz} & & {\hz} & & {\hz} & & {\hz} & & {\hs} & & {\hs} & & {\hs}  \\
 {\hs} & & {\hz} & & {\hz} & & {\hz} & & {\hz} & & {\hz} & & {\hs} & & {\hs} &  \\
 & {\hz} & & {\hz} & & {\hz} & & {\hz} & & {\hz} & & {\hs} & & {\hs} & & {\hs}  \\
 {\mathfrak{1}} & & {\hz} & & {\hz} & & {\hz} & & {\hz} & & {\hz} & & {\hs} & & {\hs} &  \\
 & {\hz} & & {\hz} & & {\hz} & & {\hz} & & {\hz} & & {\hs} & & {\hs} & & {\hs}  \\
 {\hs} & & {\hz} & & {\hz} & & {\hz} & & {\hz} & & {\hz} & & {\hs} & & {\hs} &  \\
 & {\hz} & & {\hz} & & {\hz} & & {\hz} & & {\hz} & & {\hs} & & {\hs} & & {\hs}  \\
 {\hs} & & {\hz} & & {\hz} & & {\hz} & & {\hz} & & {\hz} & & {\hs} & & {\hs} &  \\
 & {\hz} & & {\hz} & & {\hz} & & {\hz} & & {\hz} & & {\hs} & & {\hs} & & {\hs}  \\
 {\hs} & & {\hz} & & {\hz} & & {\hz} & & {\hz} & & {\hz} & & {\hs} & & {\hs} &  \\
 & {\hz} & & {\hz} &\aj& {\hz} & & {\hz} & & {\hz} & & {\mathfrak{3}} & & {\hs} & & {\hs}  \\
 {\hs} & & {\hz} & & {\hf} & & {\hz} & & {\hz} & & {\hz} & & {\hs} & & {\hs} &  \\
 & {\hz} & & {\hf} & & {\hf} & & {\hz} & & {\hz} & & {\hs} & & {\hs} & & {\hs}  \\
 {\hs} &\ai& {\hf} & & {\hf} & & {\hf} & & {\hz} & & {\hz} & & {\hs} & & {\hs} &  \\
 & {\hs} & & {\hf} & & {\hf} & & {\hf} & & {\hz} & & {\hs} & & {\hs} & & {\hs}  \\
 {\hs} & & {\hs} & & {\hf} & & {\hf} & & {\hf} & & {\hz} & & {\hs} & & {\hs} &  \\
 & {\hs} & & {\hs} & & {\hf} & & {\hf} & & {\hf} & & {\hs} & & {\hs} & & {\hs}  \\
 {\hs} & & {\hs} & & {\hs} & & {\hf} & & {\hf} & & {\hf} & & {\hs} & & {\hs} &  \\
 & {\hs} & & {\hs} & & {\hs} & & {\hf} & & {\hf} & & {\hf} & & {\hs} & & {\hs}  \\
 {\hs} & & {\hs} & & {\hs} & & {\hs} & & {\hf} & & {\hf} & & {\hf} & & {\hs} &  \\
 & {\hs} & & {\hs} & & {\hs} & & {\hs} & & {\hf} & & {\hf} & & {\hf} & & {\hs}  \\
 {\hs} & & {\hs} & & {\hs} & & {\hs} & & {\hs} & & {\hf} & & {\hf} & & {\hf} &  \\
 & {\hs} & & {\hs} & & {\hs} & & {\hs} & & {\hs} & & {\hf} & & {\hf} & & {\hs}  \\
 {\hs} & & {\hs} & & {\hs} & & {\hs} & & {\hs} & & {\hs} & & {\hf} & & {\hf} &  \\
 & {\hs} & & {\hs} & & {\hs} & & {\hs} & & {\hs} & & {\hs} & & {\hf} & & {\hs} & & {\mathfrak{2}} \\
 {\hs} & & {\hs} & & {\hs} & & {\hs} & & {\hs} & & {\hs} & & {\hs} & & {\hf} & \\
 & {\hs} & & {\hs} & & {\hs} & & {\hs} & & {\hs} & & {\hs} & & {\hs} & & {\hs}  \\
 {\hs} & & {\hs} & & {\hs} & & {\hs} & & {\hs} & & {\hs} & & {\hs} & & {\hs} &  \\
 & {\hs} & & {\hs} & & {\hs} & & {\hs} & & {\hs} & & {\hs} & & {\hs} & & {\hs}  \\
}
\end{tabular}
\end{center}
\caption{A single segment broken into regions}\label{f:segment_regions}
\end{figure}
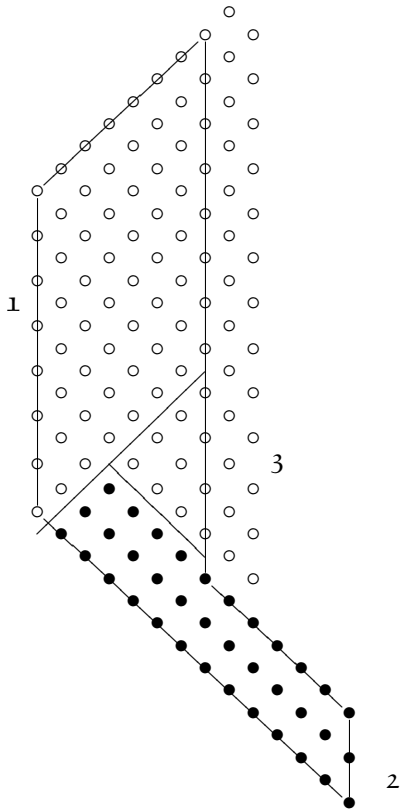

\bigskip

The following lemma, which allows us to work segment by segment in specifying
the mask $\s(t)$, follows immediately from the definition.

\begin{lemma}\label{l:distinct_entries}
Suppose $v$ and $v'$ are distinct valley columns in the heap of $\w$.  Then the
entries in the segment associated to $v$ are disjoint from the entries of the
segment associated to $v'$.
\end{lemma}

{\bf Definition of $\lambda(v, t)$:} Recall that the vertices of $T(w)$
correspond to pairs of edges in the ridgeline that lie at the same level in
the heap such that, in the lattice path associated to the ridgeline, the left
edge is a ``down step'', and the right edge is an ``up step''.  Since we also
adjoin a root vertex to the tree, we can associate each vertex, and hence each
pair of matching steps in the ridgeline, to the unique edge above it in the
rooted tree.  Although each edge in $T(w)$ is represented by two different
steps in the ridgeline, we choose the ``up steps'' as our representatives for
a given edge in the tree $T(w)$.  (Observe that this is a choice from which a
dual construction could be developed.)  We label these ``up steps'' by the
edge label given in $t$.  Restricting to a particular segment associated with
valley $v$, let $e_1, e_2, \ldots, e_{q(v)}$ denote these edge labels
associated with the ``up steps'' of the ridgeline in columns $v$ through
$v+q(v)$, where the edges are ordered sequentially with $e_1$ labeling the
leaf edge which is closest to the valley $v$.  The edge labels form a
partition
\[ p(v) \geq  e_1 \geq e_2 \geq \cdots \geq e_{q(v)} \geq 0 \] 
that we denote by $\lambda(v,t)$.  Observe that $e_1 < p(v)$ precisely when
there exists a valley to the left or right of $v$ with a large leaf edge label
that ``zeros out'' entries above in the construction of $\gamma(t)$.

For example, in Figure~\ref{f:heap_segments} the labels $d, e, \ldots, j$ give
the parts of the the partition $\lambda(v,t)$ where $v$ is the fourth valley
column from the left and $t$ is the tree shown in Figure~\ref{f:el_trees}.

{\bf Definition of the mask $\s(t)$:}
We are now in a position to define our mask $\s(t)$ for $x(t)$ on $\w$, working
segment by segment.  Fix a valley column $v$ and associated partition of edge
labels $\lambda(v,t)$.  We define $\s(t)$ so that the number of defects in
$\s(t)$ restricted to the segment is $|\lambda(v,t)| = \sum_{i=1}^{q(v)} e_i$.

The mask-values of the entries in the segment will be completely determined by
the partition $\lambda(v,t)$ of edge labels from $t$, but we define some
auxiliary partitions to simplify notation.  Let 
\[ \lambda'(v,t) = \text{ nonnegative parts of } \lambda(v,t) - (1, 2, 3, \ldots), \]
\[ \nu(v,t) = \text{ nonnegative parts of } \lambda(v,t)^\dag - (0, 1, 2, \ldots), \text{ and } \]
\[ \eta(v,t) = \text{ nonnegative parts of } \nu(v,t) - (r(v), r(v), r(v), \ldots). \]
Here, $\lambda^\dag$ denotes the transpose of a partition $\lambda$.  We
freely identify $\lambda$ with its Ferrers diagram, which has $\lambda_i$
boxes in the $i$-th row.

\begin{lemma}\label{l:distinct_rows}
The partitions $\lambda'(v,t)$, $\nu(v,t)$ and $\eta(v,t)$ have all distinct row lengths.
\end{lemma}
\begin{proof}
If $\eta(v,t)_i = \eta(v,t)_{i+1} \neq 0$ then $\nu(v,t)_{i} = \eta(v,t)_i +
r(v) = \eta(v,t)_{i+1} + r(v) = \nu(v,t)_{i+1} \neq 0$.  But then,
$\lambda(v,t)^\dag_i = \nu(v,t)_i + (i-1) < \nu(v,t)_{i+1} + i =
\lambda(v,t)^\dag_{i+1}$, contradicting that $\lambda(v,t)^\dag$ is a partition.
The argument for $\lambda'(v,t)$ is similar.
\end{proof}

Begin with the mask-values from $\gamma(t)$ restricted to the current segment.
In region $\mathfrak{1}$, set the lowest $\lambda'(v,t)_i$ entries in column
$v+i$ to be zero-defects $\hd$.  Observe that this is equivalent to inserting
$\min( \lambda(v,t)^\dag_i, i-1 )$ zero-defects on a consecutive NW-SE row in
region $\mathfrak{1}$ of valley diagonal $i$.

Next, insert row $i$ of $\nu(v,t)$ as a consecutive sequence of plain-one $\hf$
entries lying on a diagonal emanating to the northeast from the entry just above
the highest zero-defect $\hd$ in column $v+i$; if there is no zero-defect in
column $v+i$ (because $\lambda'(v,t)_i=0$), then the plain-one $\hf$ entries
start just above the lowest entry of column $v+i$ in region $\mathfrak{1}$.  All
other entries in region $\mathfrak{1}$ are plain-zeros $\hz$.

In region $\mathfrak{3}$, we define a {\bf feasible subregion} where entries
are initially defined to have mask-value 1.  All of the entries of region
$\mathfrak{3}$ lying outside of the feasible subregion will remain mask-value 0.
The feasible subregion is defined by inserting row $i$ of the partition
$\eta(v,t)$ as a consecutive sequence of plain-one $\hf$ entries lying on a
diagonal emanating to the southeast from the highest entry in column $i$ of
region $\mathfrak{3}$.

Next, we consider the strings going southeast from either the border between
region $\mathfrak{1}$ and region $\mathfrak{2}$, or the border between region
$\mathfrak{1}$ and the first $\eta(v,t)^\dag_1 + 1$ columns of region
$\mathfrak{3}$.  If we follow these strings down in the heap, they either
eventually hit the end of a row of $\eta(v,t)$, or they hit the end of a valley
diagonal at an entry of $w_0^J$ in region $\mathfrak{2}$.  In either case, the
strings then change direction so that they emanate downward in the southwest
direction and cross the remaining diagonals of regions $\mathfrak{2}$ and
$\mathfrak{3}$, which consist entirely of mask-value 1 entries.  We call the
intersection of these northeast-southwest string paths with a given valley
diagonal the {\bf cross-diagonal entries} of the valley diagonal.  Since all of
the rowlengths of $\eta(v,t)$ are distinct by Lemma~\ref{l:distinct_rows}, we
may observe that the cross-diagonal entries are ordered so that a cross-diagonal
entry corresponding to the string path emanating from column $i$ will lie below
a cross-diagonal entry corresponding to the string path emanating from column
$j>i$ on the lower boundary of region $\mathfrak{1}$.

For example, Figure~\ref{f:cross_diagonals} illustrates the string paths that
yield cross-diagonals for a particular heap.

\begin{figure}[h]
\includegraphics[scale=0.7]{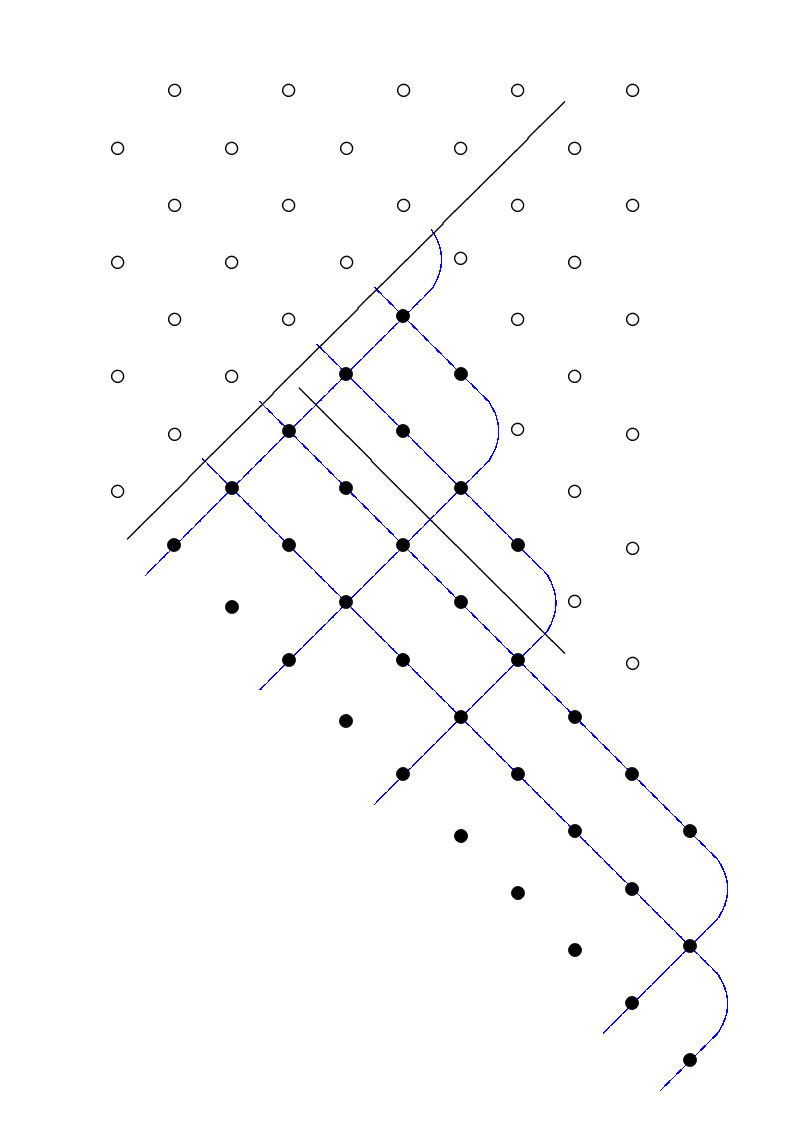}
\caption{Cross-diagonals}\label{f:cross_diagonals}
\end{figure}

In region $\mathfrak{2}$ and in the feasible subregion of region $\mathfrak{3}$,
we insert into valley diagonal $i$ exactly $\nu(v,t)_i - \eta(v,t)_i$
zero-defects in the lowest cross-diagonal entries of the valley diagonal.
(Observe that by definition, $\nu(v,t)_i - \eta(v,t)_i \leq r(v)$ for all $i$.)
We also insert row $i$ of $\eta(v,t)$ as a consecutive sequence of one-defect
$\ho$ entries lying on a diagonal emanating to the southeast from the top entry
in valley diagonal $i$ not lying in region $\mathfrak{1}$.

Table~\ref{t:part_ins} summarizes how the partitions are inserted into the
regions.

\begin{table}[h!]
\begin{center}
\caption{Partition insertion}\label{t:part_ins}
\begin{tabular}{|p{0.3in}|p{3.4in}|p{2.2in}|}
\hline
$\lambda^\dag_i$, $\lambda_i'$ & $\min( \lambda(v,t)^\dag_i, i-1 )$ inserted as a
consecutive NW-SE row of zero-defects on valley diagonal $i$; this is equivalent
to inserting $\lambda_i'$ as a column of zero-defects & region $\mathfrak{1}$ \\
\hline
$\nu_i$ &  inserted as defects on the $i$th valley diagonal; 
\begin{itemize}
    \item $\eta_i$ of these are one-defects lying at the top of the diagonal 
    \item $\nu_i - \eta_i \leq r(v)$ of these are zero-defects lying at cross-diagonal positions.
\end{itemize} & region $\mathfrak{2}$ and feasible subregion of $\mathfrak{3}$ \\
&  also inserted as a consecutive SW-NE row of plain-one entries & region $\mathfrak{1}$ \\
\hline
$\eta_i$ &  inserted as a consecutive NW-SE row of one-defect entries along valley
diagonal $i$ (followed by a single zero-defect) & region $\mathfrak{2}$ and feasible subregion of $\mathfrak{3}$ \\
&  also inserted as a consecutive NW-SE row of mask-value 1 entries defining the feasible subregion & region $\mathfrak{3}$ \\
\hline
\end{tabular}
\end{center}
\end{table}

\begin{figure}[p]
\includegraphics[scale=1.0]{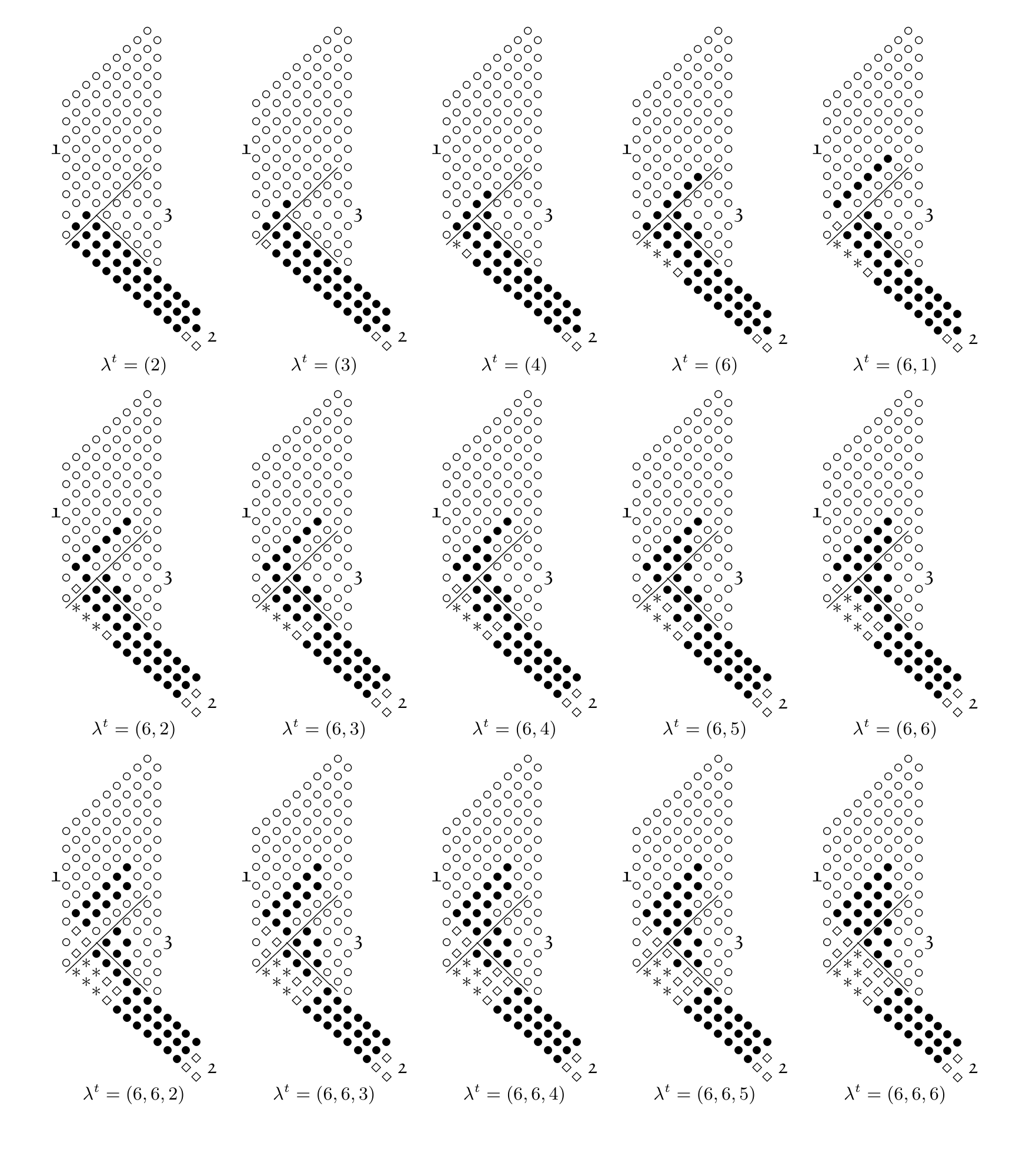}
\caption{Some mask assignments for $r=3$}\label{f:insertion}
\end{figure}

\begin{example}
Figure~\ref{f:insertion} illustrates the construction for some typical
partitions and $r=3$.  Observe that each of the partitions $\lambda(v,t)^\dag$,
$\lambda'(v,t)$, $\nu(v,t)$ and $\eta(v,t)$ are clearly embedded inside
the heap.
\end{example}

\begin{lemma}\label{l:well_defined}
The construction of $\s(t)$ is consistent in the sense that valley diagonal $i$
always has at least $\nu(v,t)_i - \eta(v,t)_i$ cross-diagonal entries.
Furthermore, the $i$th valley diagonal in $\s(t)$ contains exactly
$\lambda(v,t)^\dag_i$ defects.
\end{lemma}
\begin{proof}
First, suppose that $\eta(v,t)_i \neq 0$.  Then, the lowest one-defect $p$ on
valley diagonal $i$ has a corresponding entry $q$ in the feasible subregion of
region $\mathfrak{3}$ located exactly $r(v)$ diagonals to the NE, and $q$ has
mask-value 1.  Therefore, there exist by construction at least $r(v) \geq
\nu(v,t)_i - \eta(v,t)_i$ cross-diagonal entries below $p$ on valley diagonal
$i$.

Next, suppose that $j$ is the first diagonal having $\eta(v,t)_{j} = 0$, so
$\eta(v,t)_i \neq 0$ for all $i < j$.  Then, for each $i \geq j$, we insert into
valley diagonal $i$ precisely $\min( \lambda(v,t)^\dag_i, i-1 )$ zero-defects in
region $\mathfrak{1}$ together with $\nu(v,t)_i$ zero-defects in region
$\mathfrak{2}$ or the feasible subregion of region $\mathfrak{3}$.

By construction, there exist $r(v)$ cross-diagonal entries in valley diagonal
$j$.  As we move one diagonal to the right, we lose one cross-diagonal entry.
However, we gain a potential zero-defect position in region $\mathfrak{1}$, so
the partition inequality $\lambda(v,t)^\dag_{i+1} \leq \lambda(v,t)^\dag_i$
implies that $\nu(v,t)_i$ decreases by at least 1.  Hence, we always have
enough cross-diagonal entries to insert into by induction.

Note that although our construction may require changing some of the mask-value
1 entries in the feasible subregion of region $\mathfrak{3}$ to be one-defects,
we never change any of the entries outside of the feasible subregion in region
$\mathfrak{3}$ because $\eta(v,t)$ is a partition.

Finally, it follows from the definitions that the $i$th valley diagonal in
$\s(t)$ contains exactly $\lambda(v,t)^\dag_i$ defects, because we insert $\min(
\lambda(v,t)^\dag_i, i-1 )$ zero-defects in region $\mathfrak{1}$ together with
$\nu(v,t)_i$ total defects below region $\mathfrak{1}$.  Hence, we insert
$\lambda(v,t)^\dag_i$ defects in all.
\end{proof}

\begin{lemma}\label{l:good_insertion}
Each entry in $\s(t)$ restricted to the current segment has the defect status
that we claimed in the construction, and $\w^{\s(t)} = x(t)$.  Furthermore, the
construction preserves the mask-value and defect status of all entries outside
the current segment.
\end{lemma}
\begin{proof}
We work by induction on the number of entries in $\lambda(v,t)^\dag$.  The base
case is when $\lambda(v,t)^\dag = \emptyset$, which corresponds to the mask
$\gamma(t)$.

Since $v$ and $t$ will be fixed throughout the proof, we sometimes omit these
arguments from the partitions $\lambda^\dag$, $\nu$ and $\eta$.  Also, we denote
$\s(t)$ restricted to the current segment by $\s(\lambda^\dag)$.  To begin our
proof of the inductive case, suppose the mask $\s(\lambda^\dag)$ is constructed
as in the definition, all of the entries have the correct defect status, and
the mask restricted to the current segment encodes the element obtained by
restricting $x(t)$ to the current segment.

We now consider adding an entry $b$ at the end of row $i$ and hence in column
$j=\lambda^\dag_i+1$ of $\lambda^\dag$.  Row $i$ of $\lambda^\dag$ containing
$b$ corresponds to the valley diagonal $i$ in the heap.  There are three
cases:
\begin{enumerate}
\item[(1)]  Adding $b$ to $\lambda^\dag$ does not change $\nu$ nor $\eta$.
\item[(2)]  Adding $b$ to $\lambda^\dag$ adds an entry to $\nu$ but not $\eta$.
\item[(3)]  Adding $b$ to $\lambda^\dag$ adds an entry to both $\nu$ and $\eta$.
\end{enumerate}
Denote the result of adding $b$ to $\lambda^\dag$, $\nu$, and $\eta$
respectively by $\widetilde{\lambda^\dag}$, $\widetilde{\nu}$, and
$\widetilde{\eta}$.  We consider each case in turn.

{\bf Case (1).}
In this case, $\s(\widetilde{\lambda^\dag})$ is obtained from $\s(\lambda^\dag)$
by adding a defect to column $v+j$ of the $i$th valley diagonal in region
$\mathfrak{1}$, and the entries in regions $\mathfrak{2}$ and $\mathfrak{3}$
all remain the same.

First, we claim that $\nu_j > 0$, so some entry in column $v+j$ has mask-value 1
in $\s(\lambda^\dag)$.  To see this, consider that $b$ is inserted into row $i$
and column $j$ of $\lambda^\dag$ with $j<i$.  Otherwise, adding $b$ would have
changed $\nu$.  Hence, $\lambda^\dag_j \geq j$, so $\nu_j > 0$.

By the induction hypothesis, we have that $\s(\widetilde{\lambda^\dag})$ is
obtained from $\s(\lambda^\dag)$ by first shifting all of the entries of
$\nu_j$ up one level in their columns of the heap in region $\mathfrak{1}$,
then changing the entry $b$ that was the start of $\nu_j$ in $\s(\lambda^\dag)$
to be a zero-defect in $\s(\widetilde{\lambda^\dag})$.

The schematic shown in Figure~\ref{f:c1} shows how to move a single row of $\nu$
up one level in the heap using string moves.  Here, a {\bf string move} is an
operation on masks in which we change the mask values of two entries
$\{\sigma_p, \sigma_q\}$ in the heap whose strings meet.  The mask values
$\sigma_p$ and $\sigma_q$ are replaced by $1-\sigma_p$ and $1-\sigma_q$
respectively, and no other mask-values are changed.  Observe that each of these
string moves preserves the element being encoded by the mask as well as the
defect-status of the other entries in the heap.

\begin{figure}[h]
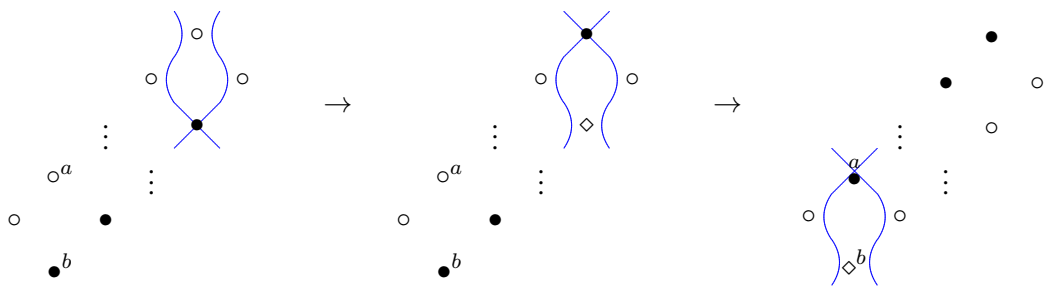

\begin{center}
\begin{tabular}{ccccc}
\heap {
{\hs} & & {\hs} & & \StringLR{\hz} & & \\
& {\hs} & & \StringR{\hz} & & \StringL{\hz} & \\
{\hs} & & {\vdots} & & \StringLRX{\hf} & & \\
& {\hz^a} & & {\vdots} & & {\hs} & \\
{\hz} & & {\hf} & & {\hs} & & \\
& {\hf^b} & & {\hs} & & {\hs} & \\
} & \parbox{0.2in}{\vspace{0.8in}$\rightarrow$} & \heap {
{\hs} & & {\hs} & & \StringLRX{\hf} & & \\
& {\hs} & & \StringR{\hz} & & \StringL{\hz} & \\
{\hs} & & {\vdots} & & \StringLR{\hd} & & \\
& {\hz^a} & & {\vdots} & & {\hs} & \\
{\hz} & & {\hf} & & {\hs} & & \\
& {\hf^b} & & {\hs} & & {\hs} & \\
} & \parbox{0.2in}{\vspace{0.8in}$\rightarrow$} & \heap {
{\hs} & & {\hs} & & {\hf} & & \\
& {\hs} & & {\hf} & & {\hz} & \\
{\hs} & & {\vdots} & & {\hz} & & \\
& \StringLRX{\stackrel{a}{\hf}} & & {\vdots} & & {\hs} & \\
\StringR{\hz} & & \StringL{\hz} & & {\hs} & & \\
& \StringLR{\hd^b} & & {\hs} & & {\hs} & \\
}
\end{tabular}
\end{center}
\caption{Case (1) of the proof}\label{f:c1}
\end{figure}

{\bf Case (2).}
In this case, we add $b$ as a zero-defect to region $\mathfrak{2}$ or
$\mathfrak{3}$.  Abusing notation, let $b$ denote the lowest cross-diagonal
entry on valley diagonal $i$ that is not a zero-defect.  By the induction
hypothesis, we have that $\s(\widetilde{\lambda^\dag})$ is obtained from
$\s(\lambda^\dag)$ by changing $b$ from a plain-one to a zero-defect and adding an
additional plain-one at the end of $\nu_i$ in region $\mathfrak{1}$.

We claim that both of these can be accomplished with a single string move.  To
see this, observe that the right string of $b$ does not meet any mask-value 0
entries in region $\mathfrak{2}$ because as we move right, we lose a
cross-diagonal entry and gain an entry from region $\mathfrak{1}$ on the valley
diagonal, so the level of the highest cross-diagonal entry that is a zero-defect is
strictly decreasing.  By the definition, the right string emanating up from a
cross-diagonal entry eventually hits a plain-zero in region $\mathfrak{3}$ and
turns towards the northwest, following a column of $\eta$.  Observe that the
total number of valley diagonals crossed by the right string is equal to the
number of zero-defects in row $i$ of $\nu$ by the definition because all of
the rowlengths in $\eta$ are distinct by Lemma~\ref{l:distinct_rows}.  Hence,
we have that the right string of $b$ meets the plain-zero $a$ that occurs just
after the last entry in row $i$ of $\nu$ in region $\mathfrak{1}$.

The left string of $b$ emanates up towards the northwest and by induction
eventually meets row $\nu_i$ in region $\mathfrak{1}$.  Hence, the left and
right strings of $b$ meet at the mask-value 0 entry $a$ just beyond the last
entry in $\nu_i$.  This is illustrated in Figure~\ref{f:c2}.  If we apply a
string move to $b$ and $a$, we achieve the desired effect.

\begin{figure}[h]
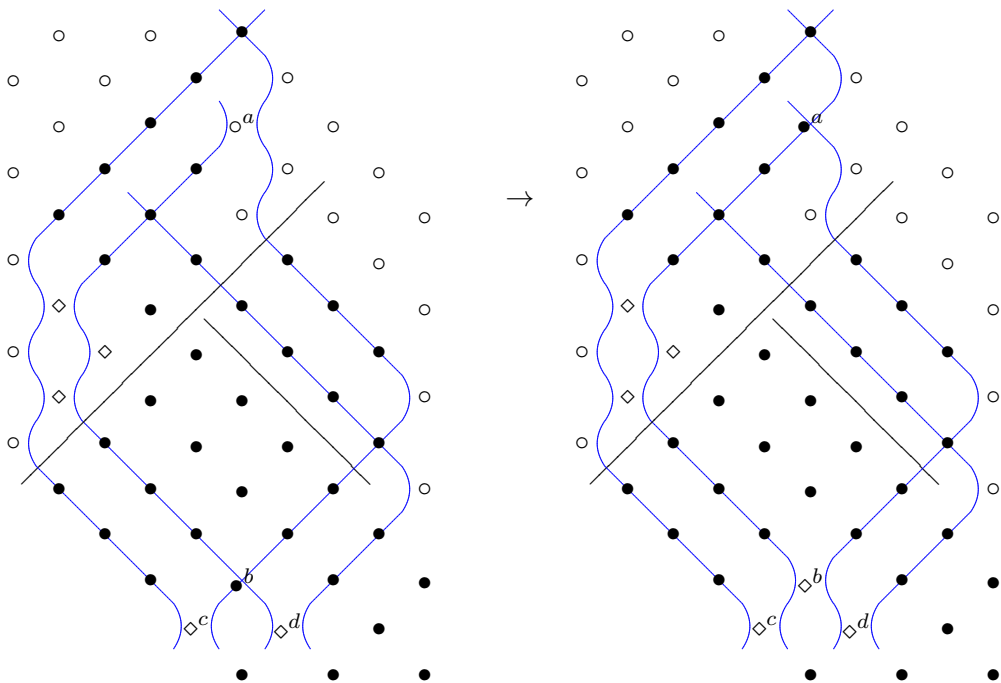

\begin{center}
\begin{tabular}{ccccc}
\heap { 
 & {\hz} & & {\hz} & & \StringLRX{\hf} & & {\hs} & & {\hs} & \\
 {\hz} & & {\hz} & & \StringRX{\hf} & & \StringL{\hz} & & {\hs} & \\
 & {\hz} & & \StringRX{\hf} & & \StringLR{\hz^a} & & {\hz} & & {\hs} \\
 {\hz} & & \StringRX{\hf} & & \StringRX{\hf} & & \StringL{\hz} & & {\hz} & \\
 & \StringRX{\hf} & & \StringLRX{\hf} & & \StringR{\hz} & & {\hz} & & {\hz} \\
 \StringR{\hz} & & \StringRX{\hf} & & \StringLX{\hf} & & \StringLX{\hf} & & {\hz} & \\
 & \StringLR{\hd} & & {\hf} &\aj& \StringLX{\hf} & & \StringLX{\hf} & & {\hz} \\
 \StringR{\hz} & & \StringL{\hd} & & {\hf} & & \StringLX{\hf} & & \StringLX{\hf} & \\
 & \StringLR{\hd} & & {\hf} & & {\hf} & & \StringLX{\hf} & & \StringL{\hz} \\
 \StringR{\hz} & & \StringLX{\hf} & & {\hf} & & {\hf} & & \StringLRX{\hf} & \\
 \ai& \StringLX{\hf} & & \StringLX{\hf} & & {\hf} & & \StringRX{\hf} & & \StringL{\hz} \\
 {\hs} & & \StringLX{\hf} & & \StringLX{\hf} & & \StringRX{\hf} & & \StringRX{\hf} & \\
 & {\hs} & & \StringLX{\hf} & & \StringLRX{\hf^b} & & \StringRX{\hf} & & {\hf} \\
 {\hs} & & {\hs} & & \StringLR{\hd^c} & & \StringLR{\hd^d} & & {\hf} & \\
 & {\hs} & & {\hs} & & {\hf} & & {\hf} & & {\hf} & \\
} & \parbox{0.1in}{\vspace{1.8in}$\rightarrow$} & 
\heap { 
 & {\hz} & & {\hz} & & \StringLRX{\hf} & & {\hs} & & {\hs} & \\
 {\hz} & & {\hz} & & \StringRX{\hf} & & \StringL{\hz} & & {\hs} & \\
 & {\hz} & & \StringRX{\hf} & & \StringLRX{\hf^a} & & {\hz} & & {\hs} \\
 {\hz} & & \StringRX{\hf} & & \StringRX{\hf} & & \StringL{\hz} & & {\hz} & \\
 & \StringRX{\hf} & & \StringLRX{\hf} & & \StringR{\hz} & & {\hz} & & {\hz} \\
 \StringR{\hz} & & \StringRX{\hf} & & \StringLX{\hf} & & \StringLX{\hf} & & {\hz} & \\
 & \StringLR{\hd} & & {\hf} &\aj& \StringLX{\hf} & & \StringLX{\hf} & & {\hz} \\
 \StringR{\hz} & & \StringL{\hd} & & {\hf} & & \StringLX{\hf} & & \StringLX{\hf} & \\
 & \StringLR{\hd} & & {\hf} & & {\hf} & & \StringLX{\hf} & & \StringL{\hz} \\
 \StringR{\hz} & & \StringLX{\hf} & & {\hf} & & {\hf} & & \StringLRX{\hf} & \\
 \ai& \StringLX{\hf} & & \StringLX{\hf} & & {\hf} & & \StringRX{\hf} & & \StringL{\hz} \\
 {\hs} & & \StringLX{\hf} & & \StringLX{\hf} & & \StringRX{\hf} & & \StringRX{\hf} & \\
 & {\hs} & & \StringLX{\hf} & & \StringLR{\hd^b} & & \StringRX{\hf} & & {\hf} \\
 {\hs} & & {\hs} & & \StringLR{\hd^c} & & \StringLR{\hd^d} & & {\hf} & \\
 & {\hs} & & {\hs} & & {\hf} & & {\hf} & & {\hf} & \\
}
\end{tabular}
\end{center}
\caption{Case (2) of the proof}\label{f:c2}
\end{figure}

Next, observe that the defect status of the other entries in the heap remain
unchanged.  Suppose there exists a defect $c$ whose right string is the right
string of $b$ in $\s(\lambda^\dag)$.  If the strings of $c$ cross above $a$ in the
heap, then $c$ remains a defect in $\s(\widetilde{\lambda^\dag})$.
Otherwise, the strings of $c$ must cross at one of the mask-value 1 entries on
the column of $\eta$ that the right string of $b$ traverses.  In this case, the
strings of $c$ will still cross at a corresponding mask-value 1 entry on the
valley diagonal containing $b$ in $\s(\widetilde{\lambda^\dag})$.  Similarly, if
there exists a defect $c$ whose right string is the left string of $b$ in
$\s(\widetilde{\lambda^\dag})$, then $c$ must have been a defect in
$\s(\lambda^\dag)$.

Suppose there exists a defect $d$ whose left string is the left string of $b$ in
$\s(\lambda^\dag)$.  If the strings of $d$ cross above $a$ in the heap,
then $d$ remains a defect in $\s(\widetilde{\lambda^\dag})$.  Otherwise, the
strings of $d$ must cross at an entry of $\nu_i$ in region $\mathfrak{1}$ as
shown.  After the string move, the left string of $d$ is the right string of
$b$.  Since the right string of $b$ proceeds past the end of row $\nu_i$, it
must cross the right string of $d$.  Hence, $d$ remains a defect in
$\s(\widetilde{\lambda^\dag})$.  Similarly, if there exists a defect $d$ whose left
string is the right string of $b$ in $\s(\widetilde{\lambda^\dag})$, then $d$ must
have been a defect in $\s(\lambda^\dag)$.

It is possible that the strings of a defect cross at $b$ in $\s(\lambda^\dag)$, in
which case the defect status of the entry is preserved in
$\s(\widetilde{\lambda^\dag})$.  The only other defects on the path of the strings
of $b$ are the zero-defects in region $\mathfrak{1}$ that lie on the left string
of $b$.  It is straightforward to see that these defects are preserved by the
string move and its inverse as well.  Hence, the defect status of the other
entries in the heap remain unchanged.

{\bf Case (3).}
Here, we have two subcases.  By induction, there must exist exactly $r$
zero-defects on valley diagonal $i$.  Either $r = 0$, or $r > 0$.

{\bf Case (3), Subcase $\mathbf{r = 0}$. }
If $r = 0$, then we let $b$ be the highest plain-zero in region $\mathfrak{3}$
on valley diagonal $i$.  By the induction hypothesis we have that
$\s(\widetilde{\lambda^\dag})$ is obtained from $\s(\lambda^\dag)$ by changing $b$
from a plain-zero to a one-defect and adding an additional plain-one at the end
of $\nu_i$ in region $\mathfrak{1}$.  

To accomplish this, we use a single string move.  To see this, observe that the
right string of $b$ does not meet any mask-value 1 entries in region
$\mathfrak{3}$ by Lemma~\ref{l:distinct_rows}.  Moreover, the number of
plain-ones of row $\nu_i$ in region $\mathfrak{1}$ is equal to the number of
one-defects in valley diagonal $i$, by induction.  Hence, the right string of
$b$ meets the plain-zero $a$ that occurs just after the last entry in row $i$ of
$\nu$ in region $\mathfrak{1}$.  The left string of $b$ emanates up towards the
northwest and by induction eventually meets row $\nu_i$ in region
$\mathfrak{1}$.  Hence, the left and right strings of $b$ meet at $a$ in
$\s(\lambda^\dag)$, as illustrated in Figure~\ref{f:c3r0}.  If we apply a string
move to $b$ and $a$, we achieve the desired effect.

\begin{figure}[h]
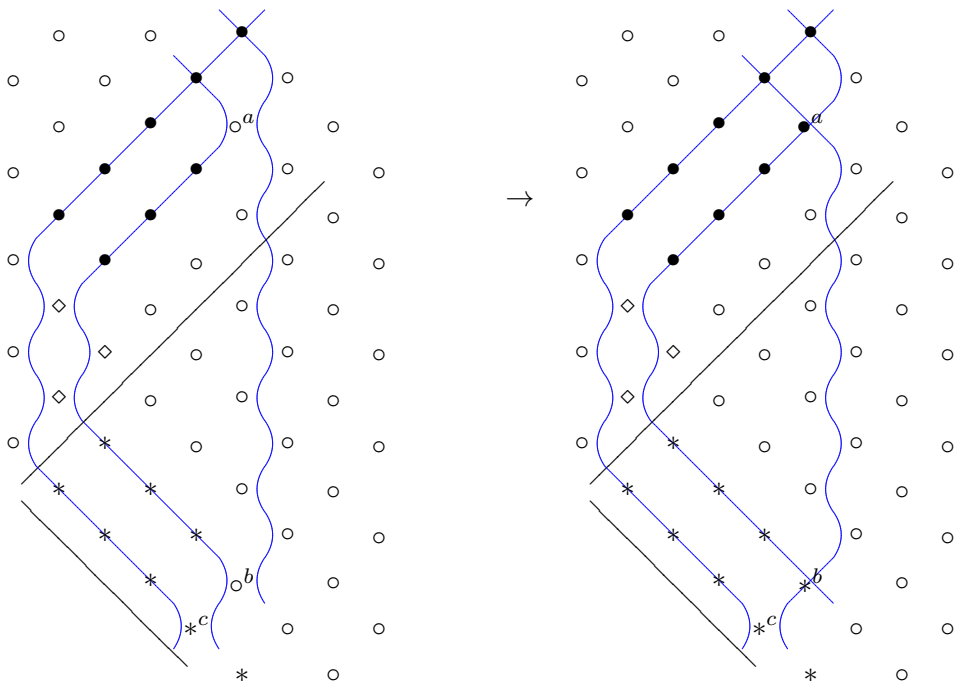

\begin{center}
\begin{tabular}{ccccc}
\heap { 
 & {\hz} & & {\hz} & & \StringLRX{\hf} & & {\hs} & & {\hs} & \\
 {\hz} & & {\hz} & & \StringLRX{\hf} & & \StringL{\hz} & & {\hs} & \\
 & {\hz} & & \StringRX{\hf} & & \StringLR{\hz^a} & & {\hz} & & {\hs} \\
 {\hz} & & \StringRX{\hf} & & \StringRX{\hf} & & \StringL{\hz} & & {\hz} & \\
 & \StringRX{\hf} & & \StringRX{\hf} & & \StringR{\hz} & & {\hz} & \\
 \StringR{\hz} & & \StringRX{\hf} & & {\hz} & & \StringL{\hz} & & {\hz} & \\
 & \StringLR{\hd} & & {\hz} & & \StringR{\hz} & & {\hz} & \\
 \StringR{\hz} & & \StringL{\hd} & & {\hz} & & \StringL{\hz} & & {\hz} & \\
 & \StringLR{\hd} & & {\hz} & & \StringR{\hz} & & {\hz} & \\
 \StringR{\hz} & & \StringLX{\ho} & & {\hz} & & \StringL{\hz} & & {\hz} & \\
 \ai \ajs & \StringLX{\ho} & & \StringLX{\ho} & & \StringR{\hz} & & {\hz} & \\
 {\hs} & & \StringLX{\ho} & & \StringLX{\ho} & & \StringL{\hz} & & {\hz} & \\
 & {\hs} & & \StringLX{\ho} & & \StringLR{\hz^b} & & {\hz} & \\
 {\hs} & & {\hs} & & \StringLR{\ho^c} & & {\hz} & & {\hz} & \\
 & {\hs} & & {\hs} & & {\ho} & & {\hz} & \\
} & \parbox{0.1in}{\vspace{1.8in}$\rightarrow$} & 
\heap { 
 & {\hz} & & {\hz} & & \StringLRX{\hf} & & {\hs} & & {\hs} & \\
 {\hz} & & {\hz} & & \StringLRX{\hf} & & \StringL{\hz} & & {\hs} & \\
 & {\hz} & & \StringRX{\hf} & & \StringLRX{\hf^a} & & {\hz} & & {\hs} \\
 {\hz} & & \StringRX{\hf} & & \StringRX{\hf} & & \StringL{\hz} & & {\hz} & \\
 & \StringRX{\hf} & & \StringRX{\hf} & & \StringR{\hz} & & {\hz} & \\
 \StringR{\hz} & & \StringRX{\hf} & & {\hz} & & \StringL{\hz} & & {\hz} & \\
 & \StringLR{\hd} & & {\hz} & & \StringR{\hz} & & {\hz} & \\
 \StringR{\hz} & & \StringL{\hd} & & {\hz} & & \StringL{\hz} & & {\hz} & \\
 & \StringLR{\hd} & & {\hz} & & \StringR{\hz} & & {\hz} & \\
 \StringR{\hz} & & \StringLX{\ho} & & {\hz} & & \StringL{\hz} & & {\hz} & \\
 \ai \ajs & \StringLX{\ho} & & \StringLX{\ho} & & \StringR{\hz} & & {\hz} & \\
 {\hs} & & \StringLX{\ho} & & \StringLX{\ho} & & \StringL{\hz} & & {\hz} & \\
 & {\hs} & & \StringLX{\ho} & & \StringLRX{\ho^b} & & {\hz} & \\
 {\hs} & & {\hs} & & \StringLR{\ho^c} & & {\hz} & & {\hz} & \\
 & {\hs} & & {\hs} & & {\ho} & & {\hz} & \\
}
\end{tabular}
\end{center}
\caption{Case (3) of the proof with $r = 0$}\label{f:c3r0}
\end{figure}

Next, observe that these moves do not change the defect status of any other
entries in the heap.  The only entries that could conceivably change must use
the strings of $b$.  If there is a defect $c$ whose right string becomes the
left string of $b$ in $\s(\lambda^\dag)$, then $c$ must be a one-defect lying in
the same segment as $b$.  
Hence, the strings of $c$ must cross above the entry $a$ in the heap,
so $c$ remains a defect in $\s(\widetilde{\lambda^\dag})$.  Similarly, if there
is a defect $c$ whose right string becomes the right string of $b$ in
$\s(\widetilde{\lambda^\dag})$, then $c$ must be a one-defect lying in the same
segment as $b$, so the strings of $c$ cross above $a$, and $c$ must have been a
defect in $\s(\lambda^\dag)$.  The only other defects on the path of the strings
of $a$ and $b$ are the zero-defects in region $\mathfrak{1}$ that lie on the
left string of $b$.  It is straightforward to see that these defects are
preserved by the string moves and their inverse as well.  Hence, the defect
status of the other entries in the heap remain unchanged.

{\bf Case (3), Subcase $\mathbf{r > 0}$. }
In this case, let $b$ be the entry immediately southeast of the highest
zero-defect $a$ in region $\mathfrak{2}$ or $\mathfrak{3}$ on valley diagonal
$i$.  By the induction hypothesis, we have that $\s(\widetilde{\lambda^\dag})$ is
obtained from $\s(\lambda^\dag)$ by changing $b$ from a plain-one to a zero-defect,
changing $a$ from a zero-defect to a one-defect, adding an additional plain-one at
the end of $\nu_i$ in region $\mathfrak{1}$, and adding an additional plain-one
at the end of $\eta_i$ in region $\mathfrak{3}$.  To accomplish this, we will
use two string moves.

Using similar reasoning as in the cases above, observe that the strings of $b$
meet at the plain-zero $c$ lying just after the last entry in row $i$ of $\eta$
in region $\mathfrak{3}$.  We perform a string move on $b$ and $c$.  This
destroys the defect status of $a$, but the strings of $a$ will now meet at the
plain-zero $d$ that occurs just after the last entry of row $i$ of $\nu$ in
region $\mathfrak{1}$.  Hence, we can perform a string move between $a$ and $d$
to recover the defect status of $a$, and this changes $a$ into a one-defect.
This is illustrated in Figure~\ref{f:c3}.

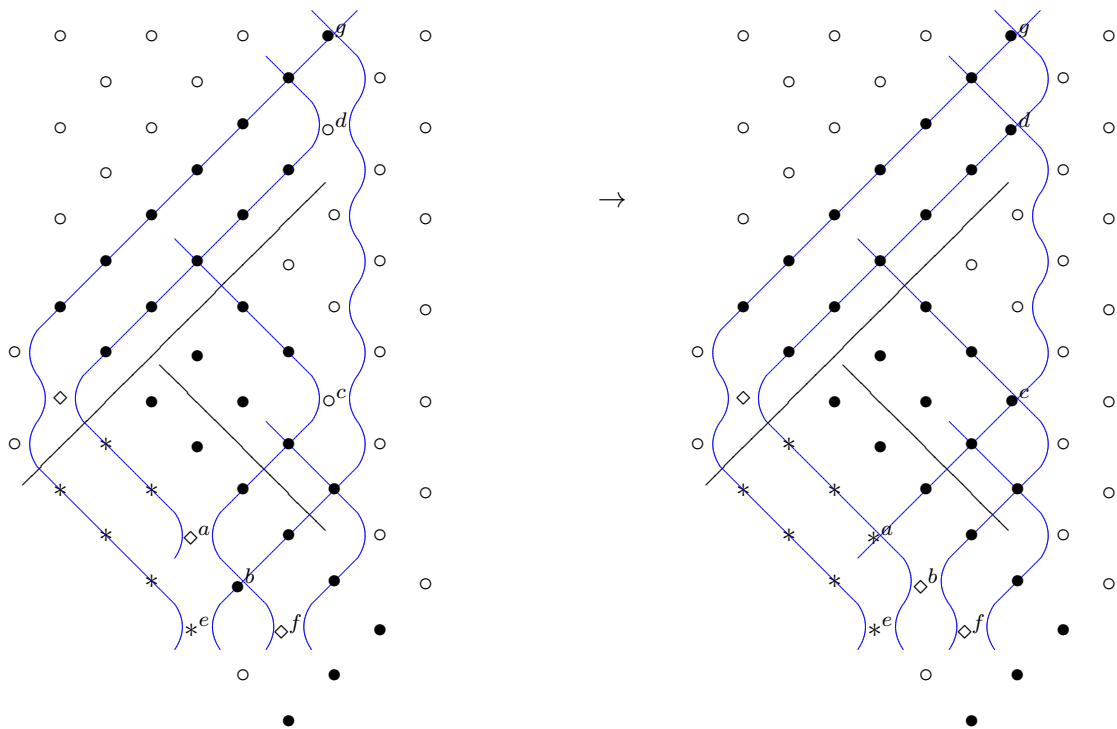
\begin{figure}[h]
\begin{center}
\begin{tabular}{ccccc}
\xymatrix @=-9pt @! {
 {\hs} & & {\hz} & & {\hz} & & {\hz} & & \StringLRX{\hf^g} & & {\hz} \\
 & {\hs} & & {\hz} & & {\hz} & & \StringLRX{\hf} & & \StringL{\hz} & & {\hs} & & \\
 {\hs} & & {\hz} & & {\hz} & & \StringRX{\hf} & & \StringLR{\hz^d} & & {\hz} \\
 & {\hs} & & {\hz} & & \StringRX{\hf} & & \StringRX{\hf} & & \StringL{\hz} & \\
 {\hs} & & {\hz} & & \StringRX{\hf} & & \StringRX{\hf} & & \StringR{\hz} & & {\hz} & \\
 & {\hs} & & \StringRX{\hf} & & \StringLRX{\hf} & & {\hz} & & \StringL{\hz} & & {\hs} \\
 {\hs} & & \StringRX{\hf} & & \StringRX{\hf} & & \StringLX{\hf} & & \StringR{\hz} & & {\hz} \\
 & \StringR{\hz} & & \StringRX{\hf} &\aj& {\hf} & & \StringLX{\hf} & & \StringL{\hz} & \\
 {\hs} & & \StringLR{\hd} & & {\hf} & & {\hf} & & \StringLR{\hz^c} & & {\hz} & \\
 & \StringR{\hz} & & \StringLX{\ho} & & {\hf} & & \StringLRX{\hf} & & \StringL{\hz} & \\
 {\hs} &\ai& \StringLX{\ho} & & \StringLX{\ho} & & \StringRX{\hf} & & \StringLRX{\hf} & & {\hz} & \\
 & {\hs} & & \StringLX{\ho} & & \StringLR{\hd^a} & & \StringRX{\hf} & & \StringL{\hz} & \\
 {\hs} & & {\hs} & & \StringLX{\ho} & & \StringLRX{\hf^b} & & \StringRX{\hf} & & {\hz} & \\
 & {\hs} & & {\hs} & & \StringLR{\ho^e} & & \StringLR{\hd^f} & & {\hf} & \\
 {\hs} & & {\hs} & & {\hs} & & {\hz} & & {\hf} & & {\hs} & \\
 & {\hs} & & {\hs} & & {\hs} & & {\hf} & & {\hs} & \\
} & \parbox{0.1in}{\vspace{1.8in}$\rightarrow$} & 
\xymatrix @=-9pt @! {
 {\hs} & & {\hz} & & {\hz} & & {\hz} & & \StringLRX{\hf^g} & & {\hz} \\
 & {\hs} & & {\hz} & & {\hz} & & \StringLRX{\hf} & & \StringL{\hz} & & {\hs} & & \\
 {\hs} & & {\hz} & & {\hz} & & \StringRX{\hf} & & \StringLRX{\hf^d} & & {\hz} \\
 & {\hs} & & {\hz} & & \StringRX{\hf} & & \StringRX{\hf} & & \StringL{\hz} & \\
 {\hs} & & {\hz} & & \StringRX{\hf} & & \StringRX{\hf} & & \StringR{\hz} & & {\hz} & \\
 & {\hs} & & \StringRX{\hf} & & \StringLRX{\hf} & & {\hz} & & \StringL{\hz} & & {\hs} \\
 {\hs} & & \StringRX{\hf} & & \StringRX{\hf} & & \StringLX{\hf} & & \StringR{\hz} & & {\hz} \\
 & \StringR{\hz} & & \StringRX{\hf} &\aj& {\hf} & & \StringLX{\hf} & & \StringL{\hz} & \\
 {\hs} & & \StringLR{\hd} & & {\hf} & & {\hf} & & \StringLRX{\hf^c} & & {\hz} & \\
 & \StringR{\hz} & & \StringLX{\ho} & & {\hf} & & \StringLRX{\hf} & & \StringL{\hz} & \\
 {\hs} &\ai& \StringLX{\ho} & & \StringLX{\ho} & & \StringRX{\hf} & & \StringLRX{\hf} & & {\hz} & \\
 & {\hs} & & \StringLX{\ho} & & \StringLRX{\ho^a} & & \StringRX{\hf} & & \StringL{\hz} & \\
 {\hs} & & {\hs} & & \StringLX{\ho} & & \StringLR{\hd^b} & & \StringRX{\hf} & & {\hz} & \\
 & {\hs} & & {\hs} & & \StringLR{\ho^e} & & \StringLR{\hd^f} & & {\hf} & \\
 {\hs} & & {\hs} & & {\hs} & & {\hz} & & {\hf} & & {\hs} & \\
 & {\hs} & & {\hs} & & {\hs} & & {\hf} & & {\hs} & \\
} 
\end{tabular}
\end{center}
\caption{Case (3) of the proof with $r > 0$}\label{f:c3}
\end{figure}

Next, observe that these moves do not change the defect status of any other
entries in the heap.  The only entries that could conceivably change must use
the strings of $b$ or $a$.  If there is a defect $f$ whose left string becomes
the right string of $a$ in $\s(\lambda^\dag)$, then the strings of $f$ cross below
$c$ in the heap.  After the string move, the left string of $f$ becomes
the right string of $b$ which passes through $c$, so the strings of $f$ will
still cross, and $f$ remains a defect in $\s(\widetilde{\lambda^\dag})$.  
Similarly, if there is a defect $f$ whose left string becomes the right string
of $b$ in $\s(\widetilde{\lambda^\dag})$, then $f$ must have been a defect in
$\s(\lambda^\dag)$.

If there is a defect $e$ whose right string becomes the right string of $b$ in
$\s(\lambda^\dag)$, then the strings of $e$ must cross above the entry $g$
directly above $d$ in the heap, so $e$ remains a defect in
$\s(\widetilde{\lambda^\dag})$.  If there is a defect $e$ whose right string
becomes the left string of $b$ in $\s(\widetilde{\lambda^\dag})$, then $e$ must
be a one-defect lying in the same segment as $b$.  This follows because, by
induction and Lemma~\ref{l:distinct_rows}, the only mask-value 0 entries in
region $\mathfrak{2}$ are zero-defects whose strings cross and remain within
their own segment.  Hence, the left string of $e$ travels along row $i-1$ of
$\nu$ in region $\mathfrak{1}$, and so the strings of $e$ cross above the entry
$g$ directly above $d$ in the heap.  Therefore, $e$ must have been a
defect in $\s(\lambda^\dag)$.

The arguments are similar for the case of a defect whose right string becomes
the left string of $a$ in $\s(\lambda^\dag)$, and for the case of a defect whose
right string becomes the right string of $a$ in $\s(\widetilde{\lambda^\dag})$.
The only other defects on the path of the strings of $a$ and $b$ are the
zero-defects in region $\mathfrak{1}$ that lie on the left string of $a$.  It is
straightforward to see that these defects are preserved by the string moves and
their inverse as well.  Hence, the defect status of the other entries in the
heap remain unchanged.

We have covered all of the cases, so the proof is complete.
\end{proof}

Lemmas~\ref{l:distinct_entries} and \ref{l:good_insertion} allow us to construct $\s(t)$
for each $t\in A_w$ by working one segment at a time.  Now we define $P(t)$ to
be the defect positions in this mask.  Then the set $F_{\w}^{P(t)}$ is a lower
order ideal containing $x(t)$ by Lemma~\ref{l:fwp}.  Let $\mathcal{E}(t)$ be the
set of masks
$$\mathcal{E}(t)=\{\sigma\in\mathcal{F}^{P(t)}_\w\mid \w^\sigma\leq x(t)\}.$$
The masks in $\mathcal{E}(t)$ precisely encode the terms of $q^{|t| +
\frac{1}{2} \ell(x(t))} B_{x(t)}'$.

\begin{lemma}\label{l:injective}
If $t$ and $t'$ are two different edge-labelings in $A_w$, then $P(t)\neq
P(t')$.  In particular, $\mathcal{E}(t)$ and $\mathcal{E}(t')$ are disjoint.
\end{lemma}
\begin{proof}
In summary, our construction associated a set of masks of the form
$\mathcal{F}_{\w}^{P(t)}$ to each edge-labeling $t \in A_w$ via
\[ t \ \ \ \ \text{-- \tiny zero out the leaves of $t$} \rightarrow \ \ \ \ 
\gamma(t) \ \ \ \ \text{-- \tiny add defects according to edge-labels of $t$} \rightarrow \ \ \ \  
\sigma(t) \ \ \ \ \text{-- \tiny take defect set of $\sigma(t)$} \rightarrow \ \ \ \ P(t). \]
Suppose we are given a defect set $P(t)$ arising from our construction.  To
prove that the construction gives an injection, it suffices to show how to
recover the edge-labels of $t$.  As a preliminary step, we must first recover
$\gamma(t)$ from $P(t)$.  

Begin by labeling the entries in the heap of $\w$ according their defect status
from $P(t)$.  Recall that the ridgeline of the heap of $\w$ corresponds to the
shape $T(w)$ of the tree, and is independent of the edge-labeling $t$.  In our
construction, each valley column $v$ of the heap of $\w$ corresponds to a leaf
of the tree.  Let $e_v$ denote the label of the leaf $v$ in $t$.  Then our
construction always places a vertical segment of either $e_v - 1$ or $e_v$
defects in column $v + 1$ as a result of processing $\lambda(v,t)_1' = e_v -
1$.

More precisely, we always place $e_v - 1$ defects in region $\mathfrak{1}$,
and sometimes one additional defect is placed in column $v+1$ on the lowest
valley diagonal in region $\mathfrak{2}$ or $\mathfrak{3}$.  The decision of
whether we place this additional defect depends on the relationship between
$\lambda(v,t)^\dag_1$ and $r(v)$.

In any case, to recover $\gamma(t)$, it suffices to determine the leaf values
for those valleys where the number of plain-zeros $p(v)$ in column $v$ is equal
to the leaf value $e_v$.  We call such valleys {\bf primitive}.  Once we
determine the leaf values $e_v$ for the primitive valleys $v$, we can construct
$\gamma(t)$ by ``zeroing out'' the top $e_v$ entries in column $v$ for each
primitive valley $v$. 

Next, observe that we can determine which valleys are primitive from $P(t)$.  If
a valley $v$ is primitive, then it will have the maximal number of defects in
region $\mathfrak{1}$ of column $v+1$, which means that the segment of defects
will come all the way up to the entry just below the ridgeline.

Suppose $v$ is a primitive valley, and suppose that from $P(t)$ we find a
segment of defects of length $k$ in column $v+1$.  Then the preimage
$\sigma(t)$ of $P(t)$ has one of the two forms shown in Figure~\ref{f:inj}.
In Case (a), the lowest defect $d$ in column $v+1$ is not part of region
$\mathfrak{1}$, while in Case (b), the lowest defect $d$ in column $v+1$ is
part of region $\mathfrak{1}$.

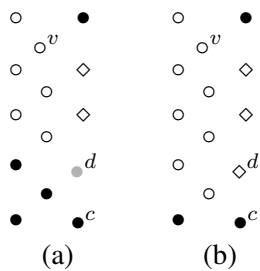
\begin{figure}[h]
\begin{center}
\begin{tabular}{cc}
\xymatrix @-1.7pc@M=0pt {
{\hs} & & {\hz} & & {\hf} & & {\hs} & \\
& {\hs} & & {\hz^v} & & {\hs} & & {\hs} \\
{\hs} & & {\hz} & & {\hd} & & {\hs} & \\
& {\hs} & & {\hz} & & {\hs} & & {\hs} \\
{\hs} & & {\hz} & & {\hd} & & {\hs} & \\
& {\hs} & & {\hz} & & {\hs} & & {\hs} \\
{\hs} & & {\hf} & & {\hb^d} & & {\hs} & \\
& {\hs} & & {\hf} & & {\hs} & & {\hs} \\
{\hs} & & {\hf} & & {\hf^c} & & {\hs} & \\
 & & & \\
} &
\xymatrix @-1.7pc@M=0pt {
{\hs} & & {\hz} & & {\hf} & & {\hs} & \\
& {\hs} & & {\hz^v} & & {\hs} & & {\hs} \\
{\hs} & & {\hz} & & {\hd} & & {\hs} & \\
& {\hs} & & {\hz} & & {\hs} & & {\hs} \\
{\hs} & & {\hz} & & {\hd} & & {\hs} & \\
& {\hs} & & {\hz} & & {\hs} & & {\hs} \\
{\hs} & & {\hz} & & {\hd^d} & & {\hs} & \\
& {\hs} & & {\hz} & & {\hs} & & {\hs} \\
{\hs} & & {\hf} & & {\hf^c} & & {\hs} & \\
 & & & \\
} \\
(a) & (b) \\
\end{tabular}
\end{center}
\caption{Cases for the preimage of $P(t)$}\label{f:inj}
\end{figure}

Observe that in Case (a), any defects that lie on the diagonal extending to the
southeast below $c$ must belong to a segment for some valley $u$ that lies to
the left of $v$.  Moreover, the diagonal containing $c$ must be the rightmost
diagonal in the segment associated to $u$.  But it follows from our construction
that the rightmost valley diagonal in region $\mathfrak{2}$ for any valley $u$
can only contain defects that lie in columns directly below region
$\mathfrak{3}$ for the valley $u$.  In particular, there are no defects lying
below $c$ on the diagonal extending to the southeast of $c$.

On the other hand, in Case (b), the diagonal containing $c$ must belong to the
segment associated to $v$.  Hence, there must exist at least one defect lying
below $c$ on the diagonal extending to the southeast of $c$ as otherwise
$\lambda(v,t)^\dag$ would not be a valid partition.

Therefore, we have that $e_v = k$ if there do not exist any defects on the
valley diagonal extending southeast from $c$, and otherwise $e_v = k + 1$.
Hence, we can determine $e_v$ from $P(t)$ and thereby recover $\gamma(t)$
from $P(t)$.  Then we can determine the correct decomposition of the heap of
$\w$ into segments and regions.  Finally, we can simply count the number of
defects on each valley diagonal in the segment associated to each valley $v$ to
recover $\lambda(v,t)^\dag$ and thus recover the edge labels of the tree $t$.
\end{proof}

Finally, Theorem~\ref{t:cog_bp} together with Lemmas~\ref{l:good_insertion} and
\ref{l:injective} imply the following result.

\begin{theorem}\label{t:main}
If $w$ is cograssmannian, then 
\[ C_{w}' = q^{-\frac{1}{2} \ell(w)} \sum_{\s \in \mathcal{E}_{\w}} q^{d(\s)}
\text{\ \ \ where \ \ \ } \mathcal{E}_{\w} = \bigcup_{t \in A_w} \bigcup_{\stackrel{\s \in \mathcal{F}_{\w}^{P(t)}}{\w^{\s} \leq x(t)} } \{ \s \}. \]
\end{theorem}

\bigskip
\section{Deodhar's model and Bott--Samelson resolutions}\label{s:b-s}

The aim of this section is to describe how Deodhar's model relates to
Bott--Samelson resolutions of Schubert varieties.  Indeed, it is possible to
deduce Deodhar's model from the Decomposition Theorem on intersection
cohomology applied to Bott--Samelson resolutions.  This connection is well
known to the appropriate experts and was hinted at in Deodhar's paper, but as
far as we know has never been carefully explained in print.

\subsection{Grassmannians, flag varieties, and Schubert varieties}

The {\bf Grassmannian} $\Gr(k,n)$ is an algebraic variety whose points
correspond to subspaces $V\subseteq\mathbb{C}^n$ of dimension $k$.  Group
elements $g\in\mathrm{GL}_n(\mathbb{C})$ send a subspace $V$ to another
subspace $gV$ of the same dimension, so $\mathrm{GL}_n(\mathbb{C})$ acts on
$\Gr(k,n)$.  We fix an ordered basis $e_1,\ldots,e_n$ of $\mathbb{C}^n$; this
implicitly fixes a Borel subgroup $B$, the group of upper-triangular matrices
with respect to this basis, and a torus $T\cong(\mathbb{C}^*)^n$ of diagonal
matrices with respect to this basis.  Note that the coordinate subspaces,
which are the $\binom{n}{k}$ subspaces spanned by $k$ of the $n$ vectors
$e_1,\ldots,e_n$, are precisely the subspaces which are fixed by every element
in $T$.

For the purposes of this paper, we define the {\bf flag variety}
$\mathcal{F}_n$ as a subvariety of
$\Gr(1,n)\times\Gr(2,n)\times\cdots\times\Gr(n-1,n)$.  A point in
$\Gr(1,n)\times\Gr(2,n)\times\cdots\times\Gr(n-1,n)$ corresponds to a sequence
of subspaces $F_1,\ldots,F_{n-1}$ with $\dim F_i=i$; we denote the point
corresponding to this sequence by $[F_1,\ldots,F_{n-1}]$.  A point
$[F_1,\ldots,F_{n-1}]$ is in $\mathcal{F}_n$ if $F_i\subseteq F_{i+1}$ for all
$i$.  If $[F_1,\ldots,F_{n-1}]\in\mathcal{F}_n$, then the collection of
subspaces $F_1\subset\cdots\subset F_{n-1}$ is called a {\bf flag}.  The torus
$T$ acts diagonally on $\Gr(1,n)\times\cdots\times\Gr(n-1,n)$ and preserves
inclusion relations on subspaces; hence it acts on $\mathcal{F}_n$.  The
$T$-fixed points correspond to flags of coordinate subspaces.  It is natural
to bijectively correspond permutations $w\in S_n$ with $T$-fixed points by
letting $$p_w=[\Span(e_{w(1)}),\Span(e_{w(1)},e_{w(2)}),
  \ldots,\Span(e_{w(1)},\ldots,e_{w(n-1)})].$$ We define $E_i$ by
$E_i=\Span(e_1,\ldots,e_i)$, so $p_\id=[E_1,\ldots,E_{n-1}]$.

Let $w\in S_n$ be a permutation.  We define the {\bf Schubert cell} $C_w$ as
the following cell in $\mathcal{F}_n$.  For all $i,j$ with $1\leq i,j\leq n$,
let $r_{ij}(w)$ be defined by $$r_{ij}(w)=\#\{k\mid k\leq j, w(k)\leq i\}.$$
Now a point $[F_1,\ldots,F_{n-1}]\in\mathcal{F}_n$ is in $C_w$ if
$\dim(F_j\cap E_i)=r_{ij}(w)$ for all $i,j$.  The Schubert cell $C_w$ is
isomorphic to affine space $\mathbb{C}^{\ell(w)}$.  The {\bf Schubert variety}
$X_w$ is the closure of the Schubert cell $C_w$, and a point
$[F_1,\ldots,F_{n-1}]\in X_w$ if and only if $\dim(F_j\cap E_i)\geq r_{ij}(w)$
for all $i,j$.  It is an alternative definition of Bruhat order that $v\leq w$
if and only if $r_{ij}(v)\geq r_{ij}(w)$ for all $i,j$, so $X_w=\bigcup_{v\leq
  w} C_v$.  Note that $p_w$ is the only $T$-fixed point in $C_w$ and
furthermore that $C_w$ is precisely the orbit of $p_w$ under the action of the
group $B$ of upper triangular matrices.

\subsection{The Bott--Samelson resolution}

Instead of the original definition~\cite{BS58} of the Bott--Samelson
resolution, we use the alternative definition given by Magyar \cite{magyar};
Perrin \cite{perrin} later suggested the connection between this definition and
heaps.  Magyar's definition allows us to easily think of a point in the
Bott--Samelson variety as corresponding to a configuration of vector subspaces
of $\mathbb{C}^n$, just as a point in a Schubert variety corresponds to a flag
in $\mathbb{C}^n$.  Furthermore, from this viewpoint, $(\mathbb{C}^*)^n$ acts
on the Bott--Samelson variety by moving the vector subspaces of $\mathbb{C}^n$
in the configuration associated to a point.  This definition makes the geometry
more concrete at the expense of making some statements harder to prove.

Given a reduced expression $\w=\w_{1} \w_{2}\cdots \w_{\ell}$ for $w\in S_n$,
we define $d_j$ for each $j$, $1\leq j\leq \ell$, by requiring that
$s_{d_j}=\w_j$.  Now we define the Bott--Samelson variety $Z_\w$ as a
subvariety of $\Gr(d_1,n)\times\cdots\times\Gr(d_\ell,n)$ as follows.

We have fixed a standard ordered basis $e_1,\ldots,e_n$ for $\mathbb{C}^n$,
which implicitly fixes a standard flag $E_1 \subset \cdots \subset E_{n-1}$.
Now for each $j$ define $\before(j)$ to be the greatest index such that
$\before(j)<j$ and $\w_{\before(j)}=s_{d_j-1}$; this is the index of the point
in the heap directly NW of (the point corresponding to) $j$.
Similarly define $\after(j)$ to be the greatest index such that $\after(j)<j$
and $\w_{\after(j)}=s_{d_j+1}$; this is the index of the point in the 
heap directly NE of $j$.  Note that there could be no indices satisfying the
required properties; in this case we leave $\before(j)$ and or $\after(j)$
undefined.

A point in $\Gr(d_1,n)\times\cdots\times\Gr(d_\ell,n)$ corresponds to a
sequence of subspaces $V_1,\ldots,V_\ell$ of $\mathbb{C}^n$.  We denote this
point by $[V_1,\ldots,V_\ell]$.  Then the {\bf Bott--Samelson variety} $Z_\w$
is defined by
$$Z_{\w}=\{[V_1,\ldots,V_\ell]\mid V_{\before(j)}\subset V_j \subset
V_{\after(j)} \forall j\}.$$ If $\before(j)$ is undefined, then we require
instead that $E_{d_j-1}\subset V_j$, and if $\after(j)$ is undefined, we
require that $V_j\subseteq E_{d_j+1}$.

Note that the covering relations in the heap poset of $\w$ are precisely that
$\before(j)$ covers $j$ and $\after(j)$ covers $j$.  In terms of the heap
diagram, we have a subspace $V_j$ living on each point of the heap.  The
dimension of $V_j$ is given by the index of the column it lives in.  The
inclusion relations state that each subspace must contain those in the same
diagonal to the NW and be contained in those in the same diagonal to the NE.

Now we describe a map $\pi_\w: Z_\w\rightarrow X_w$.  For each $d$, $1\leq
d\leq n-1$, let $\last(d)$ be the index of the last occurence of $s_d$ in $\w$.
We leave $\last(d)$ undefined if $s_d$ does not occur in $\w$ (and hence not in
any reduced word for $w$).  The map $\pi_\w$ is defined by
$$\pi_\w([V_1,\ldots,V_\ell])=[V_{\last(1)},\ldots,V_{\last(n-1)}],$$ with the
convention that $V_{\last(d)}=E_d$ if $\last(d)$ is not defined.  The map
$\pi_\w$ is known as the {\bf Bott--Samelson resolution}.  In terms of the
heap, this map forgets every subspace except for the bottom one in
each column.

It follows from the original definition of the Bott--Samelson~\cite{BS58} and
Magyar's proof of equivalence~\cite{magyar} that the image of this map is
indeed $X_w$.  This fact can also be deduced from the fixed point analysis of
the next section and the $B$-equivariance of $\pi_\w$.

\begin{example}
Let $w=[321]\in S_3$, and $\w=s_1s_2s_1$.  A point $p\in X_w$ corresponds to a
flag $F_1\subset F_2$ of subspaces of $\mathbb{C}^3$ with
$\dim(F_j)=j$.  A point $q\in Z_\w$ corresponds to a configuration
$V_1,V_2,V_3$ of subspaces of $\mathbb{C}^3$, respectively of dimensions
$1,2,1$, satisfying the conditions $V_1\subset E_2$, $V_1\subset V_2$, and
$V_3\subset V_2$.  The map $\pi_\w$ sends $q$ to the point corresponding to the
flag $V_3\subset V_2$.  
\end{example}

\begin{example}
Let $w=[3412]\in S_4$, and $\w=s_2s_1s_3s_2$.  A point $p\in X_w$ corresponds to
a flag $F_1\subset F_2\subset F_3$ in $\mathbb{C}^4$ such that $F_1\subset
E_3$ and $E_1\subset F_3$.  A point $q\in Z_\w$ corresponds to subspaces
$V_1,V_2,V_3,V_4$, of dimensions $2,1,3,2$, satisfying $E_1\subset
V_1\subset E_3$, $V_2\subset V_1\subset V_3$, and $V_2\subset V_4\subset V_3$.
The point $\pi_w(q)$ corresponds to the configuration $V_2\subset V_4\subset
V_3$.
\end{example}

\subsection{Fixed points on Bott--Samelson varieties and Deodhar masks}

Suppose we are given a smooth complex projective variety $X$ with a
$T=\mathbb{C}^*$ action with finitely many $T$-fixed points.  We denote this
action by $t \cdot x$ for $t \in T$ and $x \in X$.  Bialynicki-Birula
\cite{b-b-cells,b-b-cells2} showed that, in this situation, $X$ has a cell
decomposition $X=\bigsqcup_{x\in X^T} C_x$, where $X^T$ is the set of
$T$-fixed points of $X$, and the cell $C_x$ is defined by $$C_x=\{y\in X\mid
\lim_{t\rightarrow0} t\cdot y=x\}.$$ In particular, the homology $H_*(X)$ of
$X$ has a basis given by the classes $\{[\overline{C_x}]\mid x\in X^T\}$ of
the closures of the cells $C_x$.  If $\dim(C_x)=i$, then $[\overline{C_x}]\in
H_{2i}(X)$.

Before describing the cells, we describe the fixed points and give a bijection
between masks and $T$-fixed points on Bott--Samelson varieties.  Our bijection
agrees with the previously mentioned bijection between permutations and
$T$-fixed points on Schubert varieties.  In the next section we will also
point out that this bijection takes the defect statistic to the dimension of the
cell canonically associated with the $T$-fixed point.  The results of this
section are implicit in the work of Magyar \cite{magyar}.

Let $\s$ be a mask on $\w = \w_1 \cdots \w_{\ell}$, and recall that we define
$d_j$ by $s_{d_j} = \w_j$.  We define the point $p_\s$ to be the $T$-fixed
point
$$p_\s=[V_1,\ldots,V_\ell]\,\mbox{where}\,V_j=\Span(e_{\w^{\s[j]}(1)},\ldots,e_{\w^{\s[j]}(d_j)}).$$
From the string diagram of a heap, one can read off $V_j$ for a heap point
$\w_j$.  Label the strings $1,\ldots, n$ consecutively as they appear at the
top of the heap.  The subspace $V_j$ is then the span of the basis
vectors $e_k$ for all labels $k$ of strings to the left of $\w_j$ immediately
below $\w_j$ in the heap.  The following propositions show that this
correspondence gives a bijection between masks and $T$-fixed points.

\begin{proposition}
The $T$-fixed point $p_\s$ is actually in $Z_\w$.  In other words, the defined
collection of subspaces satisfies the inclusion conditions for configurations
corresponding to points in $Z_\w$.
\end{proposition}

\begin{proof}
Suppose $p_s=[V_1,\ldots,V_\ell]$.  We need to show that
$V_{\before(j)}\subset V_j\subset V_{\after(j)}$ for all $j$, $1\leq j\leq
\ell$.  Since there is no occurence of $s_{d_j-1}$ in $\w$ between indices
$\before(j)$ and $j$,
\begin{align*}
\{\w^{\s[\before(j)]}(1),\ldots,\w^{\s[\before(j)]}(d_j-1)\}
& =\{\w^{\s[j]}(1),\ldots,\w^{\s[j]}(d_j-1)\} \\
& \subset\{\w^{\s[j]}(1),\ldots,\w^{\s[j]}(d_j-1),\w^{\s[j]}(d_j)\},
\end{align*}
so $V_{\before(j)}\subset V_j$.  If $\before(j)$ is undefined, then
\begin{align*}
\{1,\ldots,d_j-1\}
& =\{\w^{\s[j]}(1),\ldots,\w^{\s[j]}(d_j-1)\} \\
& \subset\{\w^{\s[j]}(1),\ldots,\w^{\s[j]}(d_j-1),\w^{\s[j]}(d_j)\},
\end{align*}
so $E_{d_j-1}\subset V_j$.  Therefore $V_{\before(j)}\subset V_j$ (or
$E_{d_j-1}\subset V_j$ if $\before(j)$ is undefined).  A similar argument
shows that $V_j\subset V_{\after(j)}$ (or $V_j\subset V_{d_j+1}$ if
$\after(j)$ is undefined).
\end{proof}

\begin{proposition}
Every $T$-fixed point on $Z_\w$ is $p_\s$ for some mask $\s$ on $\w$.
\end{proposition}

\begin{proof}
Given a $T$-fixed point $[V_1,\ldots,V_\ell]$ in $Z_\w$ (so each $V_j$ is a
coordinate subspace), we construct the mask $\s$.

Assume by induction that $\s_1,\ldots,\s_{j-1}$ have already been determined
correctly from $V_1,\ldots,V_{j-1}$; this means that
$V_k=\Span(e_{\w^{\s[k]}(1)},\ldots,e_{\w^{\s[k]}(d_k)})$ for all $k<j$.
Since there are no occurences of $s_{d_j-1}$ between indices $\before(j)$ and
$j-1$,
$$\{\w^{\s[\before(j)]}(1),\ldots,\w^{\s[\before(j)]}(d_j-1)\}=\{\w^{\s[j-1]}(1),\ldots,\w^{\s[j-1]}(d_j-1)\}.$$
Similarly there are no occurences of $s_{d_j+1}$ between indices $\after(j)$ and
$j-1$, so
$$\{\w^{\s[\after(j)]}(1),\ldots,\w^{\s[\after(j)]}(d_j+1)\}=\{\w^{\s[j-1]}(1),\ldots,\w^{\s[j-1]}(d_j+1)\}.$$
Therefore, the requirement that $V_{\before(j)}\subset V_j \subset
V_{\after(j)}$ can be given explicitly as
$$\Span(e_{\w^{\s[j-1]}(1)},\ldots,e_{\w^{\s[j-1]}(d_j-1)})\subset
V_j\subset \Span(e_{\w^{\s[j-1]}(1)},\ldots,e_{\w^{\s[j-1]}(d_j+1)}).$$

As $V_j$ must be a coordinate subspace, it must be that
$$V_j=\Span(e_{\w^{\s[j-1]}(1)},\ldots,e_{\w^{\s[j-1]}(d_j-1)},e_{w^{\s[j-1]}(d_j)}),$$
or that
$$V_j=\Span(e_{\w^{\s[j-1]}(1)},\ldots,e_{\w^{\s[j-1]}(d_j-1)},e_{w^{\s[j-1]}(d_j+1)}).$$
In the first case, we let $\s_j=0$, and in the second, we let
$\s_j=1$.  In either case,
$$V_j=\Span(e_{\w^{\s[j]}(1)},\ldots,e_{\w^{\s[j]}(d_j)})$$ as desired.

The special cases where $\before(j)$ or $\after(j)$ is undefined are similar and left
to the reader.
\end{proof}

\begin{proposition}
If $\s^{(1)}\neq\s^{(2)}$ are two masks, then $p_{\s^{(1)}}\neq
p_{\s^{(2)}}$.
\end{proposition}

\begin{proof}
Let $p_{\s^{(1)}}=[V^{(1)}_1,\ldots,V^{(1)}_\ell]$ and
$p_{\s^{(2)}}=[V^{(2)}_1,\ldots,V^{(2)}_\ell]$.  Let $j$ be the first index
where $\s^{(1)}_j\neq\s^{(2)}_j$.  Assume without loss of generality that
$\s^{(1)}_j=0$ and $\s^{(2)}_j=1$.  Then
$$V^{(1)}_j=\Span(e_{\w^{\s^{(1)}[j-1]}(1)},\ldots,e_{\w^{\s^{(1)}[j-1]}(d_j-1)},e_{\w^{\s^{(1)}[j-1]}(d_j)})$$
while 
$$V^{(2)}_j=\Span(e_{\w^{\s^{(1)}[j-1]}(1)},\ldots,e_{\w^{\s^{(1)}[j-1]}(d_j-1)},e_{\w^{\s^{(1)}[j-1]}(d_j+1)}).$$
\end{proof}

Finally we show that this bijection is natural with respect to the map
$\pi_\w$.

\begin{proposition}\label{p:fixed_point_images}
For any mask $\s$, $\pi_\w(p_\s)=p_{\w^\s}.$
\end{proposition}

\begin{proof}
Suppose $p_\s=[V_1,\ldots,V_\ell]$ and $\pi_\w(p_\s)=[F_1,\ldots,F_{n-1}]$.
Since there are no occurences of $s_d$ in $\w$ after $\last(d)$, we have that
$$\{\w^{\s[\last(d)]}(1),\ldots,\w^{\sigma[\last(d)]}(d)\}=\{\w^{\s}(1),\ldots,\w^{\s}(d)\}.$$
Therefore, for each $d$, $1\leq d\leq
n-1$, $$F_d=V_{\last(d)}=\Span(e_{\w^{\s}(1)},\ldots,e_{\w^{\s}(d)}),$$
so $p_{\w^\s}=[F_1,\ldots, F_{n-1}]$ as claimed.
\end{proof}

\subsection{Cells of the Bott--Samelson variety}\label{sect:b-s-cells}

We now apply the Bialynicki-Birula theorem to $Z_\w$ to obtain a cell
decomposition for $Z_\w$ and hence a basis for $H_*(Z_\w)$.  Fix the
$\mathbb{C}^*$ action on $\mathbb{C}^n$ where $t\cdot e_i=t^{n-i}e_i$ for $t
\in \mathbb{C}^*$, extending linearly.  This induces a $\mathbb{C}^*$ action on
$\Gr(i,n)$ and therefore (diagonally) on $Z_\w$ since all inclusion relations
are preserved.  The cell $C_\s$ associated to the $T$-fixed point $p_\s$ is
then defined by $$C_\s=\{p\in Z_\w\mid \lim_{t\rightarrow0} t\cdot p=p_\s\}.$$

Given a subspace $V\subset\mathbb{C}^n$ having a basis (written with respect to
$(e_1,\ldots,e_n)$) as the rows of a matrix $M$, $\lim_{t\rightarrow0} t\cdot
V$ is the coordinate subspace spanned by the coordinate vectors corresponding
to the ``pivot columns'' of the ``right-to-left row echelon form'' of $M$.
This means that, instead of performing row reduction by the usual method of
starting with the leftmost column, trying it as a pivot, and moving rightwards
to find successive pivots, we start by trying the rightmost column as a pivot
and move leftwards to find successive pivots.  Therefore, we have the following
proposition.

\begin{proposition}\label{p:cell_pivots}
A point $p=[V_1,\ldots,V_\ell]\in Z_\w$ is in $C_\s$ if, for all $j$, the
right-to-left row echelon form of a matrix whose rows span $V_j$ has pivots
in columns $\w^{\s[j]}(1),\ldots,\w^{\s[j]}(d_j)$.
\end{proposition}

For an alternative definition of $C_\sigma$, let $r_{ij}(\s)=\#\{k\mid k<d_j,
\w^{\s[j]}(k)\leq i\}$.  A point in $Z_\w$ is in $C_\s$ if $\dim(V_j\cap
E_i)=r_{ij}(\s)$ for all $i,j$.  The only $T$-fixed point in $C_\s$ is $p_\s$.
However, unlike Schubert cells on flag varieties, the cell $C_\s$ need not be
a single $B$-orbit, and the closure of a cell might not be a union of cells
and might not include all points corresponding to subspaces satisfying
$\dim(V_j\cap E_i)\geq r_{ij}(\s)$.

We now introduce a second notation for masks that is more convenient for
describing the dimensions of cells.  Given a mask $\sigma$, let
$e(\sigma)=(e(\sigma)_1,\ldots,e(\sigma)_\ell)$ be a string of $+$'s and $-$'s
defined by 
$$ e(\sigma)_i =
\begin{cases}
- & \text{ if $\sigma$ has plain-zero or one-defect at $i$} \\
+ & \text{ if $\sigma$ has a zero-defect or plain-one at $i$.}
\end{cases} $$

\begin{proposition}
The mask $\sigma$ can be recovered from $e(\sigma)$.
\end{proposition}

\begin{proof}
Since the defect status of a position depends only on the mask to its left, we
can reconstruct $\sigma$ from $e(\sigma)$ position by position, starting from
the left.
\end{proof}

The entries of $e(\sigma)$ also have meaning in the heap.  A $+$
indicates that the string with the larger label exits to the left below the
heap point, and a $-$ indicates that the string with the smaller label exits to
the left.

\begin{proposition}\label{p:nplus}
The number of $+$'s in $e(\sigma)$ is equal to $\ell(\w^\sigma)+d(\sigma)$,
where $d(\sigma)$ is the number of defects.
\end{proposition}

\begin{proof}
We induct on the length of $\w$ and $\sigma$.  The proposition is clear for
words of length 0.  When we multiply on the right at a non-defect position,
then a 0 in $\sigma$ increases neither the length of $\w^\sigma$ nor the
number of $+$'s in $e(\sigma)$, while a 1 in $\sigma$ increases both the
length of $\w^\sigma$ and the number of $+$'s.  On the other hand, when we
multiply at a defect position, a 0 in $\sigma$ increases the number of defects
as well as the number of $+$'s, while a 1 in $\sigma$ adds one defect,
subtracts 1 from the length, and leaves the number of $+$'s unchanged.
\end{proof}

We introduced $e(\sigma)$ because, as we will see in the proof, it is the
natural way to compute the dimension.  Indeed, every $+$ corresponds to a
natural coordinate function on a cell.

\begin{proposition}\label{p:plusdim}
The dimension of $C_{\sigma}$ is the number of $+$'s in $e(\sigma)$.
\end{proposition}

\begin{proof}
Let $y=[V_1,\ldots,V_\ell]$ be a point in $C_{\sigma}$.  We show that, once
$V_1,\ldots, V_{j-1}$ are chosen, the subspace $V_j$ is already determined if
$e(\sigma)_j$ is a $-$, and there is a one-dimensional choice for $V_j$ if
$e(\sigma)_j$ is a $+$.  Recall that, given our choices, $V_j$ can be any
subspace with $V_{\before(j)}\subset V_j\subset V_{\after(j)}$.  Let
$M_j(\sigma)=\max\{\w^{\sigma[j-1]}(d_j),\w^{\sigma[j-1]}(d_j+1)\}$ and
$m_j(\sigma)=\min\{\w^{\sigma[j-1]}(d_j),\w^{\sigma[j-1]}(d_j+1)\}$; these are
the labels of the strings meeting at the heap point $j$.  Then by
Proposition~\ref{p:cell_pivots}, $V_{\after(j)}$ is spanned by
$V_{\before(j)}$, some vector $a_j=e_{m_j}+\sum_{k<m_j} c_ke_k$, and some
vector $A_j=e_{M_j}+\sum_{k<M_j} C_ke_k$.  (Note that $a_j$ and $A_j$ are not
canonically determined as we can modify them by adding vectors in
$V_{\before(j)}$, but these modifications amount only to a change of
coordinates and do not change the dimension.)  Now if $e(\sigma)_j$ is a
$-$, then the string with the smaller label, which is $m_j$, exits the mask
point $j$ to the left, so $V_j$ is spanned by $V_{\before(j)}$ and some vector
with leading coordinate $m_j$.  This vector must be $a_j$.  On the other hand,
if $e(\sigma)_j$ is a $+$, then the string with the larger label exits to the
left, so $V_j$ is spanned by $V_{\before(j)}$ and a vector with leading
coordinate $M_j$.  This vector must be $A_j+\alpha_j a_j$ for some
$\alpha_j\in\mathbb{C}$, giving a one-dimensional choice for $V_j$.
\end{proof}

\begin{example}
The following table describes the cells associated with all the masks for
$\w=s_1s_2s_1$.
\begin{tabular}{c|c|c|c|c|c}
$\sigma$ & $e(\sigma)$ & $e^\sigma$ & $d(\sigma)$ & $p_\sigma$ & $C_\sigma$ \\
\hline
$000$ & $---$ & $id$ & $0$
& $\langle e_1\rangle, \langle e_1, e_2\rangle, \langle e_1\rangle$
& $\langle e_1\rangle, \langle e_1, e_2\rangle, \langle e_1\rangle$
\\
$001$ & $--+$ & $s_1$ & $0$
& $\langle e_1\rangle, \langle e_1, e_2\rangle, \langle e_2\rangle$
& $\langle e_1\rangle, \langle e_1, e_2\rangle, \langle \alpha_3e_1+e_2\rangle$
\\
$010$ & $-+-$ & $s_2$ & 0
& $\langle e_1\rangle, \langle e_1, e_3\rangle, \langle e_1\rangle$
& $\langle e_1\rangle, \langle e_1, \alpha_2e_2+e_3\rangle, \langle e_1\rangle$
\\
$100$ & $+-+$ & $s_1$ & 1
& $\langle e_2\rangle, \langle e_1, e_2\rangle, \langle e_2\rangle$
& $\langle \alpha_1e_1+e_2\rangle, \langle e_1, e_2\rangle, \langle \alpha_3e_1+e_2\rangle$
\\
$101$ & $+--$ & $id$ & 1
& $\langle e_2\rangle, \langle e_1, e_2\rangle, \langle e_1\rangle$
& $\langle \alpha_1e_1+e_2\rangle, \langle e_1, e_2\rangle, \langle e_1\rangle$
\\
$110$ & $++-$ & $s_1s_2$ & 0
& $\langle e_2\rangle, \langle e_2, e_3\rangle, \langle e_2\rangle$
& $\langle \alpha_1e_1+e_2\rangle, \langle \alpha_1e_1+e_2, \alpha_2e_1+e_3\rangle, \langle \alpha_1e_1+e_2\rangle$
\\
$011$ & $-++$ & $s_2s_1$ & 0
& $\langle e_1\rangle, \langle e_1, e_3\rangle, \langle e_3\rangle$
& $\langle e_1\rangle, \langle e_1, \alpha_2e_2+e_3\rangle, \langle \alpha_3e_1+\alpha_2e_2+e_3\rangle$
\\
$111$ & $+++$ & $s_1s_2s_1$ & 0
& $\langle e_2\rangle, \langle e_2, e_3\rangle, \langle e_3\rangle$
& \parbox[t]{2.5in}{$\langle \alpha_1e_1+e_2\rangle, \langle \alpha_1e_1+e_2,
    \alpha_2e_1+e_3\rangle$,
$\langle \alpha_3(\alpha_1e_1+e_2)+(\alpha_2e_1+e_3)\rangle$}
\end{tabular}

Note that this example shows that the closure of a cell is not necessarily a
union of cells.  In particular, the closure of $C_{110}$ is defined by the
condition $V_1=V_3$ (which implies $V_3 \subset E_2$).  This includes the cells
$C_{000}$ and $C_{010}$, as well as the line in $C_{100}$ defined by
$\alpha_3=\alpha_1$ (in the local coordinates above for $C_{100}$).  (See
also~\cite[Sect. 1.2.5]{Kn-BSeg} for further information on this example.)

Nevertheless, the classes of the closures of the cells of (complex) dimension
$d$ do form a basis for $H_{2d}(Z_\w)\cong H^{2d}(Z_\w)$.
\end{example}

\subsection{A geometric proof of Deodhar's theorem}

Since we have developed the necessary machinery, we give in this section a
proof of \cite[Prop. 3.9]{d}, which was stated without proof in the original.
We also sketch how the existence of a bounded admissible set $\mathcal{E}$
(for $w$ in a Weyl group) can be mostly explained from the Bott-Samelson
resolution and the Decomposition Theorem \cite{bbd} of Beilinson, Berstein,
and Deligne.  We believe this proof was indeed known at
least in outline to the anonymous referee of Deodhar's paper and to others,
but it has never appeared in print.  Nothing in this subsection is necessary
for the remainder of the paper.

\begin{proposition}{\bf \cite[Prop. 3.9]{d}} \label{p:deodhar_b-s}
\begin{enumerate}
\item For any point $p\in C_x$, the polynomial $\sum_{\sigma: \w^\sigma=x}
  q^{d(\sigma)}$ is equal to the Poincar\'e polynomial $\sum_i
    \dim(H_{2i}(\pi_\w^{-1}(p))) q^i$.
\item $\dim_{\mathbb{C}}(\pi_\w^{-1}(p))=\max_{\sigma:\w^\sigma=x} d(\sigma)$.
\item If $2\cdot d(\sigma)<\ell(w)-\ell(x)$ for all $\sigma$ with
  $\w^\sigma=x$, then $\pi_\w$ is a small resolution.
\end{enumerate}
\end{proposition}

\begin{proof}
By Proposition \ref{p:fixed_point_images}, $\pi_\w(p_\sigma)=p_{\w^\sigma}$.
Because the map $\pi_\w$ is $T$-equivariant and the cells $C_{\w^\sigma}$ and
$C_\sigma$ are defined in terms of the $T$-action,
$\pi_\w(C_\sigma)=C_{\w^\sigma}$.  Therefore,
$\pi_\w^{-1}(C_x)=\bigcup_{\w^\sigma=x} C_{\sigma}$.  Since $\pi_\w$ is
equivariant under the action of the Borel subgroup $B$, and $B$ acts
transitively on $C_x$, all the fibers are isomorphic.  Therefore,
$\pi_\w^{-1}(p)$ has a cell decomposition
$\bigcup_{\w^\sigma=x}\pi_\w^{-1}(p)\cap C_{\sigma}$.  The cell
$\pi_\w^{-1}(p_x)\cap C_{\sigma}$ has $\mathbb{C}$-dimension
$$\dim(C_{\sigma})-\dim(C_x)=\ell(x)+d(\sigma)-\ell(x)=d(\sigma)$$ by
Propositions~\ref{p:nplus} and \ref{p:plusdim}.  The classes of the closures
of the cells gives a basis for the homology, proving the first statement.

The second statement follows directly from the first.

A resolution $\pi:Z\rightarrow X$ is small if for all $i$ the locus $\{p\in
X\mid \dim(\pi^{-1}(p))\geq d\}$ has codimension at least $2d$.  Since the
  codimension of $C_x$ in $X_w$ is $\ell(w)-\ell(x)$, the third statement
  follows directly from the second.
\end{proof}

\begin{remark}
We are now in a position to give a geometric proof of Theorem~\ref{t:deodhar}.
The Kazhdan-Lusztig polynomial $P_{x,w}(q)$ is known to be the Poincar\'e
polynomial for the rational local intersection homology
$IH^{p}_*(X_w,\mathbb{Q})$ for the Schubert variety $X_w$ at any point $p\in
C_x$.  The Decomposition Theorem \cite{bbd} of Beilinson, Berstein, and
Deligne, applied to the Bott-Samelson resolution $\pi_\w:Z_\w\rightarrow X_w$,
implies that $IH^{p}_*(X_w,\mathbb{Q})$ is isomorphic (as a graded vector
space) to some subspace of $H_*(\pi_\w^{-1}(p))$.  (For a statement of the
Decomposition Theorem, see \cite[Thm. 8.4.3]{KirWol},
\cite[Thm. 1.6.1]{deCatalMig}, or \cite[Sect. 2.3]{polo}.  Our statement
follows from noticing that $Z_\w$ is smooth and hence has homology equal to
intersection homology, that $X_w$ is simply connected and hence all local
systems are trivial, that $\pi_\w$ is surjective and birational and hence the
intersection homology sheaf $\mathcal{IC}(X_w)$ appears with multiplicity one
in the decomposition of $(\pi_\w)_*(\mathcal{IC}(Z_\w)$, and that local
intersection homology is simply the homology of the stalk of the intersection
homology sheaf.)

By the first part of Proposition~\ref{p:deodhar_b-s}, 
$$H_{2i}(\pi_\w^{-1}(p_x),\mathbb{Q})
\cong \bigoplus_{\substack{\w^\sigma=x \\ d(\sigma)=i}} \mathbb{Q}.$$
Since $IH^{p}_*(X_w,\mathbb{Q})$ is a subspace of
$H_*(\pi_\w^{-1}(p))$, 
$$IH_{2i}^{p}(\mathbb{Q})
\cong \bigoplus_{\substack{\sigma\in\mathcal{E} \\ \w^\sigma=x \\ d(\sigma)=i}}
\mathbb{Q}$$ for some set $\mathcal{E}$ of masks.

By the definition of $P_{x,w}(q)$, the set $\mathcal{E}$ must be bounded, and
furthermore,
$$h(\mathcal{E})=q^{-{1 \over 2} \ell(w)}\sum_{\sigma\in\mathcal{E}}
q^{d(\sigma)} T_{\w^\sigma}$$ must be invariant under the bar involution.
Since $P_{w,w}(q)=1$, the mask of all $1$'s must be in $\mathcal{E}$.  This
approach does not appear to easily show that we can force $\mathcal{E}$ to
satisfy the requirement (2) for admissibility in
Definition~\ref{d:admissible}.

Note that the known proofs of the Decomposition Theorem do not give a way to
identify $IH^{p}_*(X_w,\mathbb{Q})$ with an explicit subspace of
$H_*(\pi^{-1}_\w(p))$.  Moreover, even if such an identification was found, it
would likely produce subspaces which are not spanned by some subset of the
classes of the cells $C_\sigma$.  Hence, there is no canonical choice for a
set of bounded, admissible masks from the geometric viewpoint.
\end{remark}

\section{The Zelevinsky resolution and the geometric construction}\label{s:zel}

\subsection{The Zelevinsky resolution}

Our description of the Zelevinsky resolution is based on the original one of
Zelevinsky \cite{zelevinsky}, though following Perrin~\cite[Sect. 5]{perrin},
we unwind some of the inductive definitions.  We also re-interpret all
statements to describe the Zelevinsky resolution in a manner analogous to
Magyar's description \cite{magyar} of the Bott--Samelson resolution as given
in the previous section.

Let $w$ be a cograssmannian permutation with unique ascent $s_z$.  As noted
previously, $w=vw_0^J$ where $v$ is a grassmannian permutation.  We first fix,
as in Section~\ref{s:cog}, a reduced word $\w$ for $w$ which begins with a
reduced word $\v$ for $v$ and continues by
\[ (s_1 s_2 \cdots s_{z-1}) (s_1 s_2 \cdots s_{z-2}) \cdots \text{ and } (s_{n-1} s_{n-2} \cdots s_{z+1}) (s_{n-1}
s_{n-2} \cdots s_{z+2}) \cdots. \]

Define a {\bf peak} of $\v$ to be an element in the heap for $\v$ with
no other elements above it in the heap.  Let
$\mathbf{P}=(P_1,P_2,\ldots,P_p)$ be a complete ordered list of the peaks (so,
in particular, $p$ denotes the number of peaks).  Given the ordering
$\mathbf{P}$, define $R_j$ for each $j \in \{1, \ldots, p\}$ to be the subset
of the heap of $\v$ consisting of all entries which are below $P_j$ but
not below $P_k$ for any $k>j$.

\begin{deflemma}\label{l:rectangles}
\mbox{}
\begin{enumerate}
\item Each $R_j$ is a (diagonally aligned) rectangle with highest (in
  the heap) element $P_j$.
\item Each rectangle $R_j$ has a unique lowest element.  We denote this
  element and the index of the associated entry of the reduced word $\w$ (or
  equivalently $\v$) by $b_j$.
\item The leftmost element of $R_j$ is either immediately SE of $b_k$ for some
  $k<j$ or immediately to the right of a valley column of the ridgeline.  In
  the first case, define $\before^{\mathbf{P}}(j)=k$; otherwise, leave
  $\before^{\mathbf{P}}(j)$ undefined.  Let $\ldim(j)$ denote the column index
  of $b_{\before^{\mathbf{P}}(j)}$, or the column index of this valley.
\item The rightmost element of $R_j$ is either immediately SW of $b_k$ for
  some $k<j$ or immediately to the left of a valley column of the ridgeline.
  In the first case, define $\after^{\mathbf{P}}(j)=k$; otherwise, leave
  $\after^{\mathbf{P}}(j)$ undefined.  Let $\rdim(j)$ denote the column index
  of $b_{\after^{\mathbf{P}}(j)}$, or the column index of this valley.
\end{enumerate}
\end{deflemma}

This lemma is essentially proved both in \cite{zelevinsky} and
\cite[Sect. 5.4]{perrin}, and we leave the details to the reader.  Note that
by definition $b_p=\last(z)$, the last occurence of our unique ascent $s_z$ in
$\w$.

\begin{example}
Consider the cograssmannian permutation $w\in S_{23}$ (with $z=12$) given by
the heap of Figure~\ref{f:zel_rectangles}.

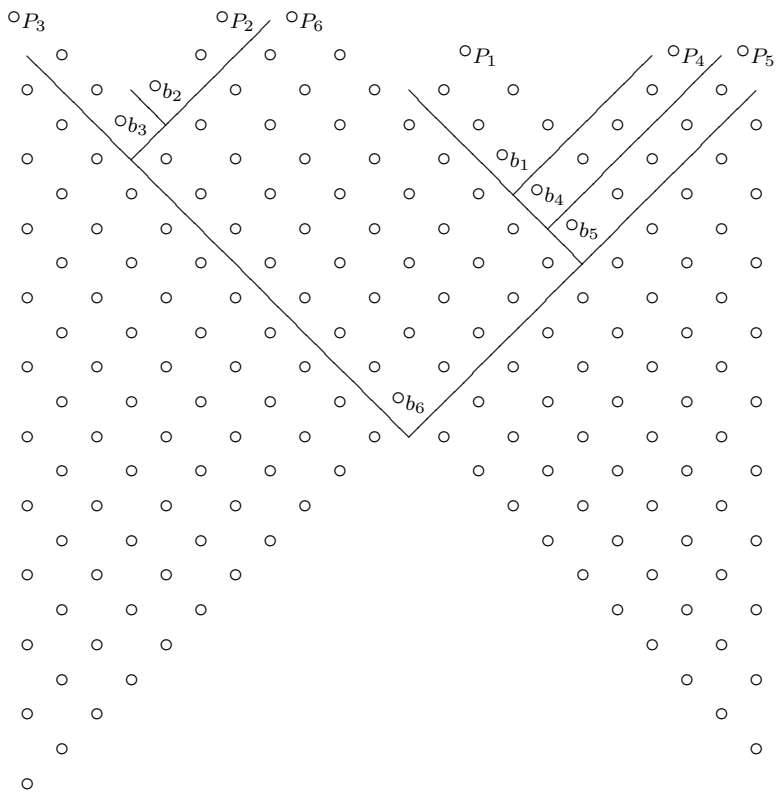
\begin{figure}[h]
\begin{center}
\begin{tabular}{c}
\xymatrix @-2.2pc @M=0pt @! {
{\hz_{P_3}} & & {\hs} & & {\hs} & & {\hz_{P_2}} & & {\hz_{P_6}} & & {\hs} & & {\hs} & & {\hs} & & {\hs} & & {\hs} & & {\hs} & \\
& {\hz} & & {\hs} & & {\hz} & & {\hz} & & {\hz} & & {\hs} & & {\hz_{P_1}} & & {\hs} & & {\hs} & & {\hz_{P_4}} & & {\hz_{P_5}} \\
{\hz} & & {\hz} & & {\hz_{b_2}} & & {\hz} & & {\hz} & & {\hz} & & {\hz} & & {\hz} & & {\hs} & & {\hz} & & {\hz} & \\
& {\hz} & & {\hz_{b_3}} & \ar@{-}[ul] & {\hz} & & {\hz} & & {\hz} & & {\hz} & & {\hz} & & {\hz} & & {\hz} & & {\hz} & & {\hz} \\
{\hz} & & {\hz} & \ar@{-}[urururur] & {\hz} & & {\hz} & & {\hz} & & {\hz} & & {\hz} & & {\hz_{b_1}} & & {\hz} & & {\hz} & & {\hz} & \\
& {\hz} & & {\hz} & & {\hz} & & {\hz} & & {\hz} & & {\hz} & & {\hz} & \ar@{-}[urururur]& {\hz_{b_4}} & & {\hz} & & {\hz} & & {\hz} \\
{\hz} & & {\hz} & & {\hz} & & {\hz} & & {\hz} & & {\hz} & & {\hz} & & {\hz} & \ar@{-}[ururururur] & {\hz_{b_5}} & & {\hz} & & {\hz} & \\
& {\hz} & & {\hz} & & {\hz} & & {\hz} & & {\hz} & & {\hz} & & {\hz} & & {\hz} & \ar@{-}[ululululul] & {\hz} & & {\hz} & & {\hz} \\
{\hz} & & {\hz} & & {\hz} & & {\hz} & & {\hz} & & {\hz} & & {\hz} & & {\hz} & & {\hz} & & {\hz} & & {\hz} & \\
& {\hz} & & {\hz} & & {\hz} & & {\hz} & & {\hz} & & {\hz} & & {\hz} & & {\hz} & & {\hz} & & {\hz} & & {\hz} \\
{\hz} & & {\hz} & & {\hz} & & {\hz} & & {\hz} & & {\hz} & & {\hz} & & {\hz} & & {\hz} & & {\hz} & & {\hz} & \\
& {\hz} & & {\hz} & & {\hz} & & {\hz} & & {\hz} & & {\hz_{b_6}} & & {\hz} & & {\hz} & & {\hz} & & {\hz} & & {\hz} \\
{\hz} & & {\hz} & & {\hz} & & {\hz} & & {\hz} & & {\hz} & \ar@{-}[ululululululululululul]\ar@{-}[urururururururururur] & {\hz} & & {\hz} & & {\hz} & & {\hz} & & {\hz} & \\
& {\hz} & & {\hz} & & {\hz} & & {\hz} & & {\hz} & & {\hs} & & {\hz} & & {\hz} & & {\hz} & & {\hz} & & {\hz} \\
{\hz} & & {\hz} & & {\hz} & & {\hz} & & {\hz} & & {\hs} & & {\hs} & & {\hz} & & {\hz} & & {\hz} & & {\hz} & \\
& {\hz} & & {\hz} & & {\hz} & & {\hz} & & {\hs} & & {\hs} & & {\hs} & & {\hz} & & {\hz} & & {\hz} & & {\hz} \\
{\hz} & & {\hz} & & {\hz} & & {\hz} & & {\hs} & & {\hs} & & {\hs} & & {\hs} & & {\hz} & & {\hz} & & {\hz} & \\
& {\hz} & & {\hz} & & {\hz} & & {\hs} & & {\hs} & & {\hs} & & {\hs} & & {\hs} & & {\hz} & & {\hz} & & {\hz} \\
{\hz} & & {\hz} & & {\hz} & & {\hs} & & {\hs} & & {\hs} & & {\hs} & & {\hs} & & {\hs} & & {\hz} & & {\hz} & \\
& {\hz} & & {\hz} & & {\hs} & & {\hs} & & {\hs} & & {\hs} & & {\hs} & & {\hs} & & {\hs} & & {\hz} & & {\hz} \\
{\hz} & & {\hz} & & {\hs} & & {\hs} & & {\hs} & & {\hs} & & {\hs} & & {\hs} & & {\hs} & & {\hs} & & {\hz} & \\
& {\hz} & & {\hs} & & {\hs} & & {\hs} & & {\hs} & & {\hs} & & {\hs} & & {\hs} & & {\hs} & & {\hs} & & {\hz} \\
{\hz} & & {\hs} & & {\hs} & & {\hs} & & {\hs} & & {\hs} & & {\hs} & & {\hs} & & {\hs} & & {\hs} & & {\hs} & \\
}
\end{tabular}
\end{center}
\caption{Rectangles for neat ordering on cograssmannian
  heap}\label{f:zel_rectangles}
\end{figure}

In the figure, the peaks $P_1,\ldots,P_6$ are labelled as are the minimal
elements $b_1,\ldots,b_6$.  The rectangles $R_1,\ldots,R_6$ are divided by
lines but not explicitly labelled.
\end{example}

For all $j$, $1\leq j\leq p$, we define $d^{\mathbf{P}}_j$ to be the index
such that $\v_{b_j}=\w_{b_j}=s_{d^{\mathbf{P}}_j}$, or, in the notation of
Section~\ref{s:b-s}, $d^{\mathbf{P}}_j=d_{b_j}$.  Now we define the {\bf
  Zelevinsky variety} $Z_\mathbf{P}$ as a subvariety of
$$\Gr(d^{\mathbf{P}}_1,n)\times\cdots\times\Gr(d^{\mathbf{P}}_{p-1},n)\times\Gr(1,n)\times\cdots\times\Gr(n-1,n)$$
as follows.

A point
$q\in\Gr(d^{\mathbf{P}}_1,n)\times\cdots\times\Gr(d^{\mathbf{P}}_{p-1},n)\times\Gr(1,n)\times\cdots\times\Gr(n-1,n)$
corresponds to a sequence of subspaces $W_1,\ldots,W_{p-1},F_1,\ldots,F_{n-1}$
of $\mathbb{C}^n$.  As before, we denote the point corresponding to such a
sequence of subspaces by $[W_1,\ldots,W_{p-1},F_1,\ldots,F_{n-1}]$.  This
point is in $Z_\mathbf{P}$ if these vector spaces satisfy all of the following
three conditions:
\begin{enumerate}
\item $$W_{\before^{\mathbf{P}}(j)}\subset W_{j}\subset
  W_{\after^{\mathbf{P}}(j)}$$ for $1\leq j\leq p-1$.  If
  $\before^{\mathbf{P}}(j)$ is undefined, we require instead that
  $E_{\ldim(j)}\subset W_{j}$, and if $\after^{\mathbf{P}}(j)$ is undefined,
  we require instead that $W_{j}\subset E_{\rdim(j)}$.
\item $$W_{\before^{\mathbf{P}}(p)}\subset F_z\subset
  W_{\after^{\mathbf{P}}(p)}.$$  (Recall that $z$ is the index of the unique
  ascent in $w$.)  As in the previous item, we require instead that
  $E_{\ldim(p)}\subset F_z$ if $\before^{\mathbf{P}}(p)$ is undefined and
  $F_z\subset E_{\rdim(p)}$ if $\after^{\mathbf{P}}(p)$ is undefined.
\item $$F_j\subset F_{j+1}$$ for $1\leq j\leq n-2$.
\end{enumerate}

\begin{example}
Given $w\in S_{23}$ and $\mathbf{P}$ as in Figure~\ref{f:zel_rectangles}, a
point $[W_1,\ldots,W_5, F_1,\ldots,F_{22}]$ is in $Z_{\mathbf{P}}$ if and only
if $E_{12}\subset W_1\subset E_{17}$, $E_4\subset W_2\subset E_8$, $W_3\subset
W_2$, $W_1\subset W_4\subset E_{21}$, $W_4\subset W_5$, $W_3\subset
F_{12}\subset W_5$, and $F_1\subset\cdots\subset F_{22}$.
\end{example}

As with Bott--Samelson varieties, there is a forgetful map $\pi_{\mathbf{P}}:
Z_{\mathbf{P}}\rightarrow X_w$, defined by
$$\pi_{\mathbf{P}}([W_1,\ldots,W_{p-1},F_1,\ldots,F_{n-1}])=[F_1,\ldots,F_{n-1}].$$
We call this map the {\bf Zelevinsky resolution}.  Zelevinsky
\cite{zelevinsky} shows that the image of $\pi_{\mathbf{P}}$ is indeed $X_w$.
Furthermore, given any ordering $\mathbf{P}$ of the peaks, there is a map
$\rho_{\mathbf{P}}:Z_{\w}\rightarrow Z_{\mathbf{P}}$ from the Bott--Samelson
variety (for our fixed choice of $\w$) to the Zelevinsky variety
$Z_{\mathbf{P}}$, defined by
$$\rho_{\mathbf{P}}([V_1,\ldots,V_\ell])=[V_{b_1},\ldots,V_{b_{p-1}},V_{\last(1)},\ldots,V_{\last(n-1)}].$$
It follows from Lemma~\ref{l:rectangles} that the image of $\rho_{\mathbf{P}}$ is
$Z_{\mathbf{P}}$ and not some other subvariety of
$\Gr(d^{\mathbf{P}}_1,n)\times\cdots\times\Gr(d^{\mathbf{P}}_{p-1},n)\times\Gr(1,n)\times\cdots\times\Gr(n-1,n)$.
It is clear from the definitions that
$\pi_\w=\pi_{\mathbf{P}}\circ\rho_{\mathbf{P}}$.

\subsection{Fixed points of the Zelevinsky resolution}

In preparation for describing a cell decomposition of the Zelevinsky variety,
we describe the $T$-fixed points and give a combinatorial indexing set for them.

A $T$-fixed point $q \in Z_{\mathbf{P}}$ is of the form
$q=[W_1,\ldots,W_{p-1},F_1,\ldots,F_{n-1}]$, where each $W_j$ and $F_k$ is a
coordinate subspace spanned by some subset of the coordinate vectors
$\{e_1,\ldots,e_n\}$.  Given such subspaces, we construct 
$\bftau = (\tau^{(1)},\ldots,\tau^{(p)}; x_\tau)$ 
consisting of a sequence of partitions $\tau^{(j)}$ such that each $\tau^{(j)}$
fits inside the rectangle $R_j$ together with a permutation $x_\tau\in
S_{z}\times S_{n-z}$.  For notational convenience, we let $W_p=F_z$.

Since the conditions defining $Z_{\mathbf{P}}$ require that
$W_{\before^{\mathbf{P}}(j)}\subset W_j\subset W_{\after^{\mathbf{P}}(j)}$,
and $W_j$ is a coordinate subspace, choosing $W_j$ once we are given
$W_1,\ldots,W_{j-1}$ amounts to choosing $d_j-\ldim(j)$ coordinate vectors
from the $\rdim(j)-\ldim(j)$ coordinate vectors in
$W_{\after^{\mathbf{P}}(j)}$ but not in $W_{\before^{\mathbf{P}}(j)}$.  We set
some notation to describe such choices.  For each $j$, $1\leq j\leq p$, we
define $A(j)$ to be set of indices
$$A(j)=\{k \mid e_k\in W_{\after^{\mathbf{P}}(j)} \} \setminus \{k \mid e_k\in
W_{\before^{\mathbf{P}}(j)} \}.$$  (Here and below,
whenever $\after^{\mathbf{P}}(j)$ or $\before^{\mathbf{P}}(j)$ is undefined,
one should substitute $E_{\rdim(j)}$ for $W_{\after^{\mathbf{P}}(j)}$ or
respectively $E_{\ldim(j)}$ for $W_{\before^{\mathbf{P}}(j)}$.)  Now let
$T(j)$ be the set of indices
$$T(j)=\{k \mid e_k\in W_j \} \setminus \{k \mid e_k\in
W_{\before^{\mathbf{P}}(j)} \};$$ these are the indices for the coordinate
vectors in $W_j$.  Finally, let $$D(j)=A(j)\setminus T(j)$$ be the indices for
the coordinate vectors not in $W_j$.

We now have two equivalent ways to define the partition $\tau^{(j)}$.  First,
we can do so by drawing a lattice path on the heap points in $R_j$ from the
leftmost heap point to the rightmost heap point of $R_j$ as follows.  Read the
elements of $A(j)$ from smallest to largest, drawing a SE segment whenever the
element is in $T(j)$ and a NE segment whenever the element is in $D(j)$.  Now
define the partition to be the heap points of $R_j$ on or below this path.

Our convention is to consider the NE-SW diagonals as the parts of the
partition.  This means that in order to fit inside $R_j$, $\tau^{(j)}$ has at
most $d^{\mathbf{P}}_j-\ldim(j)$ parts, each of size at most
$\rdim(j)-d^{\mathbf{P}}_j$.  (The dimensions of $R_j$ are determined by
Lemma~\ref{l:rectangles}.)

Equivalently, if we order $T(j)$ so that
$T(j)_1>T(j)_2>\cdots>T(j)_{d^{\mathbf{P}}_j-\ldim(j)}$, then $\tau^{(j)}$ is the
partition with $k$-th part defined by $$\tau^{(j)}_k=\#\{m\in D(j)\mid
m<T(j)_k\}.$$

We now construct $x_\tau$.  Let $u_\bftau$ be the grassmannian permutation encoding
the subspace $F_z=W_p$ in the standard fashion by $F_z=\Span(e_{u_\bftau(1)},\ldots,e_{u_\bftau(z)})$
where $u_\bftau(1)<\cdots<u_\bftau(z)$ and $u_\bftau(z+1)<\cdots<u_\bftau(n)$.
Then $x_\tau$ is the unique permutation in $S_z\times S_{n-z}$ such that
$F_j=\Span\{e_{u_\bftau x_\tau(1)},\ldots, e_{u_\bftau x_\tau(j)}\}$ for all
$j$.

\begin{example}
\label{e:zel_large_cell}
Let $w$ and $\mathbf{P}$ be as in Figure~\ref{f:zel_rectangles}.  Suppose
$p=[W_1,\ldots,W_5,F_1,\ldots,F_{22}]\in Z_{\mathbf{P}}$ is the $T$-fixed
point with
\begin{eqnarray*}
W_1 & = & \Span(e_1,\ldots,e_{12}, e_{15}, e_{16}, e_{17}) \\
W_2 & = & \Span(e_1, e_2, e_3, e_4, e_7) \\
W_3 & = & \Span(e_1, e_2, e_3, e_4) \\
W_4 & = & \Span(e_1,\ldots,e_{12}, e_{14}, e_{15}, e_{16}, e_{17}) \\
W_5 & = & \Span(e_1,\ldots,e_{12}, e_{13}, e_{14}, e_{15}, e_{16}, e_{17}) \\
W_6=F_{12} & = & \Span(e_1,\ldots,e_4, e_6, e_7, e_9, e_{10}, e_{11}, e_{13},
e_{15}, e_{17}).
\end{eqnarray*}
We leave $F_1,\ldots, F_{11}$ and $F_{13},\ldots,F_{22}$ unspecified to
concentrate on the partition part of $\bftau$.  The reader can check that
$\tau^{(1)}=(2,2,2)$, $\tau^{(2)}=(2)$, $\tau^{(3)}=\emptyset$,
$\tau^{(4)}=(1)$, $\tau^{(5)}=\emptyset$, and $\tau^{(6)}=(5,4,3,2,2,2,1,1)$.
\end{example}

We let $\mathcal{T}_{\mathbf{P}}$ be the indexing set of sequences
$\bftau=(\tau^{(1)},\ldots,\tau^{(p)};x_\tau)$ where for each $j$, $\tau^{(j)}$
is a partition fitting inside $R_j$, and $x_\tau$ is a permutation in the
Young subgroup $S_z \times S_{n-z}$.

\begin{proposition}
\label{p:zel_fixed_points}
The above construction describes a bijection between the $T$-fixed points of
$Z_{\mathbf{P}}$ and $\mathcal{T}_{\mathbf{P}}$.
\end{proposition}

\begin{proof}
Since for each $j$, $D(j)$ must have $\rdim(j)-d^{\mathbf{P}}_j$ elements, and
$T(j)$ must have $d^{\mathbf{P}}_j-\ldim(j)$ elements, the partition
$\tau^{(j)}$ fits inside $R_j$.  Therefore we need only show that we can
recover a $T$-fixed point $[W_1,\ldots,W_{p-1},F_1,\ldots,F_{n-1}]$ of
$Z_{\mathbf{P}}$ from a sequence of partitions
$(\tau^{(1)},\ldots,\tau^{(p)})$ and the permutation $x_\tau\in S_z\times
S_{n-z}$.  We do so by recovering the coordinate subspaces
$W_1,\ldots,W_{p-1},W_p=F_z$ one at a time, then recovering
$F_1,\ldots,F_{n-1}$.

Given a sequence of partitions $(\tau^{(1)},\ldots,\tau^{(p)})$, suppose
$W_1,\ldots,W_{j-1}$ have been recovered.  Since the subspaces
$W_{\before^{\mathbf{P}}(j)}$ and $W_{\after^{\mathbf{P}}(j)}$ that $W_j$ must
include and be included in are known at this point, the set $A(j)$ is
determined.  For convenience, order $A(j)$ so that
$A(j)_1>A(j)_2>\cdots>A(j)_{\rdim(j)-\ldim(j)}$.  We can now recover $T(j)$ as
the set $\{A(j)_{k+(\rdim(j)-d^{\mathbf{P}}_j-\tau^{(j)}_k)}\mid k=1,\ldots,d^{\mathbf{P}}_j-\ldim(j)\}$
for each $j$ (where we consider $\tau^{(j)}$ to have exactly
$d^{\mathbf{P}}_j-\ldim(j)$ parts, possibly including some zero parts).
Knowing $T(j)$ and $W_{\before^{\mathbf{P}}(j)}$ allows one to recover $W_j$
exactly as $W_j = \Span \{e_k\}_{k \in T(j)} \cup W_{\before^{\mathbf{P}}(j)}$.

The subspace $W_p=F_z$ determines the grassmannian permutation $u_\bftau$.
Finally we recover $F_j$ as $F_j=\Span\{e_{u_\bftau
  x_\tau(1)},\ldots,e_{u_\bftau x_\tau(j)}\}$.
\end{proof}

Given $\bftau\in\mathcal{T}_\mathbf{P}$, we denote by $p_\bftau$ the
corresponding $T$-fixed point.

\subsection{Cells of the Zelevinsky variety and their dimensions}

We again use the Bialynicki-Birula theorem to give a cell decomposition of
$Z_{\mathbf{P}}$, using the $T=\mathbb{C}^*$ action on $Z_{\mathbf{P}}$
induced by the $T$-action on $\mathbb{C}^n$ given in
Section~\ref{sect:b-s-cells}.  We denote the cell associated to $p_\bftau$ by
$C_\bftau$.  The cell $C_\bftau$ is defined according the Bialynicki-Birula
decomposition as
$$C_\bftau=\{p\in Z_{\mathbf{P}} \mid \lim_{t\rightarrow0} t\cdot p=p_\bftau\}$$

As with the Bott--Samelson resolution, we have the following proposition.
Given the complexity of the definition of $p_\bftau$, it seems unlikely this
proposition can be put into a simpler form.
\begin{proposition}\label{p:zel_cell_pivots}
Let $p_\bftau=[W^\prime_1,\ldots,W^\prime_{p-1},F^\prime_1,\ldots,F^\prime_{n-1}]$.
A point $p=[W_1,\ldots,W_{p-1},F_1,\ldots,F_{n-1}]$ of $Z_{\mathbf{P}}$ is in
$C_\bftau$ if, for all $j$ and $k$, the right-to-left row echelon forms of the
matrices whose rows span $W_j$ and $F_k$ have pivots in the same columns as the
matrices whose rows are a basis for $W^\prime_j$ and $F^\prime_k$.
\end{proposition}

As with Schubert cells and cells on the Bott--Samelson, this linear algebra
condition is equivalent to the equalities $\dim(E_m\cap W_j)=\dim(E_m\cap
W^\prime_j)$ and $\dim(E_m\cap F_k)=\dim(E_m\cap F^\prime_k)$ for all relevant
$j$, $k$, and $m$.  Unlike in the Schubert and Bott--Samelson cases, these
dimensions do not seem to have a simple combinatorial description.

We now describe the dimension of the cell $C_\bftau$.  This description is
analogous to the description of the dimension of $C_\s$ as the number of
$+$'s.

\begin{proposition}\label{p:zeldim} The dimension of $C_\bftau$ is
$$\sum_i \left|\tau^{(i)}\right| + \ell(x_\tau).$$
\end{proposition}

\begin{proof}
Let $y=[W_1,\ldots,W_{p-1},F_1,\ldots,F_{n-1}]$ be a point of $C_\bftau$.  We
show that, once $W_1,\ldots,W_{j-1}$ are chosen, we have a
$\left|\tau^{(j)}\right|$-dimensional choice for $W_j$.  Note that we are
merely describing the standard Schubert cell decomposition for the Grassmannian
$\Gr(d^{\mathbf{P}}_j-\ldim(j),\rdim(j)-\ldim(j))$ with slight complications
because the subspaces that $W_j$ is required to contain and be contained in are
not fixed.

Given $\bftau$ we recover by Proposition~\ref{p:zel_fixed_points} the fixed
point
$p_\bftau=[W^\prime_1,\ldots,W^\prime_{p-1},F^\prime_1,\ldots,F^\prime_{n-1}]$
and the sets $A(j)$, $T(j)$, and $D(j)$.  We order each $T(j)$ in decreasing
order and each $D(j)$ in increasing order.  By
Proposition~\ref{p:zel_cell_pivots}, the subspace $W_{\after^{\mathbf{P}}(j)}$
(or $E_{\rdim(j)}$) is spanned by $W_{\before^{\mathbf{P}}(j)}$ (or
$E_{\ldim(j)}$) and some vectors $a_1,\ldots, a_{\rdim(j)-\ldim(j)}$ where
$a_k=e_{A(j)_k}+\sum_{m<A(j)_k} c_{mk}e_m$ for some scalars $c_{mk}$.  (As in
the proof of Proposition~\ref{p:plusdim}, $a_k$ is not canonically determined
as it can be modified by adding vectors in $W_{\before^{\mathbf{P}}(j)}$ as
  well as $a_{k^\prime}$ where $A(j)_{k^\prime}<A(j)_k$, but this only amounts
  to a change of coordinates and does not change the dimension.)

If $y\in C_\bftau$, then $W_j$ is spanned by $W_{\before^{\mathbf{P}}(j)}$ (or
$E_{\ldim(j)}$) and vectors $b_1,\ldots,b_{d^{\mathbf{P}}_j-\ldim(j)}$
where $$b_k=a_{T(j)_k}+\sum_{m=1}^{\tau^{(j)}_k} \alpha^{(j)}_{mk}a_{D(j)_m}$$
for some parameters $\alpha^{(j)}_{mk}$, as the pivots in the right-to-left
row echelon form of $W_j$ must be in the columns indexed by $T(j)$ (and the
indices $m$ for $e_m\in W^\prime_{\before^{\mathbf{P}}(j)}$).  There are
exactly $\left|\tau^{(j)}\right|$ of these parameters, so the choice of $W_j$
adds $\left|\tau^{(j)}\right|$ to the dimension of $Z_{\mathbf{P}}$.  The same
argument applies to $W_p=F_z$, so the choices for $W_1,\ldots,W_{p-1}$ and
$F_z$ contribute $\sum_{j=1}^p \left|\tau^{(j)}\right|$ to the dimension.

The subspace $F_z$ has a basis $a_1,\ldots,a_z$
where $$a_k=e_{u_\bftau(k)}+\sum_{m=1}^{u_\bftau(k)} c_{mk}e_k$$ for some
scalars $c_{mk}$.  Let $Inv(j)$ be the inversions involving $j$ and some
$k>j$ in $x_\tau$, or $Inv(j)=\{k>j\mid x_\tau(k)<x_\tau(j)\}$.  If $y\in
C_\bftau$, then by Proposition~\ref{p:zel_cell_pivots}, $F_j$ will be spanned by
$F_{j-1}$ and some vector of the form $$a_{x_\tau(j)}+\sum_{k\in Inv(j)}
\beta_{kj}a_{x_\tau(k)}$$ for $j\leq z-1$.  For $j>z$, $F_j$ will be spanned by
$F_{j-1}$ and some vector of the form $$e_{u_\bftau x_\tau(j)}+\sum_{k\in
  Inv(j)} \beta_{kj}e_{u_\bftau x_\tau(k)}.$$  The number of parameters in the
choices for $F_1,\ldots,F_{z-1},F_{z+1},\ldots,F_n$ is the number of
inversions of $x_\tau$, which is precisely $\ell(x_\tau)$.
\end{proof}

\begin{example}
The following table describes all the cells $C_\bftau$ with $x_\tau=1234$ for
$w=4231$ under the ordering of the peaks where $P_1$ is the one to the left.\\
\begin{tabular}{c|c|c|c|c}
$\bftau$ & $u_\bftau$ & $\dim(C_\bftau)$ & $p_\bftau$ & $C_\bftau$ \\
\hline
$\emptyset, \emptyset$ & $1234$ & 0 & 
$\langle e_1\rangle, \langle e_1\rangle, \langle e_1,e_2\rangle, \langle
e_1,e_2,e_3\rangle$ &
$\langle e_1\rangle, \langle e_1\rangle, \langle e_1,e_2\rangle, \langle
e_1,e_2,e_3\rangle$ \\
$(1), \emptyset$ & $1234$ & 1 &
$\langle e_2\rangle, \langle e_1\rangle, \langle e_1,e_2\rangle, \langle
e_1,e_2,e_3\rangle$ &
$\langle e_2+\alpha^{(1)}_{11}e_1 \rangle, \langle e_1\rangle, \langle
e_1,e_2\rangle, \langle e_1,e_2,e_3\rangle$ \\
$\emptyset, (1)$ & $1324$ & 1 &
$\langle e_1\rangle, \langle e_1\rangle, \langle e_1,e_3\rangle, \langle
e_1,e_2,e_3\rangle$ &
$\langle e_1\rangle, \langle e_1\rangle, \langle
e_1,e_3+\alpha^{(2)}_{11}e_2\rangle, \langle e_1,e_2,e_3\rangle$ \\
$(1), (1)$ & $2314$ & 2 &
$\langle e_2\rangle, \langle e_2\rangle, \langle e_2,e_3\rangle, \langle
e_1,e_2,e_3\rangle$ &
\parbox[t]{2.3in}{$\langle e_2+\alpha^{(1)}_{11}\rangle$, $\langle
e_2+\alpha^{(1)}_{11}\rangle$, \\ $\langle
e_2+\alpha^{(1)}_{11},e_3+\alpha^{(2)}_{11}e_1\rangle$, $\langle
e_1,e_2,e_3\rangle$} \\
$\emptyset, (2)$ & $1423$ & 2 &
$\langle e_1\rangle, \langle e_1\rangle, \langle e_1,e_4\rangle, \langle
e_1,e_3,e_4\rangle$ &
\parbox[t]{2.3in}{$\langle e_1\rangle$, $\langle e_1\rangle$, $\langle e_1,
e_4+\alpha^{(2)}_{11}e_2+\alpha^{(2)}_{21}e_3\rangle$, $\langle e_1, e_2,
e_4+\alpha^{(2)}_{11}e_2+\alpha^{(2)}_{21}e_3\rangle$} \\
$(1), (2)$ & $2413$ & 3 &
$\langle e_2\rangle, \langle e_2\rangle, \langle e_2,e_4\rangle, \langle
e_1,e_2,e_4\rangle$ &
\parbox[t]{2.3in}{$\langle e_2+\alpha^{(1)}_{11}\rangle$, $\langle e_2+\alpha^{(1)}_{11}\rangle$,\\
$\langle e_2+\alpha^{(1)}_{11},
e_4+\alpha^{(2)}_{11}e_1+\alpha^{(2)}_{21}e_3\rangle$, \\$\langle e_1,
e_2+\alpha^{(1)}_{11}, e_4+\alpha^{(2)}_{11}e_1+\alpha^{(2)}_{21}e_3\rangle$}
\end{tabular}

The following table describes all the cells $C_\bftau$ with $x_\tau=2134$ for
$w=4231$ under the same ordering of the peaks.\\
\begin{tabular}{c|c|c|c|c}
$\bftau$ & $u_\bftau x_\tau$ & $\dim(C_\bftau)$ & $p_\bftau$ & $C_\bftau$ \\
\hline
$\emptyset, \emptyset$ & $2134$ & 1 & 
$\langle e_1\rangle, \langle e_2\rangle, \langle e_1,e_2\rangle,
\langle e_1,e_2,e_3\rangle$ &
$\langle e_1\rangle, \langle e_2+\beta_{21}e_1\rangle, \langle e_1,e_2\rangle,
\langle e_1,e_2,e_3\rangle$ \\

$(1), \emptyset$ & $2134$ & 2 &
$\langle e_2\rangle, \langle e_2\rangle, \langle e_1,e_2\rangle, \langle
e_1,e_2,e_3\rangle$ &
\parbox[t]{2.3in}{$\langle e_2+\alpha^{(1)}_{11}e_1 \rangle, \langle e_2+\beta_{21}e_1\rangle,
\langle e_1,e_2\rangle$,\\ $\langle e_1,e_2,e_3\rangle$} \\
$\emptyset, (1)$ & $3124$ & 2 &
$\langle e_1\rangle, \langle e_3\rangle, \langle e_1,e_3\rangle, \langle
e_1,e_2,e_3\rangle$ &
\parbox[t]{2.3in}{$\langle e_1\rangle, \langle e_3+\alpha^{(2)}_{11}e_2+\beta_{21}e_1\rangle$,\\
$\langle e_1,e_3+\alpha^{(2)}_{11}e_2\rangle, \langle e_1,e_2,e_3\rangle$} \\
$(1), (1)$ & $3214$ & 3 &
$\langle e_2\rangle, \langle e_3\rangle, \langle e_2,e_3\rangle, \langle
e_1,e_2,e_3\rangle$ &
\parbox[t]{2.3in}{$\langle e_2+\alpha^{(1)}_{11}\rangle, \langle
e_3+\beta_{21}(e_2+\alpha^{(1)}_{11})\rangle$,\\ $\langle
e_2+\alpha^{(1)}_{11},e_3+\alpha^{(2)}_{11}e_1\rangle, \langle
e_1,e_2,e_3\rangle$} \\
$\emptyset, (2)$ & $4123$ & 3 &
$\langle e_1\rangle, \langle e_4\rangle, \langle e_1,e_4\rangle, \langle
e_1,e_3,e_4\rangle$ &
\parbox[t]{2.3in}{$\langle e_1\rangle, \langle
e_4+\alpha^{(2)}_{11}e_2+\alpha^{(2)}_{21}e_3+\beta_{21}e_1\rangle$,\\ $\langle
e_1, e_4+\alpha^{(2)}_{11}e_2+\alpha^{(2)}_{21}e_3\rangle$,\\ $\langle e_1, e_2,
e_4+\alpha^{(2)}_{11}e_2+\alpha^{(2)}_{21}e_3\rangle$} \\
$(1), (2)$ & $4213$ & 4 &
$\langle e_2\rangle, \langle e_4\rangle, \langle e_2,e_4\rangle, \langle
e_1,e_2,e_4\rangle$ &
\parbox[t]{2.3in}{$\langle e_2+\alpha^{(1)}_{11}\rangle$,\\ $\langle
e_4+\alpha^{(2)}_{11}e_1+\alpha^{(2)}_{21}e_3+\beta_{21}(e_2+\alpha^{(1)}_{11})
\rangle$,\\ $\langle e_2+\alpha^{(1)}_{11},
e_4+\alpha^{(2)}_{11}e_1+\alpha^{(2)}_{21}e_3\rangle$,\\ $\langle e_1,
e_2+\alpha^{(1)}_{11}, e_4+\alpha^{(2)}_{11}e_1+\alpha^{(2)}_{21}e_3\rangle$}
\end{tabular}
\end{example}

\subsection{Zelevinsky resolutions and Deodhar sets}

In this section, we come to the reason that Zelevinsky resolutions were
introduced.  First we need two more definitions.  Each peak $P$ has a {\bf
height} $h(P)$, which is defined to be the length of a maximal chain in the
heap poset from the unique lowest element of $\v$ (considered as a heap).  We
say our ordering of peaks $\mathbf{P}$ is {\bf neat} if, for all $j$,
$h(P_j)\geq h(P_{\before^{\mathbf{P}}(j)})$ and $h(P_j)\geq
h(P_{\after^{\mathbf{P}}(j)})$ whenever $\before^{\mathbf{P}}(j)$ or
$\after^{\mathbf{P}}(j)$ is defined.  Explicitly, this occurs if for all $j<k$,
either $h(P_j)\leq h(P_k)$, or there exists $m>j$ such that $P_m$ is between
$P_j$ and $P_k$ (in the ridgeline of the heap).  By induction, if in a neat
ordering $j<k$ and $h(P_j)>h(P_k)$, then there necessarily exists $m>k$ such
that $P_m$ is between $P_j$ and $P_k$, and $h(P_m)\geq h(P_j)$.

Even though the heights of $P_2$ and $P_3$ are greater than those of $P_4$ and
$P_5$, our ordering of peaks in Figure~\ref{f:zel_rectangles} is neat as $P_6$
is in between $P_2$ and $P_4$.

A map $\pi:X\rightarrow Y$ is {\bf small} if, for all $i$, the locus $$\{p\in
  Y\mid \dim(\pi^{-1}(p))\geq i\}$$ has codimension (in $Y$) strictly greater
  than $2i$.  Goresky and MacPherson~\cite[Sect. 6.2]{g-m} (see also
  \cite[Thm. 8.4.7]{KirWol}) show that when $\pi$ is a small map, the
  intersection homology of $Y$ is isomorphic (as a group) to the intersection
  homology of $X$, which is simply the homology of $X$ if $X$ is smooth.
  Zelevinsky~\cite{zelevinsky} showed to following as an explanation of the
  Lascoux--Sch\"utzenberger formula.

\begin{theorem}[Zelevinsky \cite{zelevinsky}]\label{t:zel}
If $\mathbf{P}$ is a neat ordering of the peaks, then $\pi_{\mathbf{P}}$ is a
small map.  In particular, if $\mathbf{P}$ is neat,
  then $$P_{x,w}(q)=\sum_{\bftau: \pi_{\mathbf{P}}(p_\bftau)=p_x}
  q^{\dim(C_\bftau)-\ell(x)}.$$
\end{theorem}

Note that $\pi_{\mathbf{P}}(p_\bftau)=p_x$ if and only if $u_\bftau x_\tau =
x$.  The combinatorial bijection to the trees of Lascoux--Sch\"utzenberger is
far from obvious.

Deodhar's formula states that says that a bounded admissible set of masks
$\mathcal{E}$ is one for which
$$P_{x,w}(q)=\sum_{\sigma \in \mathcal{E} : \w^\sigma=x}
q^{\dim(C_\sigma)-\ell(x)},$$ recalling that $d(\sigma)+\ell(x) = $ number of
$+$'s in $e(\sigma) = \dim(C_\sigma)$.  Comparing these two formulas, we see
that one way to show that a set of masks $\mathcal{E}$ is bounded and
admissible is to take a neat ordering $\mathbf{P}$ and find a combinatorial
bijection $ R_{\mathbf{P}}: \mathcal{E}\rightarrow \mathcal{T}_{\mathbf{P}}$
having the following properties for all $\sigma\in\mathcal{E}$:
\begin{itemize}
\item $\pi_{\mathbf{P}}(p_{ R_{\mathbf{P}}(\sigma)})=p_{\w^\sigma}$
\item $\dim(C_{ R_{\mathbf{P}}(\sigma)})-\ell(\w^\sigma)=d(\sigma)$
\end{itemize}

Furthermore, recall that for any ordering of the peaks $\mathbf{P}$ there
exists a map $\rho_{\mathbf{P}}: Z_\w \rightarrow Z_{\mathbf{P}}$.  Therefore,
given some neat ordering $\mathbf{P}$, one can construct a bounded and
admissible mask set by finding one mask $\sigma(\bftau)$ for each
$\bftau\in\mathcal{T}_{\mathbf{P}}$ with the properties that
$\rho_{\mathbf{P}}(p_{\sigma(\bftau)})=p_\bftau$ and
$\dim(C_{\sigma(\bftau)})=\dim(C_\bftau)$. The combinatorial bijection
$R_{\mathbf{P}}$ in this case will just send $\sigma(\bftau)$ to $\bftau$.

We say that a mask set $\mathcal{E}$ is {\bf geometric} (with respect to
$\mathbf{P}$) if it can arise from such a construction.  One can check whether
a mask set $\mathcal{E}$ is geometric by checking if $\rho_{\mathbf{P}}$
restricts to a bijection between $\{p_\sigma\mid \sigma\in\mathcal{E}\}$ and
$\{p_\bftau\mid \bftau\in\mathcal{T}_{\mathbf{P}}\}$.

For all $\sigma$ and $\bftau$, $\pi_\w(C_\sigma)$ and
$\pi_{\mathbf{P}}(C_\bftau)$ will be all of some Schubert cell $C_x$, as $C_x$
is a single orbit of the group $B$, and the maps $\pi_\w$ and
$\pi_{\mathbf{P}}$ are $B$-equivariant.  However,
$\rho_{\mathbf{P}}(C_\sigma)$ can be some subvariety (which will be a linear
subspace under the appropriate coordinates) of $C_\bftau$.  Therefore, even if
$\dim(C_\sigma)=\dim(C_\bftau)$, it is not necessary for $\rho_{\mathbf{P}}$
to restrict to an isomorphism between $C_\sigma$ and $C_\bftau$.  We say that
a mask set $\mathcal{E}$ is {\bf strongly geometric} (with respect to
$\mathbf{P}$) if $\rho_{\mathbf{P}}$ restricts not only to a bijection between
$\{p_\sigma\mid \sigma\in\mathcal{E}\}$ and $\{p_\bftau\mid
\bftau\in\mathcal{T}_{\mathbf{P}}\}$ but also to a pointwise bijection between
$\bigcup_{\sigma\in\mathcal{E}} C_\sigma$ and $Z_{\mathbf{P}}$.  Note this
bijection will {\em not} be a topological homeomorphism as $Z_{\mathbf{P}}$ is
compact whereas $\bigcup_{\sigma\in\mathcal{E}} C_\sigma$ is not.

Note that it is not automatically desirable for a mask set to be geometric
with respect to some neat ordering $\mathbf{P}$.  In particular, while the
cograssmannian formula has been generalized to covexillary permutations by
Lascoux, there seem to be covexillary permutations $w$ for which $X_w$ admits
no small resolution.  This seems to imply that it may be easier to generalize
a mask set that is not geometric to generalize to the
covexillary case.

\subsection{A geometric mask set}

The aim of this subsection is to give an algorithm that, given an ordering
$\mathbf{P}$ of the peaks, produces a mask set $\mathcal{E}_{\mathbf{P}}$
which is geometric with respect to $\mathbf{P}$.  If our ordering $\mathbf{P}$
is neat, then this mask set will automatically bounded and admissible.
We do not know if our mask set is strongly geometric.


We aim to construct for each $\bftau\in\mathcal{T}_{\mathbf{P}}$ a mask
$\sigma(\bftau)$ (which we will describe as $+$'s and $-$'s drawn on the
heap) such that $\rho_{\mathbf{P}}(p_\sigma)=p_\bftau$, and
$\dim(C_\sigma)=\dim(C_\bftau)$.  We construct the mask $\sigma(\bftau)$
rectangle by rectangle, starting from $R_1$.  Note that if $\bftau_1\neq\bftau_2$,
the masks $\s_1=\sigma(\bftau_1)$ and $\s_2=\sigma(\bftau_2)$ constructed from
them will necessarily be different, since $p_{\sigma_1}$ and $p_{\sigma_2}$
will be different points (as they map under $\rho_{\mathbf{P}}$ to different
points of $Z_{\mathbf{P}}$).

Fix for the remainder of this section some
$\bftau\in\mathcal{T}_{\mathbf{P}}$, and let the $T$-fixed point $p_\bftau$
associated to $\bftau$ be $p_\bftau=[W_1,\ldots,W_{p-1}, F_1,\ldots,F_{n-1}]$.
Recall that associated to $\bftau$ (or to $p_\bftau$) are sets $A(j)$, $T(j)$,
and $D(j)$ for all $j$, $1\leq j\leq p$.  The following lemma is key to our
construction.

\begin{lemma}\label{l:plusind}
Fix some $j$ with $1\leq j\leq p$.  Suppose that we have a string diagram for
some mask on the heap of $\w$ such that for all $k$, $1\leq k\leq
j-1$, the strings exiting $R_k$ to the SW are the ones with labels in $T(k)$,
and the strings exiting $R_k$ to the SE are the ones with labels in $D(k)$.
Then
\begin{enumerate}
\item the strings entering $R_j$ from the NW and NE are the ones with labels
  in $A(j)$.
\item for all $k\leq j-1$, the strings to the left of $\w_{b_k}$ immediately
  below $\w_{b_k}$ are the ones labelled $m$ for $e_m\in W_k$.
\end{enumerate}
\end{lemma}

\begin{proof}
Our proof is by induction on $j$.  For part (2), the inductive hypothesis
allows us to assume that the statement is true for $k\leq j-2$.  In
particular, it is true for $k=\before^{\mathbf{P}}(j-1)$ if
$\before^{\mathbf{P}}(j-1)$ is defined.  Therefore in this case, the strings
labelled $m$ for $e_m\in W_{\before^{\mathbf{P}}(j-1)}$ are to the left of
$\w_{b_{\before^{\mathbf{P}}(j-1)}}$.  Since
$\w_{b_{\before^{\mathbf{P}}(j-1)}}$ and $\w_{b_{j-1}}$ are on the same
diagonal, these same strings must pass to the left of $\w_{b_{j-1}}$.  If
$\before^{\mathbf{P}}(j-1)$ is undefined, then the same argument shows that
the strings labelled $m$ for $m\leq\ldim(j-1)$ pass to the left of
$\w_{b_{j-1}}$; these are the indices $m$ for $e_m\in E_{\ldim(j-1)}$, .

By hypothesis, the strings with labels in $T({\before^{\mathbf{P}}(j-1)})$
pass to the left of $\w_{b_{j-1}}$.  Since $$\{m\mid e_m\in W_{j-1}\}=\{m\mid
e_m\in W_{\before^{\mathbf{P}}(j-1)}\}\cup T(j-1)$$ (with the substitution of
$E_{\ldim(j-1)}$ for $W_{\before^{\mathbf{P}}(j-1)}$ if necessary), this
proves part (2).

The strings entering $R_j$ from the NW and NE are the ones which pass to the
left of $\w_{b_{\after^{\mathbf{P}}(j)}}$ and to the right of
$\w_{b_{\before^{\mathbf{P}}(j)}}$.  By part (2), these are the strings
labelled $m$ for $e_m\in W_{\after^{\mathbf{P}}(j)}$ but $e_m\not\in
W_{\before^{\mathbf{P}}(j)}$.  This is precisely the definition of
$A(j)$.  (If $\before^{\mathbf{P}}(j)$ or $\after^{\mathbf{P}}(j)$ is
undefined, then the necessary equivalent statements are trivial.)  This
proves part (1).
\end{proof}

Based on this lemma, we construct $e(\sigma(\bftau))$ which then uniquely
defines $\sigma(\bftau)$.  We think of each rectangle $R_j$ as a sorting
process.  The inputs to the process are strings with labels in $A(j)$ coming
in through the NW and NE edges.  The process takes these strings and uses $+$'s
and $-$'s at the heap points to rearrange them so that strings with labels
in $T(j)$ exit through the SW edge and strings with labels in $D(j)$ exit
through the SE edge.  Additionally, we want the number of $+$'s in each
rectangle $R_j$ to equal $\left|\tau^{(j)}\right|$ and the number of $+$'s
below the grassmannian part to equal $\ell(x_\tau)$; this will ensure that
$\dim(C_\sigma)=\dim(C_\bftau)$.

If the set $A(j)$ always entered $R_j$ in sorted order with the lowest
labelled string on the left, then we would be able to simply put $\tau^{(j)}$
at the bottom of $R_j$, fill the area covered by $\tau^{(j)}$ with $+$'s, and
fill the rest of $R_j$ with $-$'s.  However, since $A(j)$ does not always
enter in sorted order, we need a more complicated construction.

Our construction fills $R_j$ with $+$'s and $-$'s one NE-SW diagonal at a
time, starting with the diagonal at the NW edge.  We wish to put $+$'s and
$-$'s on that diagonal so that the string $k$ leaving $R_j$ at the SW end of
the diagonal is in $T(j)$ and furthermore, that the number of $+$'s we put in
the diagonal equals $\#\{m\in D(j)\mid m<k\}$.

Let $C$ be the set of labels of strings entering the diagonal we are working
on.  We first choose the string $k\in C\cap T(j)$ which will exit the diagonal
to the SW.  To make the number of $+$'s come out correctly, we want $k$ to
have the property that the number of elements of $\{m\in C \mid m<k\}$ (which
we call fake inversions) equals the number of elements of $\{m \in D(j) \mid
m<k\}$ (which we call real inversions).

We now show that such an element $k\in C\cap T(j)=C\setminus D(j)$ exists.
Place the elements of $A(j)$ in order, and for each $k\in A(j)$ let
$$g(k)=\#\{m\in C \mid m<k\} - \#\{m\in D(j) \mid m<k\}.$$ Observe that $g$
stays the same, goes up one, or goes down one with each successive $k\in
A(j)$.  Furthermore, $g$ is $0$ for the smallest element of $A(j)$, and since
$C$ has exactly one more element than $D(j)$, we have that $g$ is $1$ for the
largest element of $A(j)$.  Now, let $k$ be the first element of $A(j)$ for
which $g(k)=0$ but $g(k^\prime)=1$, where $k^\prime$ is the next largest
element from $k$ in $A(j)$.  By the definition of $g$, we have both that $k\in
C\setminus D(j)$ and that $k$ has an equal number of real and fake inversions.

Next, we fill the diagonal with $+$'s and $-$'s.  To be precise, we put a $+$
at each entry of the diagonal where the string entering from the NW (i.e. the
long side) is labelled $m$ with $m<k$.  Moreover, if the string entering the
diagonal from the NE is labeled $m$ with $m<k$, then we also put a $+$ at the
entry where $k$ enters from the NW.

Observe that the number of $+$'s in the diagonal equals $\#\{m \in C \mid
m<k\}$, which by our choice of $k$ is $\#\{m \in D(j) \mid m<k\}$.  The reader
can check that our construction will result in the string labelled $k$ exiting
the diagonal to the SW.  In particular, if the string entering the diagonal
from the NE is labelled $m$ for some $m<k$, the string meeting $k$ from the NE
where $k$ enters from the NW will be labelled $m^\prime$ for some
$m^\prime<k$, as any string with a label greater than $k$ entering above $k$
is immediately met with a $-$ and sent out of the diagonal.

We repeat this entire procedure for each diagonal of $R_j$.  The number of
$+$'s in $R_j$ will then be 
$$\sum_{k\in T(j)} \#\{m \in D(j) \mid m<k\} = \left|\tau^{(j)}\right|.$$ 
(Note however that the $m$-th diagonal of $R_j$ may not have $\tau^{(j)}_m$
$+$'s.)  Iterating this procedure for all rectangles $R_j$ produces a mask on
the grassmannian permutation $v=ww_0^J$ where the number of $+$'s equals $\sum_j
\left| \tau^{(j)}\right|$.  Observe that we could have used a similar
construction using NW-SE diagonals instead.  In general this may produce a
different mask.

Now we need to produce a mask for the bottom part of the heap representing
$w_0^J$ such that the number of $+$'s equals $\ell(x_\tau)$.  We use two
different procedures that are mirror images of each other for the two sides of
the bottom.

On the left side of the heap for the $w_0^J$ part, consider the NW-SE
diagonals one at a time, starting from the top.  For any given diagonal, $x$
specifies that a particular string with label $k$ must exit that diagonal to
the SE.  Define the set 
\[ Inv(k)=\{m\mid m>k, x^{-1}(m)<x^{-1}(k)\}. \]
(Since $u_\bftau(1)<\cdots<u_\bftau(z)$, it does not matter whether we consider
$x_\tau$ or $x=u_\bftau x_\tau$.)  We also want the number of $+$'s on the
diagonal to equal the size of $Inv(k)$.  As above, we place a $+$ in the
diagonal wherever a string labelled $m$ with $m<k$ enters from the NE (the long
side).  Furthermore, we place a $+$ where the string labelled $k$ enters if the
the string entering the diagonal at the NW edge is labelled $m$ for some $m<k$.
The reader can check that the string labelled $k$ exits the diagonal to the SE
given our construction.  Since a string enters our diagonal only if
$x^{-1}(m)<x^{-1}(k)$, the number of $+$'s we place in the diagonal is equal to
the size of $Inv(k)$.

On the right side of the heap for the $w_0^J$ part, we consider the NE-SW
diagonals one at a time.  As before, let $k$ be the label for the string which
must exit the diagonal to the SW according to $x$.  We let
\[ Inv^{\prime}(k)=\{m\mid m<k, x^{-1}(m)>x^{-1}(k)\}. \]
Place a $+$ wherever a string labelled $m$ with $m<k$ enters from the NW and a
$+$ where the string $k$ enters if the string entering the diagonal at the NE
edge is labelled $m$ for some $m<k$.  The reader can check that the string
labelled $k$ exits the diagonal to the SW, and the number of $+$'s we place in
the diagonal is the size of $Inv^\prime(k)$.  This entire procedure places
$\ell(x_\tau)$ $+$'s in the $w_0^J$ part, as desired.

\begin{example}
Let $w$, $\mathbf{P}$, and $\tau^{(1)},\ldots,\tau^{(6)}$ be as in
Example~\ref{e:zel_large_cell}, and let $x_\tau$ be the permutation (in 1-line
notation) $$x_\tau=(1,2,4,3,5,6,7,11,12,9,8,10,15,13,16,14,18,17,19,20,21,22,23).$$

The reader can check that $\sigma$ is the mask in Figure~\ref{f:maskfortau}.
For example, we assign $+$'s and $-$'s on the highest NE-SW diagonal of $R_6$
using $C = \{7, 5, 6, 8, 9, 10\}$, $A(6) = \{5, 6, \ldots, 17\}$, and $D(6) =
\{5, 8, 12, 14, 16\}$, so $k = 6$. 

\begin{figure}[h]
\begin{center}
\begin{tabular}{c}
\xymatrix @-2.4pc @M=0pt @! {
{\hm_{P_3}} & & {\hs} & & {\hs} & & {\hm_{P_2}} & & {\hm_{P_6}} & & {\hs} & & {\hs} & & {\hs} & & {\hs} & & {\hs} & & {\hs} & \\
& {\hm} & & {\hs} & & {\hp} & & {\hm} & & {\hm} & & {\hs} & & {\hp_{P_1}} & & {\hs} & & {\hs} & & {\hm_{P_4}} & & {\hm_{P_5}} \\
{\hm} & & {\hm} & & {\hp} & & {\hm} & & {\hm} & & {\hm} & & {\hp} & & {\hp} & & {\hs} & & {\hm} & & {\hm} & \\
& {\hm} & & {\hm} & \ar@{-}[ul] & {\hp} & & {\hm} & & {\hm} & & {\hm} & & {\hp} & & {\hp} & & {\hm} & & {\hm} & & {\hm} \\
{\hm} & & {\hm} & \ar@{-}[urururur] & {\hm} & & {\hm} & & {\hm} & & {\hm} & & {\hm} & & {\hp} & & {\hm} & & {\hm} & & {\hm} & \\
& {\hm} & & {\hm} & & {\hp} & & {\hp} & & {\hm} & & {\hm} & & {\hp} & \ar@{-}[urururur]& {\hp} & & {\hm} & & {\hm} & & {\hm} \\
{\hm} & & {\hm} & & {\hm} & & {\hp} & & {\hp} & & {\hm} & & {\hp} & & {\hm} & \ar@{-}[ururururur] & {\hm} & & {\hm} & & {\hm} & \\
& {\hm} & & {\hm} & & {\hm} & & {\hp} & & {\hp} & & {\hp} & & {\hp} & & {\hm} & \ar@{-}[ululululul] & {\hm} & & {\hm} & & {\hm} \\
{\hm} & & {\hm} & & {\hm} & & {\hm} & & {\hp} & & {\hp} & & {\hp} & & {\hm} & & {\hm} & & {\hm} & & {\hm} & \\
& {\hm} & & {\hm} & & {\hm} & & {\hm} & & {\hp} & & {\hp} & & {\hp} & & {\hm} & & {\hm} & & {\hm} & & {\hm} \\
{\hm} & & {\hm} & & {\hm} & & {\hm} & & {\hp} & & {\hp} & & {\hp} & & {\hm} & & {\hm} & & {\hm} & & {\hm} & \\
& {\hm} & & {\hm} & & {\hm} & & {\hp} & & {\hp} & & {\hp} & & {\hp} & & {\hm} & & {\hm} & & {\hm} & & {\hm} \\
{\hm} & & {\hm} & & {\hm} & & {\hp} & & {\hp} & & {\hm} & \ar@{-}[ululululululululululul]\ar@{-}[urururururururururur] & {\hp} & & {\hm} & & {\hm} & & {\hm} & & {\hm} & \\
& {\hm} & & {\hm} & & {\hm} & & {\hp} & & {\hp} & & {\hs} & & {\hm} & & {\hm} & & {\hm} & & {\hm} & & {\hm} \\
{\hm} & & {\hm} & & {\hm} & & {\hm} & & {\hm} & & {\hs} & & {\hs} & & {\hp} & & {\hm} & & {\hm} & & {\hm} & \\
& {\hm} & & {\hm} & & {\hm} & & {\hm} & & {\hs} & & {\hs} & & {\hs} & & {\hm} & & {\hm} & & {\hm} & & {\hm} \\
{\hm} & & {\hm} & & {\hm} & & {\hm} & & {\hs} & & {\hs} & & {\hs} & & {\hs} & & {\hp} & & {\hm} & & {\hm} & \\
& {\hm} & & {\hm} & & {\hm} & & {\hs} & & {\hs} & & {\hs} & & {\hs} & & {\hs} & & {\hm} & & {\hm} & & {\hm} \\
{\hm} & & {\hm} & & {\hm} & & {\hs} & & {\hs} & & {\hs} & & {\hs} & & {\hs} & & {\hs} & & {\hm} & & {\hm} & \\
& {\hm} & & {\hm} & & {\hs} & & {\hs} & & {\hs} & & {\hs} & & {\hs} & & {\hs} & & {\hs} & & {\hm} & & {\hm} \\
{\hm} & & {\hp} & & {\hs} & & {\hs} & & {\hs} & & {\hs} & & {\hs} & & {\hs} & & {\hs} & & {\hs} & & {\hm} & \\
& {\hm} & & {\hs} & & {\hs} & & {\hs} & & {\hs} & & {\hs} & & {\hs} & & {\hs} & & {\hs} & & {\hs} & & {\hm} \\
{\hm} & & {\hs} & & {\hs} & & {\hs} & & {\hs} & & {\hs} & & {\hs} & & {\hs} & & {\hs} & & {\hs} & & {\hs} & \\
}
\end{tabular}
\end{center}
\caption{Mask for specified $\bftau$}\label{f:maskfortau}
\end{figure}
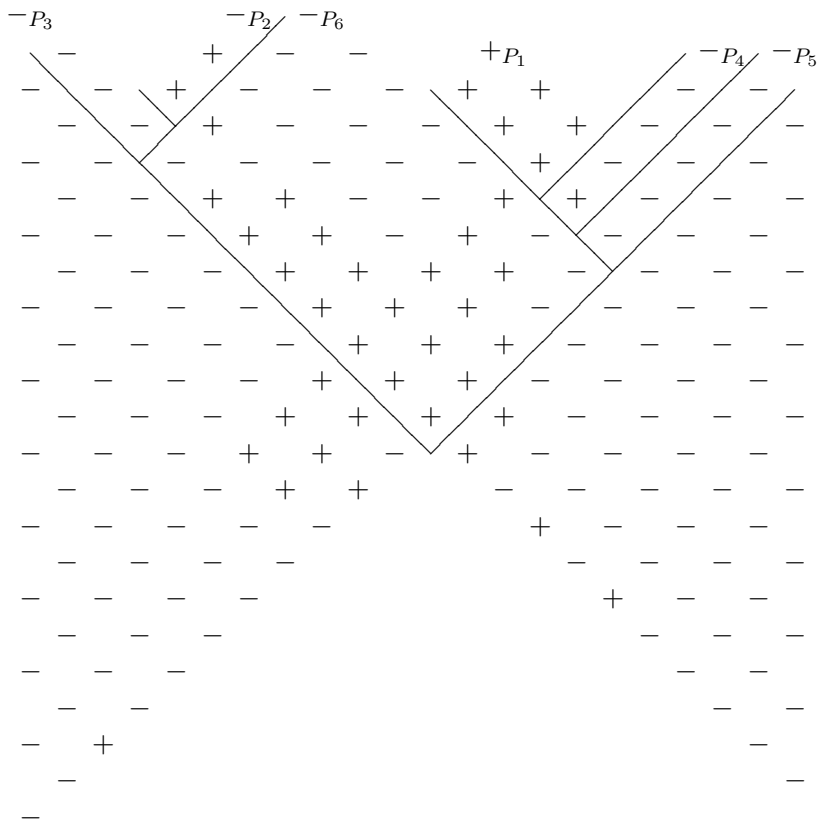
\end{example}

\begin{lemma}
The number of $+$'s in $e(\sigma(\bftau))$ equals $\dim(C_\bftau)$.
\end{lemma}

\begin{proof}
The number of $+$'s the construction of $\sigma(\bftau)$ places in $R_j$
is $$\sum_{k\in T(j)} \#\{m \in D(j) \mid k>m\} = \left|\tau^{(j)}\right|.$$

The permutation $x_\tau$ only has inversions between elements both before (or
including) $z$ and inversions between elements after $z$.  Each of these
inversions is counted exactly once either in $Inv(k)$ or in $Inv^\prime(k)$
(depending on whether $u_\bftau(k)\leq z$ or not) for some $k$.  Therefore,
the number of $+$'s placed below the cograssmannian part is $\ell(x_\tau)$.
Hence the number of $+$'s in $\sigma(\bftau)$ is
$$\sum_{j=1}^p \left|\tau^{(j)}\right| + \ell(x_\tau),$$ which is
$\dim(C_\bftau)$ by Proposition~\ref{p:zeldim}.
\end{proof}

\begin{lemma}
The map $\rho_{\mathbf{P}}$ sends $p_{\sigma(\bftau)}$ to $p_\bftau$.
\end{lemma}
\begin{proof}
Let $p_{\sigma(\bftau)}=[V_1,\ldots,V_\ell]$ and
$p_\bftau=[W_1,\ldots,W_{p-1},F_1,\ldots,F_{n-1}]$.  Since the construction of
$\sigma(\bftau)$ is such that the strings with labels in $T(k)$ always exit to
the left of $b_k$, by Lemma~\ref{l:plusind}, the strings exiting to the left
of $\w_{b_j}$ immediately below $\w_{b_j}$ are the ones labelled $m$ for
$e_m\in W_j$.  Applying the lemma to all $j$, $1\leq j\leq p-1$, shows that
$V_{b_j}=W_j$ for all $j$, $1\leq j\leq p-1$.

The construction of $\sigma(\bftau)$ explicitly ends up with
$\w^{\sigma(\bftau)}=x$, where $x=u_\bftau x_\tau$.  This implies
$p_x=\pi_{\mathbf{P}}(p_\bftau)$.  Therefore, $V_{\last(j)}=F_j$ for all $j$.

The map $\rho_{\mathbf{P}}$ sends $p_{\sigma(\bftau)}$ to the point
$[V_{b_1},\ldots,V_{b_{p-1}},V_{\last(1)},\ldots,V_{\last(n-1)}]$, which as
argued above is
precisely $[W_1,\ldots,W_{p-1}, F_1,\ldots,F_{n-1}]$.  This point is
$p_\bftau$.
\end{proof}

\begin{theorem}
\label{t:main2}
If $w$ is cograssmannian and $\mathbf{P}$ is a neat order of the peaks of
$w$, then
$$C^\prime_w=q^{-\frac{1}{2}\ell(w)}\sum_{\sigma\in\mathcal{E}_{\mathbf{P}}}
q^{d(\sigma)}, \text{\ \  where \ \ } \mathcal{E}_{\mathbf{P}}=\{\sigma(\bftau) \mid \bftau\in\mathcal{T}_{\mathbf{P}}\}.$$
\end{theorem}

\begin{proof}
By Theorem~\ref{t:zel},
$$C^\prime_w=q^{-\frac{1}{2}\ell(w)}\sum_{\bftau\in\mathcal{T}_{\mathbf{P}}}
q^{\dim(C_\tau)-\ell(x)},$$ where $x=u_\bftau x_\tau$.  Since we have a
bijection between $\mathcal{E}_{\mathbf{P}}$ and $\mathcal{T}_{\mathbf{P}}$
such that $\dim(C_{\sigma(\bftau)})=\dim(C_\bftau)$ and $x=\w^\sigma$, and
$d(\sigma)=\dim(C_\sigma)-\ell(\w^\sigma)$ by Propositions~\ref{p:nplus}
and~\ref{p:plusdim}, the theorem is proved.
\end{proof}

\subsection{A non-geometric mask set}

In this section we explain why the mask set $\mathcal{E}_\w$ constructed in
Section~\ref{s:cog} may in some cases not be geometric with respect to any
(neat) ordering $\mathbf{P}$.  In particular, we give an example of a
permutation $\w$ and two masks $\sigma_1,\sigma_2\in\mathcal{E}_\w$ such that
$\rho_{\mathbf{P}}(p_{\sigma_1})=\rho_{\mathbf{P}}(p_{\sigma_2})$.  As a
similar argument can be applied to any of the obvious variants of both
constructions, this shows that our two constructions of mask sets are in some
sense fundamentally different.

Consider the permutation $w$ of Figure~\ref{f:heap_segments}, and the
pair of labelled trees $t_1$ and $t_2$ in Figure~\ref{f:cexp_trees}.

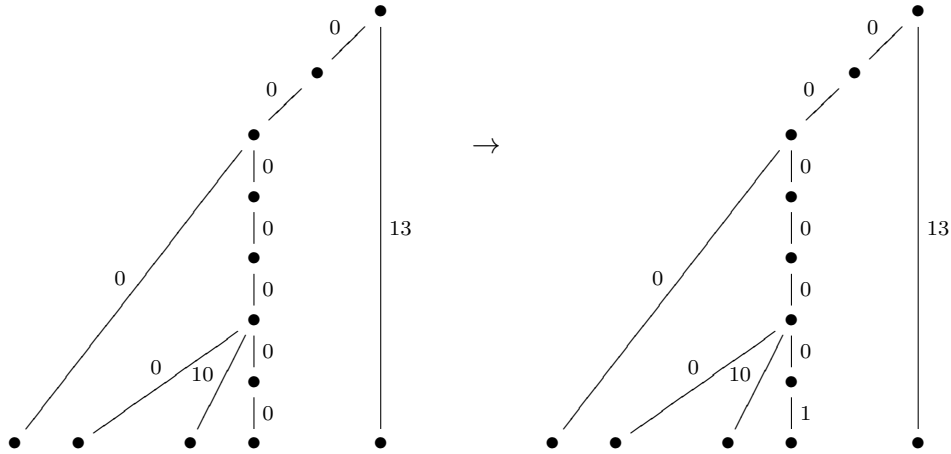
\begin{figure}[h]
\begin{center}
\begin{tabular}{ccc}
\xymatrix @-1pc {
& & & & & & \gn \ar@{-}[dl]_{0} \ar@{-}[ddddddd]^{13} & \\
& & & & & \gn \ar@{-}[dl]_{0} & & \\
& & & & \gn \ar@{-}[d]^{0} \ar@{-}[dddddllll]_{0} & & \\
& & & & \gn \ar@{-}[d]^{0} & & \\
& & & & \gn \ar@{-}[d]^{0} & & \\
& & & & \gn \ar@{-}[d]^{0} \ar@{-}[ddl]_{10} \ar@{-}[ddlll]_{0} & & \\
& & & & \gn \ar@{-}[d]^{0} & & \\
\gn & \gn & & \gn & \gn & & \gn \\
} & \parbox{0.2in}{\vspace{1.5in}$\rightarrow$} &
\xymatrix @-1pc {
& & & & & & \gn \ar@{-}[dl]_{0} \ar@{-}[ddddddd]^{13} & \\
& & & & & \gn \ar@{-}[dl]_{0} & & \\
& & & & \gn \ar@{-}[d]^{0} \ar@{-}[dddddllll]_{0} & & \\
& & & & \gn \ar@{-}[d]^{0} & & \\
& & & & \gn \ar@{-}[d]^{0} & & \\
& & & & \gn \ar@{-}[d]^{0} \ar@{-}[ddl]_{10} \ar@{-}[ddlll]_{0} & & \\
& & & & \gn \ar@{-}[d]^{1} & & \\
\gn & \gn & & \gn & \gn & & \gn \\
}
\end{tabular}
\end{center}
\caption{Edge-labelled trees for }\label{f:cexp_trees}
\end{figure}

Let $\sigma_1=\sigma(t_1)$ and $\sigma_2=\sigma(t_2)$ as constructed in
Section~\ref{s:cog}.  These are both masks in $\mathcal{E}_\w$ for the
permutation $x=x(t_1)=x(t_2)$ whose constant mask is shown in
Figure~\ref{f:heap_segments}.  These masks only differ in two places, both in
the fourth segment from the left.  One place is in the lowest diagonal in
region 1, and the other is at the end of the right edge.

Let $p_{\s_1}=[V_1,\ldots,V_{\ell(w)}]$ and
$p_{\s_2}=[V^\prime_1,\ldots,V^\prime_{\ell(w)}]$.  Let $\mathbf{P}$ be any
order of the peaks, and let the rectangles $R_j$ and the heap points $b_j$ be
defined as above.  Note that both places where $\sigma_1$ and $\sigma_2$
differ are below every peak.  No matter what $\mathbf{P}$ is, the first
difference is in the rectangle $R_6$, and the second is below the
grassmannian part.  Therefore, $V_{b_j}=V^\prime_{b_j}$ for all $j$, $1\leq
j\leq 5$.  Furthermore, since $\w^{\sigma_1}=\w^{\sigma_2}=x$,
$V_{\last(j)}=V^\prime_{\last(j)}$ for all $j$.  Therefore,
$\rho_{\mathbf{P}}(p_{\sigma_1})=\rho_{\mathbf{P}}(p_{\sigma_2})$.

Since $\rho_{\mathbf{P}}$ is not a bijection, $\mathcal{E}_\w$ is not a
geometric mask set.  This failure to be geometric applies to any ordering
$\mathbf{P}$ of the masks.

\section*{Acknowledgments}

The first author wishes to thank F.~Brenti for helpful conversations related
to this work.  The second author wishes to thank A.~Yong for discussions
leading up to this work and A.~Knutson for properly
explaining the Bott--Samelson resolution to him many years ago, including at
BADMath Day in Fall 2003.


\bibliographystyle{alpha}
\bibliography{our}

\end{document}